\documentclass[11pt]{amsart}
\usepackage[margin=3cm]{geometry}
\usepackage{amsfonts, enumerate}
\usepackage{amssymb}
\usepackage{amscd}
\usepackage[final]{graphicx, graphics,stmaryrd}
\input xy
\xyoption{all}
\usepackage[english]{babel}
\usepackage[all]{xy}
\usepackage[breaklinks,bookmarksopen,bookmarksnumbered]{hyperref}

\newtheorem{theorem}{Theorem}[section]
\newtheorem{fact}[theorem]{Fact}
\newtheorem{proposition}[theorem]{Proposition}
\newtheorem{corollary}[theorem]{Corollary}
\newtheorem{lemma}[theorem]{Lemma}
\theoremstyle{definition}
\newtheorem{definition}[theorem]{Definition}
\newtheorem{remark}[theorem]{Remark}
\newtheorem{example}[theorem]{Example}

\newtheorem{conjecture/question}[theorem]{Conjecture/Question}

\newtheorem{remark/definition}[theorem]{Remark/Definition}
\newtheorem{terminology/notation}[theorem]{Terminology/Notation}
\newtheorem{assumption}[theorem]{Assumption}
\newcommand{\marginlabel}[1]%
  {\mbox{}\marginpar{\raggedleft\hspace{0pt}\bfseries\sf#1}}

\DeclareMathOperator{\defo}{Def}
\DeclareMathOperator{\aut}{Aut}
\DeclareMathOperator{\diag}{Diag}

\DeclareMathOperator{\ord}{ord}
\DeclareMathOperator{\im}{im}
\DeclareMathOperator{\id}{id}

\def\pmmu{{\pmb \mu}}

\newcommand{\ram}{\operatorname{ram}}
\newcommand{\etale}{\text{{\'e}t}}
\newcommand{\reg}{\operatorname{reg}}

\newcommand{\Sing}{\operatorname{Sing}}
\newcommand{\looop}{\operatorname{loop}}
\newcommand{\stable}{\operatorname{stable}}
\newcommand{\length}{\operatorname{length}}
\newcommand{\GL}{\operatorname{GL}}
\newcommand{\modulo}{\operatorname{mod}}
\newcommand{\val}{\operatorname{val}}
\newcommand{\lcm}{\operatorname{lcm}}
\newcommand{\Spec}{\operatorname{Spec}}
\newcommand{\age}{\operatorname{{age}}}
\newcommand{\Pic}{\operatorname{Pic}}
\newcommand{\nor}{\operatorname{nor}}

\def\ol{\overline}
\def\wt{\widetilde}
\def\ul{\underline}
\def\mbb{\mathbb}
\def\al{\alpha}
\def\ssM{\mathsf{M}}
\def\ssR{\mathsf{R}}
\def\ssB{\mathsf{B}}
\def\ssf{\mathsf{f}}
\def\ssC{\mathsf{C}}
\def\ssL{\mathsf{L}}

\def\ssphi{\mathsf{\phi}}
\def\cO{\mathcal{O}}

\newcommand{\sta}{\mathsf}

\newcommand{\iso}{\sta{s}}
\newcommand{\coaiso}{{s}}

\def\ZZ{{\mathbb Z}}
\def\GG{{\textbf G}}
\def\QQ{{\mathbb Q}}
\def\CC{{\mathbb C}}
\def\PP{{\mathbb P}}

\def\cM{\mathcal{M}}
\def\cR{\mathcal{R}}
\def\rr{\ol{\mathcal{R}}}

\def\mm{\ol{\mathcal{M}}}

\def\qr{\langle\mathrm{QR}\rangle}
\def\lqr{\mathrm{QR}}

\newcommand{\level}{(\leveli,\levelii,\leveliii)}
\newcommand{\leveli}{\ssC}
\newcommand{\levelii}{\ssL}
\newcommand{\leveliii}{\ssphi}
\newcommand{\rsbt}{Reid--Shepherd-Barron--Tai}

\newcounter{mylist}
{\begin{list}%
    {(\roman{mylist})}%
    {\usecounter{mylist}%
     \setlength{\rightmargin}{0pt}%
     \setlength{\leftmargin}{0pt}%
     \setlength{\itemindent}{0.5em}%
     \setlength{\itemsep}{0pt}%
     \setlength{\parsep}{0ex plus0.1ex}
     \setlength{\topsep}{0ex}}}%
{\end{list}}

\begin{document}
\title{Singularities of the moduli space of level curves}

\author[A.~Chiodo]{Alessandro Chiodo}
\address{Universit\'e Pierre et Marie Curie, Institut de math\'ematiques de Jussieu, 75252 Paris}
\email{{\tt
chiodoa@math.jussieu.fr}}

\author[G.~Farkas]{Gavril Farkas}

\address{Humboldt Universit\"at zu Berlin, Institut F\"ur Mathematik,
10099 Berlin}
\email{{\tt farkas@math.hu-berlin.de}}
\thanks{Research of the first author partially supported by the ANR grant Th\'eorieGW. \\
\indent Research of
the second author partially supported by the Sonderforschungsbereich 647 Raum-Zeit-Materie.}

\markboth{A.~CHIODO and G.~FARKAS}
{Singularities of $\rr_{g,\ell}$}

\begin{abstract}
We describe the singular locus of the 
compactification of the moduli space $\mathcal{R}_{g,\ell}$ of 
curves of genus $g$ paired with an $\ell$-torsion point 
in their Jacobian. Generalising previous work for $\ell\le 2$, we also
describe 
the sublocus 
of noncanonical singularities for any positive integer $\ell$. 
For $g\ge 4$ and $\ell=3,4, 6$, this allows us 
to provide a lifting result on pluricanonical forms 
playing an essential role in the computation of the Kodaira dimension of 
$\mathcal{R}_{g,\ell}$: for those values of $\ell$, every pluricanonical form on the smooth 
locus of the moduli space extends to a desingularisation 
of the compactified moduli space. 
\end{abstract}

\maketitle


The modular curve
$X_1(\ell):=\mathcal{H}/\Gamma_1(\ell)$ classifying
elliptic curves together with an $\ell$-torsion point 
in their Jacobian is
among the most studied objects in arithmetic geometry.
In a series of recent papers,
the birational geometry of its higher genus generalisations
and their variants (\emph{e.g.} theta characteristics)
has been systematically studied and proved 
to be, in many cases such as $\ell=2$,
better understandable than that of the
underlying moduli space of curves $\mathcal M_g$.
As an example, we refer to the complete
computation of the Kodaira dimension of all
components of the moduli of theta 
characteristics ($L^{\otimes 2}\cong \omega$), see
\cite{lu2007, Faspin, FV1, FV2}.

In this paper, for $g\ge 2$ and for all positive levels $\ell$,
we consider the moduli space 
$\mathcal R_{g,\ell}$ parametrizing 
level-$\ell$ curves, \emph{i.e.} triples $(C,L,\phi)$
where $C$ is a smooth curve equipped 
with a line bundle $L$ and a trivialisation
$\phi\colon L^{\otimes \ell}\xrightarrow{\sim }\cO$.
The Kodaira dimension of $\mathcal R_{g,\ell}$ is
defined as the Kodaira dimension of an arbitrary resolution of singularities
of a completion; therefore, as a
first step toward the birational classification of 
$\cR_{g,\ell}$,
we consider a natural 
compactification $\ol{\cR}_{g,\ell}$ and 
study the singular locus $\Sing(\ol{\cR}_{g,\ell})$.
More precisely one needs to determine
the sublocus 
 $\Sing_{\mathrm {nc}}(\ol{\cR}_{g,\ell})\subseteq \Sing(\ol{\cR}_{g,\ell})$
of noncanonical singularities. 

For $\ell=2$, this analysis has been carried out by 
the second author and Ludwig in \cite{FaLu}
using Cornalba's compactification in terms of quasistable curves 
\cite{Cornalba} of $\ol {\mathcal R}_{g, 2}$. Clearly, we can
leave out the case 
$\ell=1$, which coincides with Deligne and Mumford's functor 
of stable curves $\ol{\mathcal M}_{g}=\ol{\cR}_{g,1}$.
The passage to all
higher levels presents a new feature
from Abramovich and Vistoli's theory of stable maps to stacks: 
the points of the compactification 
cannot be interpreted  
in terms of $\ell$-torsion line bundles on a scheme-theoretic curve, but rather 
on a stack-theoretic curve.
Instead of the above triples 
$(C, L\in \Pic(C),\phi\colon L^{\otimes \ell}\xrightarrow{\sim}\cO)$, 
we simply consider their stack-theoretic analogues 
$$(\ssC,\ssL\in \Pic(\ssC),\phi\colon 
\ssL^{\otimes \ell}\xrightarrow{\sim}\cO)\in \ol {\mathcal R}_{g,\ell},$$
where $\ssC$ is a one-dimensional stack, 
whose nodes may have nontrivial stabilisers $\pmmu_r\subseteq \pmmu_\ell$,
and where $\ssL\to \ssC$ is a line bundle whose fibres are faithful 
representations,
see Definition \ref{defn:leveltwistedcurve}.
This yields a compactification which is represented by a 
smooth Deligne--Mumford stack. 

In analogy with the moduli space of stable curves $\ol{\mathcal M}_g$, 
the boundary locus 
$\ol {\mathcal R}_{g,\ell}\setminus  {\mathcal R}_{g,\ell}$ 
can be described 
in terms of the combinatorics of the standard
dual graph $\Gamma$ whose vertices correspond 
to the irreducible components of the curve
and whose edges correspond to the nodes of the curve. 
In \S\ref{subsubsect:multiplicity}, we 
revisit this well known description by 
emphasising the natural role of 
an extra \emph{multiplicity} datum enriching the graph.
Indeed, the stack-theoretic structure 
of the underlying curve $\ssC$ and 
the line bundle $\ssL\to \ssC$ are determined, locally at a node, by assigning to 
each oriented edge $e$ a character $\chi_e\in \mathrm{Hom}(\pmmu_r,\GG_m)=\ZZ/r\subseteq \ZZ/\ell$ of the 
stabiliser.
Hence, to each point of the boundary we attach  
a dual graph $\Gamma$ and a
$\ZZ/\ell$-valued 1-cochain 
$M\colon e\mapsto \chi_e$ in $C^1(\Gamma;\ZZ/\ell)$ 
which we refer to as the \emph{multiplicity
of the level curve}.
(Proposition \ref{pro:Misclosed} recalls that a multiplicity co-chain  
arises at the boundary if and only if it lies in the kernel 
of $\partial\colon C^1(\Gamma;\ZZ/\ell)\to C^0(\Gamma;\ZZ/\ell)$.)

In order to describe the singular locus of $\ol \cR_{g,\ell}$, 
we lift to the moduli of level curves
a result of Harris and Mumford \cite{HM}. 
Theorem 2 in \cite{HM} implies  
that, for $g\ge 4$, the local structure $\defo(C)/\aut(C)$
of $\ol{\mathcal M}_{g}$
is singular if and only if $C$ is equipped with 
an automorphism which is 
\emph{not} the product of 
``\emph{elliptic tail involutions}'' (ETI for short): 
$$\Sing(\ol{\mathcal M}_g)=
N_1:=\{C\mid \aut(C)\ni \al\  \text{\emph{not} a product of 
ETI}\}.
$$
By definition, an ETI operates nontrivially on the curve $C$ 
only at a genus-$1$ component $E$ which meets 
the rest of the curve
at exactly one node $n$; its restriction to 
the ``tail'' $(E,n)$ is the canonical involution. 
These automorphisms are the only 
nontrivial automorphisms of  
curves (and also of level curves) 
which do not yield singularities: their action 
on moduli is simply a quasireflection. 
An example of a
point of $N_1$ is given 
by choosing a tail $(E,n)$ with $\aut(E,n)\cong \pmmu_6$.
This type of curves fill-up a sublocus $T_1 \subset  N_1$, 
of codimension 2 within $\ol{\mathcal M}_g$, which plays
a remarkable role in this paper.
Indeed, the 
order-$6$ automorphism $\al$ spanning $\aut(E,n)$
and fixing $C\setminus E$
is clearly not a product of ETI and, most important, 
yields a noncanonical singularity. 
This can be checked by the \rsbt{} criterion: $\al$  
operates on the regular space $\defo(C)/
\langle \text{ETI}\rangle$  as $(\frac13,\frac13,0,\dots,0):=
\diag(\xi_3,\xi_3,1,\dots,1)$ and  
modding out $\al$ yields a 
noncanonical singularity, since    
the \emph{age} $\frac13+\frac13+0+\ldots+0$ 
of $\al$ is less than $1$ (see Defn. \ref{defn:junior}).
Harris and Mumford show that these special tailed curves
are the only possible curves carrying a \emph{junior} (\emph{i.e.} aged 
less than $1$) 
automorphism; this amounts to the following statement. 
$$\Sing_{\mathrm{nc}}(\ol{\mathcal M}_g)=
T_1:=\{C\mid C\supset E,\ C\cap \ol{C\setminus E}=\{n\}, \ 
\aut(E,n)\cong \pmb \mu_6\}.
$$

The generalisation of this statement to level-$\ell$ curves 
poses no problems on the interior:
the variety 
${\mathcal R}_{g,\ell}$ has only 
canonical singularities; furthermore, the 
singular locus is contained in the inverse image 
of the singular locus of $\mathcal M_g$, 
but may be smaller in general: an automorphism $\al$ 
of a smooth curve $C$ does not necessarily 
give rise to an automorphism of $(C, L)$ if $\al^*L\not \cong L$.

When we consider the boundary locus
$\ol {\mathcal R}_{g,\ell}\setminus  {\mathcal R}_{g,\ell}$ 
the analysis becomes 
subtle due to a new phenomenon: stack-theoretic curves $\ssC$
may be equipped with 
\emph{ghost automorphisms} $\sta a\in \aut_C(\ssC)$ which fix
all geometric points of $\sta C$ and yet operate nontrivially 
on the stack $\sta C$.
The group $\aut_C(\ssC)$ has been completely 
determined by Abramovich, Corti, and Vistoli \cite{ACV};
here, we describe the ghosts of level structures $(\ssC,\ssL,\phi)$
$$\ul{\aut}_C(\ssC,\ssL,\phi)=
\{\sta a \in \aut_C(\ssC) \mid \sta a^*\ssL\cong \ssL\}.$$
The loci $N_1$ and $T_1$ naturally lift to $N_{\ell}$ and $T_{\ell}$ within
$\ol{\cR}_{g,\ell}$. For the definition of 
$N_\ell$,  
 no modification 
is needed 
(we require that $\ul\aut\level$ contains at least 
an automorphism which is not the product of ghost automorphisms or 
of ETI, 
using the obvious generalisation of ETI to stack-theoretic curves, 
Definition \ref{defn:elltail}). 
The locus $T_\ell$ is defined as we did for $T_1$ 
by requiring the presence of an elliptic tail 
$(E,n)$ with $\aut(E,n)\cong \pmmu_6$, but also by imposing 
the extra condition 
that the line bundle be trivial on the genus-$1$ tail 
(see Definition \ref{defn:Tcurve}). For general 
values of $\ell$, we have proper inclusions 
$N_{\ell}\subsetneq \Sing(\ol{\mathcal R}_{g,\ell})$ and $T_{\ell}\subsetneq \Sing_{\mathrm nc}(\ol{\mathcal R}_{g,\ell})$.
In order to obtain 
$\Sing(\ol{\mathcal R}_{g,\ell})$ one needs to include also
the entire locus of level curves with a nontrivial ghosts (\emph{haunted} level
curves) 
$$H_\ell=\{(\ssC,\ssL,\phi)\mid \ul\aut_C(\ssC,\ssL,\phi)\neq 1\}.$$
Similarly, in order to obtain 
$\Sing_{\mathrm{nc}}(\ol{\mathcal R}_{g,\ell})$ one needs to take 
the union of $T_\ell$ and of  
the locus of level curves haunted by a \emph{junior ghost} 
$$J_\ell=\{(\ssC,\ssL,\phi)\mid \ul\aut_C(\ssC,\ssL,\phi)\ni \sta a, \ \age(\sta a)<1\},$$
where, as above, the age refers to the action 
on the regular space $\defo(C)/\langle \text{ETI}\rangle$.
This locus turns out to be entirely contained in the inverse 
image  of the locus of curves with 
at least three nonseparating nodes, see Remark \ref{rem:codim}. 
In this way, $J_\ell$ has codimension at least three 
within $\ol \cR_{g,\ell}$ and is a closed subvariety, reducible in 
general, 
lying in 
the inverse image of the boundary divisor $\delta_0^{\stable}$, closure of the locus 
of irreducible one-nodal curves. 
We deduce that  
$T_\ell$ is the only irreducible component of 
$\Sing_\mathrm{nc}(\ol \cR_{g,\ell})$ having 
codimension $2$ within $\ol \cR_{g,\ell}$.

Summarising the above discussion and 
taking advantage of the study of $H_\ell$ and $J_\ell$ 
carried out in Theorems~\ref{thm:smooth_compositel},
\ref{thm:no_junior_ghosts} and \ref{thm:noncanonical},
we provide the 
desired extension of pluricanonical forms 
for $\ell=3,4,6$.  

\noindent \textbf{Theorem.}
\emph{Let $g\ge 4$.
We have} 
$$\Sing(\ol \cR_{g,\ell})
=N_\ell \cup H_\ell\qquad 
\text{\emph{and}}
\qquad\Sing_{\mathrm {nc}}(\ol \cR_{g,\ell})
=T_{\ell} \cup J_{\ell}.$$
\emph{Furthermore,
the locus $J_{\ell}$ is empty 
if and only if $5\neq\ell\le 6$; therefore, for $g\ge 4$ and $5\neq\ell\le 6$, we have
\begin{equation}\label{eq:extensionofpcf}
\Gamma \bigl((\rr_{g,\ell})^{\mathrm{reg}},K_{\rr_{g,\ell}}^{\otimes q}\bigr) \cong
\Gamma\bigl(\widehat{\cR}_{g,\ell},
K_{\widehat{\cR}_{g,\ell}}^{\otimes q} \bigr)
\end{equation}for any desingularisation
$\widehat{\cR}_{g,\ell}\rightarrow\rr_{g,\ell}$
and  for all
integers $q\geq 0$.
 }

The case $\ell=1$ is proven by Harris and Mumford
in \cite{HM}. The case $\ell=2$ is proven by 
the second author in collaboration 
with Ludwig \cite{FaLu} (following work of 
Ludwig, \cite{lu2007}). The above formulation 
presents the 
isomorphism \eqref{eq:extensionofpcf} as a consequence of   
$J_\ell=\varnothing$ (and Harris and Mumford's work on the locus $T_1$). 
However, the question of whether \eqref{eq:extensionofpcf} holds 
in the remaining cases (for $\ell=5$ or $\ell>6$) 
remains open. 
In \S\ref{subsect:compositel}, we provide a complete computation 
of the group $\ul\aut_C(\ssC,\ssL,\phi)$ which 
is interesting in its own right. 
This shows in particular that 
the existence of a point in $J_\ell$ over 
a stable curve $C$ 
is a combinatorial 
condition depending only on the dual graph of $C$ and on $\ell$.
The computation of $\ul\aut_C(\ssC,\ssL,\phi)$ will  
certainly allow to further detail the geometry of $J_\ell$ 
(\emph{e.g} the irreducible components) and of $\ol{\cR}_{g,\ell}$.
We show for instance 
a simple combinatorial device (ghost camera) detecting the presence of ghosts and counting their number.  

{Write $\ell$ as $\prod_{p\mid \ell} p^{e_p}$,
where $p$ denotes a prime divisor of $\ell$ and $e_p$ the 
$p$-adic valuation of $\ell$. Fix a level curve 
$(\ssC,\ssL,\phi)$, its dual graph $\Gamma$ and the multiplicity 
$M\colon e\mapsto \chi_e$. }
{Consider the sequence of subgraphs} 
\begin{equation}\label{eq:subgraphsintro}
\varnothing \subseteq \Delta_p^{e_p}
\subseteq \ldots 
\subseteq\Delta_p^k:=\{e\mid \chi_e\in (p^k) \text{\ in $\ZZ/(p^{e_p})$}\} 
\subseteq \ldots \subseteq \Delta_p^{1}\subseteq  \Delta_p^{0}=\Gamma,
\end{equation}
where $\chi_e\in \ZZ/\ell$ is regarded as an element of $\ZZ/(p^{e_p}).$
{The contraction to points of the respective subsets of edges yields}
\begin{equation}\label{eq:graphscontractionintro}
\Gamma\to \Gamma_p^{e_p}
 \rightarrow\ldots \rightarrow
\Gamma_p^k  \rightarrow\ldots \rightarrow
\Gamma_p^{1}\rightarrow  \bullet.
\end{equation}
Then all the ghost automorphisms are trivial, \emph{i.e.} 
$\ul\aut_C(\ssC,\ssL,\phi)=1$, if and only if, 
$\Gamma_p^{e_p}$  
are bouquet (connected graphs with a single vertex), all $p$. 
Lemma \ref{tsohg} provides an explicit description of the group structure 
of $\ul\aut_C(\ssC,\ssL,\phi)$. In particular, we get the number of ghosts.

\noindent\textbf{Corollary.} \emph{We have 
$\#\ul\aut_C(\ssC,\ssL,\phi)=\frac{1}\ell\prod_{p\mid \ell}  
p^{V_p},$
where $V_p$ is the total number of vertices appearing in the graphs $\Gamma_p^j$ for $1\le j\le e_p$.}  

Note that, if $\Gamma_p^j$ is a bouquet for all $p$ and $j$, then 
$\#\ul\aut_C\level=\frac{1}\ell\prod_{p\mid \ell}  
p^{e_p}=1.$ See Example \ref{exa:ghostgroup} for a simple demonstration. 
In \S\ref{rem:CCC} the above formula is used to match 
Caporaso, Casagrande, and Cornalba's computation \cite{CCC} of 
the length of the fibre of the moduli of level curves over 
the moduli of stable curves. 

The above description leads to the claim
that junior ghosts (hence 
noncanonical singularities of the form
$\defo/\aut_C\level$) can be completely ruled out 
for $5\neq \ell\le 6$ and 
are relatively rare in general: 
their appearance is due to the presence of age-delay edges 
which we describe in the proof of the 
No-Ghost Lemma \ref{thm:no_junior_ghosts}.

The computation of the
Kodaira dimension of $\ol{\cR}_{g,\ell}$ for $\ell\le 6$ and $\ell\neq 5$ can 
be carried out without further study of resolutions of 
noncanonical singularities; for instance,
in \cite{CEFS}, in collaboration with 
Eisenbud and Schreyer, we show the following statement.

\noindent\textbf{Theorem (\cite[Thm.~0.2]{CEFS}).} \emph{
${\cR}_{g,3}$
is a variety of general type for
$g\ge 12$. 
Furthermore, the Kodaira dimension
of ${\cR}_{11,3}$
is at least $19$.
}

\bigskip

%

\paragraph{\textbf{Structure of the paper.}}
In Section \ref{sect:level3} we introduce moduli of smooth
level curves, their compactification, the relevant combinatorics and the boundary locus of the
compactified moduli space.
In Section \ref{sect:sing} we study the local structure of the moduli space,
we develop the suitable machinery for the computation of the ghost automorphism group and we deduce the theorem stated above.

\paragraph{\textbf{Acknowledgements.}}
We are grateful to Roland Bacher for several illuminating conversations
that put on the right track the computation of the group of ghost automorphisms in the case where
$\ell$ is composite. We thank Dimitri Zvonkine for several useful comments and explanations. The first author wishes to thank the
Mathematics Department of the \linebreak Humboldt Universit\"at zu Berlin
where this work started. 
Finally, we are extremely grateful to the anonymous referee for his 
careful reading and for his precious comments.

\section{Level curves}\label{sect:level3}

We work over an algebraically closed field $k$ and we always denote 
by $\ell$ a positive integer prime to $\mathrm{char}(k)$.

\subsection{Preliminary conventions on coarse spaces and local pictures.}
The interplay between stacks and their coarse spaces is crucial in this paper.
Any stack $\sta X$ of Deligne--Mumford (DM) type
admits  an algebraic space $X$
and  a morphism  $\epsilon_{\sta X}\colon \sta X\to X$
universal with respect to morphisms from $\sta X$ to algebraic spaces,
\cite{KM}.
We regard this operation as a functor.
The \emph{coarsening} of any DM stack 
$\sta X$ is the algebraic space $X$ (also called coarse space).
The \emph{coarsening of a morphism} 
$\ssf\colon \sta X\to \sta Y$ between DM stacks 
is the corresponding morphism $f\colon X\to Y$
$$\sta X \rightsquigarrow X\qquad 
\text{and} \qquad \sta f \rightsquigarrow f \qquad 
\text{(coarsening).}
$$
We will use this notion both for 
curves, possibly stack-theoretic ones and equipped with level structures, and for their moduli, which are represented by stacks. 
For clarity let us provide two simple examples. 
(1) Consider
the quotient DM stack $\ssC=[\PP_z^1/\pmb \mu_k]$ 
with $\zeta\in \pmmu_k$ acting as 
$z\mapsto \zeta z$ ($k\ge 2$); the coarsening $C$ 
of $\ssC$
is the (smooth) quotient scheme $C=\PP_z^1/\pmb \mu_k\cong \PP_{z^k}^1$. 
(2) The coarsening  of the proper, smooth, $3g-3$-dimensional DM stack
$\ol{\sta M}_g$ of stable curves of genus $g\ge 2$ is the
$3g-3$-dimensional projective scheme $\ol{\mathcal M}_g$.

When we refer to 
the local picture of $\sta X$ at the geometric point $\sta p$, we  mean
the strict Henselisation of $\sta X$ 
at $\sta p$. Hence, the local pictures of 
$\ol {\sta M}_{g}$ and of
$\ol {\mathcal M}_g$
at the points representing $C$ are 
the quotient stack 
$[\defo(C)/\aut(C)]$ and 
the quotient scheme $\defo(C)/\aut(C)$, respectively.

\subsection{Smooth level curves}\label{sect:smooth}
We set up $\sta R_{g,\ell}$, the 
space $\cR_{g,\ell}$, and the compactification problem.

\subsubsection{The moduli stack of level smooth curves}
The integers $g\ge 2$ and $\ell\ge 1$ denote the \emph{genus} and the \emph{level}. In this way, we do not consider smooth curves with infinite automorphism groups. 
We further assume that the level is prime to the characteristic 
of the base field.
\begin{definition}
The stack $\ssR_{g,\ell}$ is the category of
level-$\ell$ curves $(C,L,\phi)$ where
$C$ is a smooth genus-$g$ curve (over a base scheme $B$),
$L$ is a line bundle on $C$,
$\phi$ is an isomorphism \linebreak $\phi\colon L^{\otimes \ell}\to \cO_C$.
We additionally require that the order of the isomorphism class of $L$ in $\Pic(C)$ is exactly $\ell$.
A morphism from a family $(C\to B,L,\phi)$ to
a family $(C'\to B',L',\phi')$ is given by a
pair $(s,\rho)$ where  $s\colon (C'/B')\to (C''/B'')$ 
is a morphism of curves 
and $\rho$ is an isomorphism of line bundles $s^*L''\to L'$ satisfying
$\phi'\circ\rho^{\otimes \ell}=s^*\phi''.$
\end{definition}
The category $\ssR_{g,\ell}$ is a  DM stack.
Its points have finite stabilisers and we have a coarsening $\cR_{g,\ell}$ and a morphism
$${\ssR_{g,\ell}}\to \cR_{g,\ell}.$$
The forgetful functor
$\ssf\colon \ssR_{g,\ell}\to \ssM_g$ to the category of smooth genus-$g$ curves
is an \'etale,
connected cover,
and indeed a finite morphism of stacks.
Finiteness can be regarded as a consequence
of the fact that every fibre (pullback of $\mathsf{f}$ via a geometric point)
consists of $\Phi_{2g}(\ell)$ geometric points, with 
$$\Phi_n(\ell) = \ell^n\prod_{p\mid \ell} \left(1-\frac{1}{p^n}\right) \qquad 
\left(\Phi_{n}(\ell)={\ell^{n}-1} \ \text{if $\ell$ is prime}\right).$$
Each of
such points of the fibre is isomorphic to the
stack $\ssB\pmb\mu_\ell=[\Spec\CC/ \pmb\mu_\ell]$. This happens
because each point has \emph{quasitrivial} automorphisms
acting on $C$ as the identity (\emph{i.e.}, $s$ equals $\id_C$),
and scaling the fibres
of $L$ by multiplication by  $\zeta\in \pmmu_\ell$.
Since $\ssB\pmb \mu_\ell$ has degree $1/\ell$ over $\Spec \CC$, we get
$$\deg \left(\sta f\colon \sta R_{g,\ell}\to \sta M_g\right)=\frac{\Phi_{2g}(\ell)}\ell\qquad
\left(\ =\frac{\ell^{2g}-1}\ell \ \text{if $\ell$ is prime}\right).$$
When we pass to the coarsening
$f\colon \cR_{g,\ell}\to \cM_g$
the automorphisms are forgotten. The morphism $f$
is still a finite connected cover, but it may well be ramified.


The stack $\ssR_{g,\ell}$ is not compact.
If we allow triples $(C_{\text{st}},L,\phi),$
where $C_{\text{st}}$
is a stable genus-$g$ curve, and keep 
the rest of the definition unchanged
we obtain an \'etale cover of $\ol{\ssM}_g$. Properness fails: 
the cardinality of the fibre is not constant, as can be easily checked when $C$ is a one-nodal irreducible 
curve: $\#\Pic(C)[\ell]=\ell^{2g-1}$. 

\subsection{Twisted level curves.}\label{subsect:twisted}
The compactification becomes straightforward once we use the
analogue of nodal curves in the context of
DM stacks (for a scheme-theoretic translation
see Remark \ref{rem:backtoschemes})

\subsubsection{Twisted curves} \label{subsubsect:twisted}
We point out that a less restrictive definition of twisted curve
occurs in the literature, where
no stability condition on $C$ is preimposed (see for instance \cite{Olsson}).
\begin{definition}
A twisted curve $\ssC$ is a DM stack
whose coarse space is a stable curve,
whose smooth locus is represented by a scheme,
and whose singularities are nodes whose local picture is given by 
$[\{xy=0\}/\pmb \mu_r]$ 
with  $\zeta\in \pmmu_r$ acting as
$\zeta \cdot(x,y)=(\zeta x, \zeta^{-1}y)$.
\end{definition}

\subsubsection{Faithful line bundles}\label{subsect:faithlinebundles}
A line bundle $\ssL$ on a twisted curve $\ssC$
may be pulled back from the coarse space $C$
or from  
an intermediate twisted curve fitting in a sequence of morphisms 
$\ssC\to \ssC'\to C$ (with 
$\ssC'\neq \ssC$ and $C'=C$). The following condition rules out this possibility.
\begin{definition}\label{defn:repmap} 
A \emph{faithful} line bundle on a twisted curve is a line bundle $\ssL \to \ssC$ for which
the associated morphism $\ssC\to \ssB \mathbb G_m$ is representable.
\end{definition}
\begin{remark}\label{rem:explicitindices}
Let us phrase the condition explicitly in terms of 
the local picture of the fibre bundle mapping from the 
total space of $\levelii$ to $\leveli$. 
The local picture of $\ssL\to \ssC$
at a node $\sta n$ of $\ssC$ is the projection
$\mathbb A^1\times \{xy=0\}\longrightarrow \{xy=0\}$,
with $\zeta\in \pmmu_r$ 
acting as $\zeta\cdot(x,y)=(\zeta x,\zeta^{-1}y)$ on   $\{xy=0\}$ and as
$\zeta \cdot(t,x,y)=(\zeta^m t,\zeta x,\zeta^{-1}y)$
on  $\CC\times \{xy=0\}$ for a suitable index $m$ (modulo $r$). 
Notice that the 
index $m\in \ZZ/r$ is uniquely determined as soon as we assign a
 privileged choice of a branch
of the node on which $\pmmu_r$ acts by the 
character $1\in \mathrm{Hom}(\pmmu_r,\GG_m)$
(the action on the
remaining branch is opposite).
In this setting, we may restate faithfulness as follows
$$\text{$\ssL$ is faithful at $\sta n$} \quad \Longleftrightarrow \quad
\text{the representation $\ssL\mid _{\sta n}$ is faithful} \quad \Longleftrightarrow \quad
\gcd(m,r)=1.$$
Notice that if we switch the roles of the two branches, then $m$ changes sign
modulo $r$. Faithfulness does not depend on the sign of $m$ and on the choice of the branch.
\end{remark}

\subsubsection{Twisted curves and their level structures}\label{subsubsect:levelontwisted}
Once the notion of twisted curve and the notion of faithful line bundle
are given, level-$\ell$ structures 
are defined as for smooth curves.
This is the main advantage of the twisted curve approach.
\begin{definition}\label{defn:leveltwistedcurve} A \emph{level-$\ell$ twisted curve} $(\ssC\to B, \ssL, \ssphi)$
consists of a {twisted curve} $\ssC$ of genus $g$
over a base scheme $B$,
a faithful line bundle $\ssL$,
and an isomorphism $\phi\colon\ssL^{\otimes \ell}\to \cO_{\ssC}$.
We additionally require that the order of the isomorphism class of $\ssL$ in $\Pic(\ssC)$ is exactly $\ell$.
\end{definition}

The category of level-$\ell$ twisted curves forms a smooth DM
stack $\ol{\ssR}_{g,\ell}$ of dimension $3g-3$, with a finite forgetful morphism
over the stack of stable curves
${{\sta f}}\colon \ol {\sta R}_{g,\ell} \to \ol {\sta M}_g$ of degree $\deg({{\sta f}})=\Phi_{2g}(\ell)/\ell$ (or, simply,
$(\ell^{2g}-1)/\ell$ when $\ell$ is prime).
This definition is given implicitly in \cite{AV} 
by Abramovich and Vistoli (level-$\ell$  curves correspond to a connected
component of the moduli stack of
stable maps to $\sta B\pmb\mu_\ell$).
The forgetful morphism ${{\sta f}}$ is ramified as we illustrate in \S\ref{subsect:boundary}.
See also
work of the first author \cite{Cstab} for a slightly modified version, which preserves the \'etaleness  of
the forgetful morphism from level-$\ell$  smooth curves.

\begin{remark}\label{rem:backtoschemes}
We can regard the data $\sta L\colon \sta C\to \sta B\mathbb G_m$ alongside with 
$\phi\colon\ssL^{\otimes \ell}\to \cO_{\ssC}$
as a representable map $\sta f\colon \sta C\to \sta B\pmmu_\ell$.
Then, by exploiting the representability of the map  $\sta f$,
one can pullback the universal $\pmmu_\ell$-cover  
$\Spec k\to  \sta B\pmmu_\ell$ to $\sta C$ 
and obtain a scheme-theoretic curve $P$ equipped with a $\pmmu_\ell$-action. 
In this way we can equivalently interpret the data of a level curve $\level$ or, more 
simply the data of a map $\sta f\colon \sta C\to \sta B \pmmu_\ell$, 
as a $\pmmu_\ell$-action on a scheme-theoretic curve $P$, 
with  $\zeta\in \pmmu_r$ acting  as
$\zeta \cdot(x,y)=(\zeta x, \zeta^{-1}y)$ at each node $\{xy=0\}$.
We refer the reader to \cite{ACV} and 
 \cite[p.506, (i)--(iii)]{ACG} for this interpretation.
We notice that $P$, equipped with its $\pmmu_\ell$-action, 
is a $\pmmu_\ell$-torsor on $\ssC$ (notice that all fibres over 
geometric points $\Spec k\to \ssC$ consist of $\ell$ distinct points which constitute a $\pmmu_\ell$-orbit). On the other hand, when 
we regard $P$ as a cover of $C$ (after composition with $\sta C\to C$)
we get an admissible 
$\pmmu_\ell$-cover of the coarsening $C$ 
in the sense of \cite[p.506,(i)--(iii)]{ACG} (some orbits may consist of  
$\ell/r<\ell$ points and, in this case, all points in the orbit are nodes).
\end{remark}

\subsubsection{Local indices} \label{subsubsect:explicitlevel_twisted_indices}
Consider the local picture from Remark \ref{rem:explicitindices}
 of a level-$\ell$   curve at a node:
$$\zeta \cdot(t,x,y)=(\zeta^m t,\zeta x,\zeta^{-1}y),\qquad 
\zeta\in \pmmu_{r}.$$
Notice that $\ssL^{\otimes \ell}\cong \cO$ implies $(\zeta^m)^\ell=1$; that is $\ell m\in r\ZZ$ with
$r\ge 1$ and $m\in \{0,\ldots, r-1\}$.
Faithfulness implies $\gcd(r,m)=1$; hence $r\mid \ell$.
In the rest of the paper, we  often use a single \emph{multiplicity
index} $M=m\ell/r$ to encode the local indices $r$ and $m$:
\begin{eqnarray}\label{eq:rmM}
\ {} & r(M)=\frac{\ell}{\gcd(M,\ell)},\  m(M)=\frac{M}{\gcd(M,\ell)} &(M\in\{0,\ldots,\ell-1\}),
\\
\ {} & M(r,m)=m\ell/r  &(r\mid \ell, \ m\in \{0,\ldots, r-1\},
\ \gcd(r,m)=1). \notag
\end{eqnarray}

The first interesting example is $\ell=3$. In this case, $M$ equals $m$ and,
once we choose a privileged branch at a node,
there are three possible local pictures:\begin{itemize}
\item[$M=0$] (\emph{i.e.}~$(m,r)=(0,1)$), trivial stabiliser;
\item[$M=1$] (\emph{i.e.}~$(m,r)=(1,3)$), nontrivial 
$\pmb\mu_3$-action: the restriction of 
$\levelii$ to the privileged branch parametrized by $x$
is $\zeta\cdot(t,x)=(\zeta t,\zeta x)$ 
(with $\zeta\in\pmmu_3$);
\item[$M=2$] (\emph{i.e.}~$(m,r)=(2,3)$), 
nontrivial $\pmb\mu_3$-action: 
the restriction of 
$\levelii$ to the privileged branch parametrized by $x$
is $\zeta\cdot(t,x)=(\zeta^2 t,\zeta x)$ 
(with $\zeta\in\pmmu_3$).
\end{itemize}
Let us fix a 
node with a given choice of a branch falling under  
case $(M=1)$ (resp. $(M=2)$); note that, 
if we change the choice of the branch, this case falls under
case $(M=2)$ (resp. $(M=1)$).
Therefore, we can
summarise
this analysis by saying that the nodes of
level-$3$ twisted  curves are either trivial ($M=0$) or
nontrivial $(M\neq 0)$ and, in this case, equipped with a
distinguished choice of a branch 
so that $M$ equals $1$.

\subsection{Dual graphs of twisted curves and multiplicity of level curves}\label{subsect:dualgraphsenriched}
The dual graph of a twisted curve is simply the dual graph of the coarse curve.

\subsubsection{Dual graphs.}
Dual graphs arising from the standard construction recalled below are
connected nonoriented graphs, possibly
containing multiple edges (edges linking the same two vertices) and
loops (edges starting and ending at the same
vertex). Consider a twisted curve $\ssC$ and its 
normalisation $\sta {nor}\colon {\ssC}' \to \ssC$.
Locally at a node of $\ssC$ the normalisation 
is given by $[\Spec \CC[x]/\pmmu_r] \sqcup
[\Spec \CC[y]/\pmmu_r] \longrightarrow [\{xy=0\}/\pmmu_r]$ with 
$\zeta\in \pmmu_r$ operating 
on $x$ as $\zeta\cdot x=\zeta x$ and on $y$ as $\zeta\cdot y=\zeta^{-1} y$.

\begin{definition}
The vertex set $V$ of the dual graph is the set of connected components of ${\ssC}'$.
The edge set $E$ of the dual graph is the set of nodes of $\ssC$. The two sets $V$ and $E$ determine a
graph as follows: a node identifies the
connected components of ${\ssC}'$ where its preimages lie,
in this way an edge links two (possibly equal) vertices.
\end{definition}

  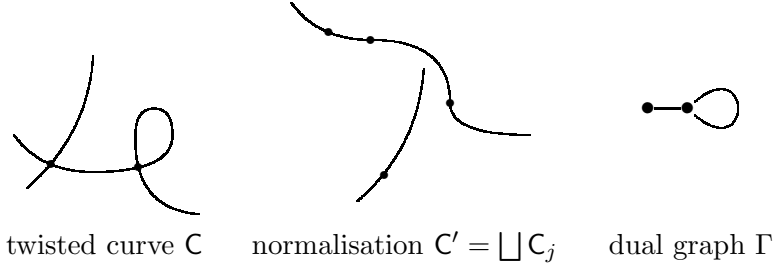
\begin{figure}[h]
  \begin{picture}(400,80)(-35,0)
   \qbezier(30,20),(53,40),(55,70)   
   \qbezier(25,40),(35,26),(55,26)     
   \qbezier(55,26),(85,26),(85,40)   
   \qbezier(85,40),(85,50),(78,50)  
   \qbezier(78,50),(71,50),(71,40) 
   \qbezier(71,40),(70,12),(95,10)
   \put(72,28){\circle*{3}}
   \put(39,29){\circle*{3}}
   \put(22,-5){\text{twisted curve $\ssC$}}
 \qbezier(155,15),(178,35),(180,65)   
   \put(165,25){\circle*{3}}
   \qbezier(130,90),(140,76),(160,76)     
   \put(144,79){\circle*{3}}
  \put(160,76){\circle*{3}}
   \qbezier(160,76),(190,76),(190,52)   
   \put(190,52){\circle*{3}}
    \qbezier(190,52),(190,40),(220,40)
     \put(115,-5){normalisation $\ssC'=\bigsqcup \ssC_j$}
   \put(262,50){\xymatrix@=.8pc{*{\bullet}\ar@{-}[r] & *{\bullet} \ar@{-} @(rd,ur) } }
    \put(250,-5){\text{dual graph $\Gamma$}}
  \end{picture}
 \caption{Normalisation and dual graph of a twisted curve} \label{fig:normdualgenus}
  \end{figure}

\subsubsection{Cochains.}
Each node of $\ssC$ has two branches.
Let $\mbb E$ be the set of branches of each node of $\ssC$.
The cardinality of $\mbb E$ is twice that of  $E$; there is a $2$-to-$1$ projection
$\mbb E\to E$ and an involution $e\mapsto \ol e$ of $\mbb E$.
On $\mbb E$ we can define a function $\mbb E\to V$, noted $e\mapsto e_+\in V$, assigning to
each oriented edge the vertex $v=e_+$ corresponding to the connected component $\ssC'_v$ of $\ssC'$ where the chosen
branch lies. We get $e\mapsto e_-$ by applying $(\ )_+$ after the involution.
If $e_+=e_-$ we have a loop (Figure \ref{fig:normdualgenus}): 
$\ol e\neq e$ in $\mbb E$
map to the same vertex via $e\mapsto e_+$.

We define the group of $1$-cochains and $0$-cochains of the dual graph with coefficient in $\ZZ$. We define
$C^0(\Gamma)$ as the set of $\ZZ$-valued functions on $V$
$$C^0(\Gamma)=\{a\colon V\to \ZZ\}=
{\bigoplus}_{\substack{v\in  V}} \ZZ.$$
We define $1$-cochains as antisymmetric $\ZZ$-valued functions 
on $\mbb E$
$$C^1(\Gamma)=\{b\colon \mbb E\to \ZZ
\mid b(\ol e)=-b(e)\},$$
where $\ol e$ and $e$ are oriented edges with opposite
orientations.
After assigning an orientation for each edge $e\in E$, we may identify
$C^1(\Gamma)$ to
${\bigoplus}_{{e\in E} } \ZZ,$ but 
we prefer working  with $\mbb E$.

The space of $\ZZ$-valued 
$0$-cochains and $1$-cochains $C^0(\Gamma)$
and $C^1(\Gamma)$
are equipped with non-degenerate bilinear $\ZZ$-valued forms
\begin{equation}\label{eq:pairings}
\langle a_1,a_2 \rangle= \sum\nolimits_{v\in V} a_1(v)a_2(v)\qquad
\langle b_1,b_2\rangle=\frac12 \sum\nolimits_{e\in \mbb E} b_1(e)b_2(e)
\end{equation}
with $a_1,a_2\in C^0$ and  $b_1,b_2\in C^1$.
The exterior differential is
\begin{align*}
\qquad
\delta\colon C^0(\Gamma)&\to C^1(\Gamma),\\
a&\mapsto \delta a,\qquad \qquad \qquad \text{ with \ \ $\delta a(e)=a(e_+)-a(e_-)$}.
\end{align*}
The adjoint operator with respect to $\langle\ \ ,\ \rangle$ is given
by
\begin{align*}
\partial\colon C^1(\Gamma)&\to C^0(\Gamma), \\
b&\mapsto \partial b, \qquad \qquad \qquad \text{ with\ \  $\partial b(v)=\sum_{\substack{{e\in\mbb E}\\{e_+=v}}} b(e)$.}
\end{align*}

\begin{remark}[cuts and circuits]\label{rem:circuits}
The image
$\im(\delta)$ is freely generated by $\#V-1$  {cuts} (see \cite[Ch.~4]{Biggs}),
\begin{equation}\label{eq:imiso}
\im(\delta)\cong\ZZ^{\oplus (\#V-1)}.                                           \end{equation}
We recall that a cut
is determined by a proper nonempty subset 
$W$ of the vertex set $V$ of $\Gamma$:
the sets $W$ and $V\setminus W$ form a partition of $V$.
Cuts are $1$-cochains  $b\colon \mathbb E\to \ZZ$ in 
$C^1(\Gamma)$
equal to $1$ on the (nonempty) set $H_W$ of edges
having only one end on $W$ 
and oriented from $W$ to $V\setminus W$,
equal to $-1$ on $\ol H_W=\{\ol e\mid e\in H_W\}$, and vanishing elsewhere.
By construction,  $H_W$ and $\ol H_W$ contain
no loops.
For this reason, in graph theory literature the image of 
$\delta$ is often referred to as the cut space.

The kernel $\ker \partial$ is freely generated by
$b_1=1-\chi(\Gamma)=1-\#V+\#E$  {circuits}
$$\ker\partial\cong \ZZ^{\oplus (1-\#V+\#E)}.$$
We recall that a {circuit} within a graph
is a sequence of $n$ oriented
edges $e_0, \ldots, e_{n-1}\in \mbb E$ labelled 
by $i\in \ZZ/n$, overlying $n$ distinct nonoriented 
edges in $E$, 
so that the head $(e_i)_+$ is also the tail 
$(e_{i+1})_-$
for all $i\in \ZZ/n$ 
and the $n$
vertices $v_i=(e_i)_-$ are distinct.
If we remove the condition $(e_0)_-=(e_{n-1})_+$
we obtain a path of edges joining $v_0=(e_0)_-$ to $v=(e_{n-1})_+$.
Here, we treat circuits and paths 
as $1$-cochains regarding their
characteristic function
(given by $1$ on $e_i$, $-1$ on $\ol{e}_i$ and $0$ elsewhere)
as an element of $C^1(\Gamma)$.
Circuits formed by a single oriented edge will
be called loops. 

Since $\delta$ is the adjoint of $\partial$,
for every element of $s\in \ker \partial$ and 
$\delta t\in \im\delta$ we have $\langle s,\delta t\rangle= 
\langle \partial s, t\rangle =0$. 
Conversely, the condition 
$\langle s,b\rangle=0$ for all $s\in \ker \partial$ implies 
$b\in \im \delta$. In order to see this, (for every connected 
component) fix 
a vertex $v_0\in V$ and define 
$a\in C^0(\Gamma)$ as $a(v)=\sum_i^{n-1} b(e_i)$ for a
path joining $v_0$ to $v$. The definition of $a$
does not depend on the chosen path because the difference
between two paths lies in $\ker \partial$ and we have
$\langle s,b\rangle=0$ for all $s\in \ker \partial$.
By construction, we have $\delta a=b$.
In this way, we get a simple criterion for 
$b\in C^1(\Gamma)$ to lie in  $\im\delta$:
\begin{equation} \label{eq:imandcirc}\text{$b\in C^1(\Gamma)$ is in $\im\delta$}\  \Longleftrightarrow\  
\text{$b(K)=\textstyle{\sum_{0\le i<n}} b(e_i)=0$ for all circuits $K=\textstyle{\sum_{0\le i<n} e_i}$ of $\Gamma$}.
\end{equation}
\end{remark}

\begin{remark}
For any abelian group $A$, by taking  
$\partial\otimes_\ZZ A$ and $\delta\otimes_{\ZZ} A$,
we recover the 
simplicial cohomology and homology complexes 
with coefficients in $A$ 
$$\delta_A\colon C^0(\Gamma;A)\to C^1(\Gamma;A),\qquad 
\partial_A\colon C^1(\Gamma;A)\to C^0(\Gamma;A)\qquad 
(C^i(\Gamma;A)=C^i\otimes_{\ZZ}A).$$
The forms \eqref{eq:pairings}
extend to pairings of the form
$$\langle \ , \ \rangle \colon C^0(\Gamma)\otimes_{\ZZ} C^0(\Gamma;A)\to A
\quad \text{and} \quad 
\langle \ , \ \rangle \colon C^1(\Gamma)\otimes_{\ZZ} C^1(\Gamma;A)\to A.$$
with the same definition, where  $a_1(v)a_2(v)$ is in $A$, 
for $a_1(v)\in \ZZ$ and $a_2(v)\in A$,
and similarly $b_1(e)b_2(e)$ is in $A$, 
for $b_1(e)\in \ZZ$ and $b_2(e)\in A$.
Notice that 
we still have equations of the form  
$\langle \delta s_0, t_1\rangle=\langle s_0,\partial_A t_1\rangle$
and $\langle s_1, \delta_A t_0\rangle=\langle \partial s_1,t_0\rangle$ 
for any $s_i\in C^i(\Gamma)$ and $t_j\in C^j(\Gamma;A)$. 

Then, any $\delta_A t\in \im\delta_A\subset C^1(\Gamma;A)$ and 
any $s\in \ker \partial\in C^1(\Gamma)$ 
satisfy $\langle s,\delta_A t\rangle=\langle \partial s, t\rangle= 0$. 
As above, the condition 
$\langle s,b\rangle =0$ for 
all $s\in \ker\partial$ implies $b\in \im\delta_A$.
We conclude that \eqref{eq:imandcirc} still holds.
More precisely, for a circuit $K=\sum_{0\le i<n} e_i$, 
we set 
$b(K):=\textstyle{\sum_{0\le i<n} b(e_i)}$
so that the claim \eqref{eq:imandcirc} generalises for 
$C^1(\Gamma;A)$ and $\im\delta_A$  \emph{verbatim}.
Notice that if $A$ is multiplicative (\emph{e.g.} $\pmb \mu_\ell$ and 
$\mathbb G_m$ below) all notations should be 
read accordingly; for instance, the condition
$b(\overline e)=-b(e)$ defining $C^1(\Gamma,A)$ should be read as $b(\overline e)=b(e)^{-1}$ and, similarly, 
the sum $\textstyle{\sum_{0\le i<n} b(e_i)}$
defining $b(K)$ just above should be read as
$\textstyle{\prod_{0\le i<n} b(e_i)}$.

\end{remark}

\subsubsection{The group of line bundles with trivial normalisation 
and the $\mathbb G_m$-valued cut space.}\label{sssect:taudefn}
The cohomology of the short exact sequence of sheaves
 $1\to {\mathbb G_m}\to \sta{nor}_*\sta{nor}^*{\mathbb G_m}
 \to{\mathbb G_m}|_{{\Sing}\ssC}\to 1$ and the analogue 
 sequence for $\pmb \mu_\ell$ yields the exact sequences
 \begin{eqnarray}\label{eq:Picsequenceall}
&C^0(\Gamma; \mathbb G_m)\xrightarrow{\ \delta\ } C^1(\Gamma; \mathbb G_m)\xrightarrow{\ \tau \ } \Pic(\ssC)
\xrightarrow {\ \sta {nor}^*\ } \Pic(\ssC'),\\
&\label{eq:Picsequence}C^0(\Gamma; \pmb \mu_\ell )\xrightarrow{\ \delta\ } C^1(\Gamma; \pmb \mu_\ell )\xrightarrow{\ \tau \ } \Pic(\ssC)[\ell]
\xrightarrow {\ \sta {nor}^*\ } \Pic(\ssC')[\ell].
\end{eqnarray}
Let us state explicitly  the definition of the homomorphism $\tau$. 
It is enough to consider a 
$1$-cochain $b$ vanishing on all edges except
$e_0$ and $\ol{e}_0$ where it equals $\zeta$ and $\zeta^{-1}$ respectively (the cochain is $\mathbb G_m$-valued  
if $\zeta$ lies in $\mathbb G_m$ and 
$\pmb\mu_\ell$-valued if $\zeta^\ell$ equals $1$).
The line bundle $\tau(b)$
is the locally free sheaf of regular functions $f$ on the normalisation of $\ssC$ at
the node $\sta n$
satisfying $f(\sta x)=\zeta f(\sta y)$ for $\sta x$ and $\sta y$ pre-images of $\sta n$,
with $\sta x$ lying on the branch corresponding to $e_0$ 
and $\sta y$ lying on the remaining branch.

\subsubsection{Line bundles on an $\ell$-twisted 
curve $\sta C$ up to pullbacks from $C$ are $\pmb\mu_\ell$-valued 
circuits.}
For any stable curve $C$, up to isomorphism, we can define a 
unique twisted curve $\wt{\sta C}$ with order-$\ell$ stabilisers 
at all nodes. We may call the curve $\sta C$ the 
$\ell$-twisted curve attached to $C$ (in another context
\cite{Cstab} this is called $\ell$-stable curve, because 
imposing that all stabilisers have the same cardinality amounts to 
a stability 
condition). We consider the line bundles of $\Pic(\wt{\sta C})$
up to pullbacks from $\Pic(C)$, or---which is the 
same---$\Pic(\wt{\sta {C}})[\ell]$ modulo
$\Pic(\wt{C})[\ell]$. 
By \cite[Cor.~3.1]{Cstab}, the long exact sequence of 
cohomology of 
the Kummer sequence $1\to \pmb\mu_\ell\to \mathbb G_m\to \mathbb G_m\to 1$ combined with that of 
$1\to {A}\to \sta{nor}_*\sta{nor}^*{A}
 \to{A}|_{{\Sing}\ssC}\to 1$ for $A=\mathbb G_m$ and $\pmb\mu_\ell$, 
yields the exact  sequence 
$$1\to \Pic(C)[\ell]\to \Pic(\wt \ssC)[\ell]\xrightarrow{} C^1(\Gamma;\ZZ/\ell)\xrightarrow{\ \partial \ } 
C^0(\Gamma;\ZZ/\ell).$$
Here, it should be noticed 
that the cohomology with coefficients in  $\pmb \mu_\ell$ 
naturally produces $\ZZ/\ell$-valued cochains. 
For instance, the $\pmb\mu_\ell$-valued 
second cohomology group of a curve is identified canonically 
with $\ZZ/\ell$, see \cite[\S14]{Milne}.
On the other hand 
$C^1(\Gamma;\ZZ/\ell)$ 
equals $\bigoplus_{e\in E} H^1(B\pmb\mu_\ell, \pmb\mu_\ell)$,
where each summand is the $\ell$-torsion subgroup of 
the group of characters  
$\mathrm{Hom} (\pmb \mu_\ell, \mathbb G_m)$, 
which---by definition---equals $\ZZ/\ell$.

\subsubsection{Multiplicity and $\ker \partial$.} \label{subsubsect:multiplicity}
Since oriented edges are in one-to-one correspondence with branches of nodes of $C$,
using \S\ref{subsubsect:explicitlevel_twisted_indices}, we define the multiplicity cochain.
\begin{definition}[$\ZZ/\ell$-valued multiplicity $1$-cochain of level-$\ell$ curves]\label{defn:mult_twisted}
Consider a  level-$\ell$ curve $\level$. To each oriented edge $e$, we can attach the
multiplicity $M(e)$ of $\level$ at the node (with its prescribed branch).
The function $M\colon e\mapsto M(e)\in \ZZ/\ell$ satisfies $M(\ol e)=-M(e)$ for all $e\in \mbb E$; in this way we have 
$$M\in C^1(\Gamma; \ZZ/\ell).$$
\end{definition}

\begin{proposition}\label{pro:Misclosed}
Let $C$ be a stable curve and consider the set of level-$\ell$ curves $\level$
with coarsening $C$.
Consider the dual graph of $C$ and the
differential $\partial$.
Then 
associating to $\level$ its multiplicity 1-cochain $M$ 
yields a surjective map from the set of level-$\ell$ 
curves with coarsening C to $ker \partial\subseteq C^1(\Gamma; Z/\ell)$.
\end{proposition}
\begin{proof}
All level-$\ell$  structures
overlying $C$ may be regarded
as elements in
$\Pic(\wt \ssC)[\ell]$, where $\wt{\sta C}$ is the 
twisted curve $\wt \ssC$
with $\pmb\mu_\ell$-stabilisers at all nodes 
(note that in  
$\Pic(\wt \ssC)[\ell]$ we do not impose faithfulness).
The multiplicity cochain lifts to a homomorphism $M\colon \Pic(\wt \ssC)[\ell]\to C^1(\Gamma;\ZZ/\ell)$.
The above claim follows from the exact  sequence 
$1\to \Pic(C)[\ell]\to \Pic(\wt \ssC)[\ell]\xrightarrow{} C^1(\Gamma;\ZZ/\ell)\xrightarrow{} C^0(\Gamma;\ZZ/\ell)$
(see \cite[Cor.~3.1]{Cstab})
and from the existence of an element in 
 $\Pic(C)[\ell]$ (for $g\ge 2$) whose order equals $\ell$.
 \end{proof}
 
 \begin{remark}\label{rem:Cesar}
 The above $\ZZ/\ell$-valued $1$-cochain is 
 (in another context) the same as the $1$-cochain 
 associated to a weighted subgraph of $\Gamma$ in 
 \cite[Rem.~2.2.1]{CCC}, including the fact that it takes values in $\ker\partial$. 
 Similarly, the description of ghost automorphisms 
 for $\ell$-prime in \S2.2 is related to Proposition 4.1.11 in 
 \cite{Ja1} and Lemma 2.3.2 in \cite{CCC}.
\end{remark}
 
 \begin{remark}
 The fact that $M$ takes values in $\ker\partial$ 
 may be regarded as saying that 
 the multiplicities $M_1,\dots, M_N$ at all special points 
 $\sta p_1,\dots, \sta p_N$ of a connected component 
 of the normalisation must add up to zero mod $\ell$.
 This is easy to see also directly once we express the line bundle induced by $\levelii$ 
 on a connected component $\sta X$ of the normalisation of $\leveli$
 as the line bundle $\mathcal O_{\sta X}(D)$, where $D$ is a Cartier 
 divisor on a smooth stack-theoretic curve $\sta X$. The divisor $D$ 
 is the sum of a divisor $I$ with 
 integer coefficients plus the rational coefficients divisor $S=\sum_{i=1}^N\frac{M_i}{\ell} [\sta p_i]$
 supported on the special points. 
 The condition $\levelii^{\otimes \ell}\cong \mathcal O$ implies 
 that $\ell I+\ell S$ is a principal divisor; in particular, it has 
 degree zero and we have $$\textstyle{\sum_{i=1}^N }M_i= \ell \deg S=-\ell \deg I\in \ell\ZZ.$$
 \end{remark}

\begin{example}\label{exa:twisted}
Consider a two-component twisted curve obtained as the union of
two smooth one-dimensional stacks $\sta X$ and $\sta Y$ meeting transversely at $2$ nodes.
For each node, let us measure the multiplicities
with respect to the branch lying in $\sta X$.
Proposition \ref{pro:Misclosed} says that
the multiplicities $M_1$ and $M_2$ should add up to $0$  (modulo $\ell$).
Let us examine in greater detail the
case $\ell=3$, $M_1=1$ and $M_2=2$.
Over $\sta X$ the third root $\ssL$ of $\cO$ is given by a divisor $D'$
of degree $0$ (a root of $\cO_{\sta X}$) with rational coefficients
of the form
$D'=\lfloor D'\rfloor + {[\sta x_1]} /3+ 2 {[\sta x_2]}/3,$
where $\sta x_1\colon \Spec \CC\to \sta X$ and $\sta x_2\colon \Spec \CC\to \sta X$
are the geometric points
lifting $\sta n_1$ and $\sta n_2$ to $\sta X$.
Conversely $\ssL|_{\sta Y}$ can be expressed as the degree-$0$ line bundle $\cO(D'')$ with
$D''=\lfloor D''\rfloor + 2 {[\sta y_1]}/3 + {[\sta y_2]}/3,$
where, again, $\sta y_1$ and $\sta y_2$
lift $\sta n_1$ and $\sta n_2$ to $\sta Y$.
\end{example}

The multiplicity $1$-cochain encodes much of the relevant
topological information characterising a level curve.
In what follows, we describe some natural invariants of
$\ZZ/\ell$-valued $1$-cochains.

\subsubsection{The support and its characteristic function.}
For any $1$-cochain $c\colon \mbb E\to \ZZ/\ell$ we consider  
the characteristic function of the support of  $c$
taking values in the extended set $\ZZ\cup \{\infty\}$
(we use the standard conventions $a < \infty$ and 
$a + \infty =\infty$ for $a \in \ZZ$).
\begin{equation}\label{eq:single-valued_characteristic}
\nu_c(e)=\begin{cases}
\infty      & \text{if } c(e)=0\in \ZZ/\ell  \\
0      & \text{otherwise.}
     \end{cases}
\end{equation}
Proposition \ref{pro:Misclosed} implies 
$\nu_c(e)=\infty$ for
any separating edge.

For any abelian group $A$, we 
present a natural subcomplex
$C^\bullet_\nu(\Gamma;A)$ of $C^\bullet(\Gamma;A)$
attached to a given symmetric characteristic 
function $\nu\colon \mbb E\to \{0,\infty\}$;
\emph{i.e.} to any subset of $E$.
In \S\ref{subsect:compositel} we generalise
this construction by allowing, instead of characteristic functions, more
general functions
arising
as the truncated valuations of $M$,
see \eqref{eq:vector-valued_characteristic}.
When $\ell$ is prime we recover 
the above defined function $\nu_c$.

\subsubsection{The contracted graph $\Gamma(\nu)$.}
We define precisely the graphs 
obtained by iterated edge-contractions of $\Gamma$ mentioned in the introduction.
Let us consider any symmetric characteristic function $\nu\colon \mbb E\to \{0,\infty\}$
(since $\nu$ is symmetric it descends to $E$ and
we sometimes abuse the notation by regarding it as a function on $E$). We attach
to $\Gamma$ a new graph $\Gamma(\nu)$ whose
sets of vertices and edges  $(\ol V,\ol E)$  are obtained from  $(V,E)$
\begin{enumerate}
    \item
by setting $E(\nu)=\{e\mid \nu(e)=0\}$,
\item
by modding out $V$ by the relations
$e_+\sim_\nu  e_-$ if $\nu(e)=\infty$,
\emph{i.e.} $V(\nu)=V/\sim_\nu$.
\end{enumerate}
In the new graph, the set of vertices of the edge $e\in E(\nu)$
is the set of vertices of $e \in E$ in $V$ modulo the relation $\sim_\nu$.
In simple terms, $\Gamma(\nu)$ is the contraction of all edges where $v>0$.
We  refer to $\Gamma(\nu)$ as a contraction of $\Gamma$ and,
conversely, to $\Gamma$ as a blowup of $\Gamma(\nu)$
(often in graph theory literature, the graph obtained from an
iterated edge-contraction is a ``minor'' of the initial graph, but we do not use this terminology here.)

\subsubsection{The complex {$C^\bullet_\nu(\Gamma;A)$}.}
The inclusion $i\colon  E(\nu) \hookrightarrow E$
and the projection $p\colon V\twoheadrightarrow V(\nu)$
yield homomorphisms
$p_*\colon C^0(\Gamma;A)\twoheadrightarrow C^0(\Gamma(\nu);A)$
and $i^*\colon C^1(\Gamma;A)\twoheadrightarrow C^1(\Gamma(\nu);A)$
and the \emph{contraction homomorphism} between complexes with differentials given by  $\partial$
\begin{equation}\label{eq:contraction}
\mathcal C\colon (C^\bullet(\Gamma;A), \partial)\twoheadrightarrow
(C^\bullet(\Gamma(\nu);A),\partial).
\end{equation}
Conversely, the homomorphisms  
$p^*\colon C^0(\Gamma(\nu);A)\hookrightarrow C^0(\Gamma;A)$
and $i_*\colon C^1(\Gamma(\nu);A)\hookrightarrow C^1(\Gamma;A)$ 
yield  the \emph{blowup homomorphism} between complexes with differential $\delta$
\begin{equation}\label{eq:blowup}
\mathcal B\colon (C^\bullet(\Gamma(\nu);A), \delta)\hookrightarrow
(C^\bullet(\Gamma;A),\delta).\end{equation}
The subcomplex 
$\mathcal B(C^\bullet(\Gamma(\nu);A),\delta)$ 
consists of the $0$-cochains $a\in C^0(\Gamma;A)$
and the $1$-cochains $b\in C^1(\Gamma;A)$
satisfying
$a(e_+)=a(e_-)$ and
$b(e)=0$ if  $\nu(e)=\infty$.
Within $(C^\bullet(\Gamma;A),\delta)$
we denote such a subcomplex by
$$C_{\nu}^\bullet(\Gamma;A)\subseteq C^\bullet(\Gamma;A)$$
In fact, we have
\begin{equation}\label{eq:imcap}
\mathcal{B} ( \im (\delta))
=\im (\delta)
\cap C^1_\nu(\Gamma;A).
\end{equation}
The inclusion from left to right follows from \eqref{eq:blowup}. Conversely,
$b=\delta(a)$
is in $C_\nu^1(\Gamma;A)$
only if, for any contracted edge $e$,
we have $a(e_+)=a(e_-)$; that is only if $a$ lies in $C_\nu^0(\Gamma;A))$.
Passing to the adjoint operator we also get
\begin{equation}\label{eq:inclusions}\mathcal{C}(\ker\partial)=\ker\partial.\end{equation}
Summarising, the contraction of a  circuit is a circuit
and  the blowup  of a cut is a cut.

\subsection{The boundary locus}\label{subsect:boundary}
We describe $\ol{\ssR}_{g,\ell}\setminus\ssR_{g,\ell}$ by classifying one-nodal level curves.
\subsubsection{Reducible one-nodal curves}\label{subsubsect:seplprime} Consider the union
$\ssC=\ssC_1\cup \ssC_2$ of two smooth stack-theoretic curves $\ssC_1$ and
$\ssC_2$ of genus $i$ and $g-i$ meeting transversally at
a point.
Proposition \ref{pro:Misclosed} implies
that the node has multiplicity zero or, in other words, trivial stabiliser. Hence, we have $\ssC=C$;
\emph{i.e.} $\ssC$ is an ordinary stable curve of compact type
$C=C_1 \cup C_2$.
The line bundle $\ssL=L$ on $C$  is determined by the choice of two line
bundles $L_1$ and $L_2$
satisfying $L_1^{\otimes \ell}\cong\cO_{C_1}$ and
$L_2^{\otimes \ell}\cong\cO_{C_1}$ respectively.
There are three possibilities:
\[\text{(i) $L_1\cong \cO$, $L_2\not\cong \cO$; \qquad 
(ii) $L_1\not \cong \cO$, $L_2\cong \cO$\qquad (iii) $L_1, L_2\not\cong \cO$}\]
(since  $\ssL\not\cong \cO$,  the possibility that both line bundles are trivial is excluded).
If $0<i<g/2$, these three cases characterise
three loci in the moduli space whose closures are
the divisors $\Delta_{g-i}, \Delta_{i}$ and $\Delta_{i:g-i}$ respectively.
We write $\delta_{g-i}, \delta_i$ and $\delta_{i:g-i}$ for the corresponding 
$\QQ$-divisors defined by the same conditions in the moduli stack. 
The morphism ${{\sta f}}$ is not
ramified along these divisors. We have that
\begin{equation}\label{eq:deltai}
{{\sta f}}^{*}(\delta^{\stable}_i)= \delta_{g-i} + \delta_{i} + \delta_{i:g-i}\end{equation}
where
$\delta^{\stable}_{i}$ is the $\QQ$-divisor class in $\ol {\sta M}_g$ defined by stable curves
with at least one node separating the curve
into two components of genus $i$ and $g-i$.

If $i=g/2$ the same classification
reduces to two divisors:
the closure of the
locus of one-nodal level curves for which only
one line bundle among $L_1$ and $L_2$ is trivial yields
$\Delta_{g/2}$, the closure of the locus classifying curve where both $L_1$ and $L_2$ are nontrivial yields
$\Delta_{g/2:g/2}$.

\subsubsection{Irreducible one-nodal curves} \label{subsubsect:irredlprime}
If $\ssC$
is irreducible and has one node,
then the node is of nonseparating type:
the normalisation $\sta {nor}\colon \ssC'\to \ssC$ is
given by a connected curve.
There are three possibilities:
\[\text{(i) $M=0$ and $\sta{nor}^* \ssL\not \cong \cO$; \qquad 
(ii) $M=0$ and $\sta{nor}^* \ssL \cong \cO$;\qquad (iii) $M\neq 0$}\]
The closures of
the loci of level curves satisfying the three conditions above
determine three divisors denoted by $\Delta_0',\Delta_0'', \Delta_0^{\ram}$ in the moduli space.
We write again $\delta_0',\delta_0'', \delta_0^{\ram}$ for the corresponding 
classes of divisors defined by the same conditions in the moduli stack.
The morphism $\mathsf{f}$ is not ramified along
$\delta_{0}'$ and $\delta_{0}''$. When $\ell$ is prime, $\mathsf{f}$ is
ramified with order $\ell$ along
$\delta_{0}^{\ram}$. Precisely, we have that, cf. \cite{CEFS}
\begin{equation}
\label{eq:lramification}{{\sta f}}^{*}(\delta^{\stable}_0)=
\delta_{0}' + \delta_{0}'' + \ell \delta_{0}^{\ram} \qquad (\ell \text{ prime}).                                                 \end{equation}
In general, $\delta_{0}^{\ram}$ can be decomposed 
into several components depending on the value of
the multiplicity index $M$; we refer to \S\ref{subsubsect:irredlcomp} for the study of the order of the ramification. 

This calls for
an analysis of the irreducible components of the boundary divisors
 $\delta_0',\delta_0'', \delta_0^{\ram}$ as well as for the previous divisors $\delta_i, \delta_{i:g-i}$, for $1\le i\le g/2$.
We  carry it out in the last part of this section (\S\ref{subsubsect:seplcomp} and \S\ref{subsubsect:irredlcomp}) as
a nice application of the combinatorial invariants  of level curves illustrated above.
On the other hand,
the present description of the boundary locus is sufficient for the entire Section \ref{sect:sing}
and may already be already regarded as a decomposition into irreducible components of the boundary for
$\ell=3$ (see Examples \ref{rem:level3reddecomp} and \ref{rem:level3irrdecom}).
Therefore, it is worthwhile to illustrate it further by an example, which will play an important role
in the rest of the paper: the case of level structures on elliptic-tailed curves.

\begin{example}[two level-$\ell$ structures on the elliptic-tailed curve]\label{exa:level_on_elltails}
We provide examples of two distinct twisted level curves, one representing a point
of $\Delta_{1}\cap \Delta_0^{\ram}$, and the other
representing a  point in $\Delta_{1}\cap \Delta_{0}''$.
Consider the stack-theoretic quotient $\sta E$ of
$\wt E=\PP^1/(\Omega_{P'})$ by $\pmb\mu_\ell$, with
$\zeta\in \pmmu_\ell$ operating by multiplication on the 
local parameter of $\PP^1$ at $0$. 
Now let $\ssC$ be a twisted curve
containing, as a subcurve, a copy of
such a genus-$1$ stack-theoretic curve $\sta E$.
We assume $\sta E\cap \ol{\sta C\setminus \sta E}=\{n\}$, where $n$ is a separating node with
trivial stabiliser (see Proposition \ref{pro:Misclosed}).

Level-$\ell$ structures in $\Delta_{1}$ can be defined on $\ssC$
by extending trivially on $\ol{\ssC\setminus \sta E}$
nontrivial $\ell$th roots of $\cO$ on $\sta E$.
To this effect, we can exploit
$\sta p\colon \wt E\to \sta E$, which is an \'etale $\pmb\mu_\ell$-cyclic cover of $\sta E$. The rank-$\ell$ locally free sheaf 
$\sta p_*\cO$ carries a $\pmmu_\ell$-representation 
and admits an isotypical decomposition $\sta p_*\cO=
\bigoplus_{\chi \in \ZZ/\ell=\mathrm{Hom}(\pmmu_\ell,\GG_m)}\sta L_\chi$.
We set 
$\levelii_{\ram}:=\sta L_1$, where $\chi=1$ is the character
$(1\colon \pmmu_\ell\subset \GG_m)\in \ZZ/\ell$.
In this way $\levelii_{\ram}$ is  
is equipped with an isomorphism 
$\leveliii_{\ram}\colon \ssL_{\ram}^{\otimes \ell}\cong \cO$.
Then $\sta L_{\ram}\to \ssC$
yields an object $(\ssC,\sta L_{\ram},\leveliii_{\ram})$ in $\Delta_{1}\cap \Delta_0^{\ram}$ because the multiplicity of $\sta L_{\ram}$ at
the nonseparating node is $\neq 0$ ($1$ or $l-1$ depending on the chosen branch).

The projection to the coarse space $\epsilon_\sta E \colon \sta E \to E$ allows us to 
define other nontrivial 
line bundle in $\Pic(\ssC)[\ell]$ as follows. 
On $\sta E$,  simply consider the
pullback  of
the line bundle of regular functions $f$
on the normalisation $E'\cong \PP^1$ satisfying $f(\infty)=\zeta f(0)$ for any $\zeta\in \pmmu_\ell$.
This is $\tau(\zeta)$ in the notation of \S\ref{sssect:taudefn}.
If $\zeta$ is a primitive root of unity, than we get a
line bundle $\ssL_{\etale}\to \sta C$ yielding a point 
in
$\Delta_{1}\cap \Delta_0''$ (the multiplicity at the nonseparating node is $0$ and $\ssL_{\etale}$ is trivial on
the normalisation
by construction).
\end{example}

\subsubsection{The closure of the locus of reducible one-nodal curves: irreducible components.}\label{subsubsect:seplcomp}
We provide a decomposition into irreducible components of the divisor defined above as 
the closure of the substack of reducible level-$\ell$ one-nodal curves.  
It is convenient to reformulate the problem in $\ol{\sta M}_g$: we study the 
divisor 
\begin{equation}\label{eq:Dstablered}
{\mathsf{D}}^{\stable}_{\operatorname{red}}=\sum\nolimits_{1\le i\le g/2} {\delta}^{\stable}_i                                                                                      \end{equation}
of stable curves with at least one separating node. We do so, by analysing the 
degree-$2$ map 
$\wt {\mathsf D}_{\operatorname{red}}^{\stable}\to {\mathsf{D}}^{\stable}_{\operatorname{red}}$ classifying stable curves alongside with a
separating node and a branch of the node. We have the natural decomposition
$\wt{\mathsf{D}}^{\stable}_{\operatorname{red}}= \bigsqcup_{i=1}^{g-1}\wt{\mathsf{D}}^{\stable}_{i}$
where $\wt{\mathsf{D}}^{\stable}_{i}$ classifies objects where the chosen
branch lies in the genus-$i$ connected component $\sta Z$ of the normalisation of the separating node.
Then, for $i=1,\dots, g-1$, we write ${\mathsf{D}}_i^{\stable}$
for the pushforward in $\ol {\sta M}_g$ of
the cycle $\wt{\mathsf{D}}_i^{\stable}$
via the map forgetting the branch; for $i\neq g/2$, the forgetful map from
$\wt {\mathsf{D}}_{i}^{\stable}$ has degree $1$ and
we have  ${\mathsf{D}}_i^{\stable}= {\mathsf{D}}_{g-i}^{\stable}$, for $i=g/2$ the forgetful map $\wt {\mathsf{D}}_{g/2}^{\stable}$
is a degree-$2$ morphism. In this way, we reformulate \eqref{eq:Dstablered} as follows
$${\mathsf{D}}^{\stable}_{\operatorname{red}}=
\frac12 \sum\nolimits_{i=1}^{g-1} {\mathsf{D}}^{\stable}_i.$$
For level curves, consider the stack
$\wt {\mathsf{D}}_{\operatorname{red}}$ classifying  level-$\ell$ curves alongside with a
separating node and a branch of the node. Hence, we get 
the decomposition  of $\wt{\mathsf{D}}_{\operatorname{red}}$ into connected components
and the corresponding decomposition of ${\mathsf{D}}_{\operatorname{red}}$ into irreducible components
$$\wt{\mathsf{D}}_{\operatorname{red}}= \bigsqcup\nolimits_{d_1,d_2,i}
\wt{\mathsf{D}}^{d_1,d_2}_{i}
\qquad \text{and} \qquad
{\mathsf{D}}_{\operatorname{red}}= \frac12 \sum\nolimits_{d_1,d_2,i}
{\mathsf{D}}^{d_1,d_2}_{i},
$$
where $d_1$ and $d_2$ are divisors of $\ell$ whose least common multiple equals $\ell$,
 $i$ ranges between $1$ and $g-1$, and the
loci $\wt{\mathsf{D}}^{d_1,d_2}_{i}$ and ${\mathsf{D}}^{d_1,d_2}_{i}$ are defined as follows.
The stack $\wt{\mathsf{D}}^{d_1,d_2}_{i}$ is the full subcategory of objects
where the data of the chosen branch and of
the genus-$i$ connected component $\sta Z$ of the normalisation of the
separating node satisfy
\begin{enumerate}
\item the branch lies in $\sta Z$ and $g(\sta Z)=i$,
\item the order of $\sta L$ on ${\sta Z}$ equals $d_1$,
\item the order of $\sta L$ on ${\ol{\sta C\setminus \sta Z}}$ equals
$d_2$.
\end{enumerate}
The divisor ${\mathsf{D}}^{d_1,d_2}_{i}$ is the pushforward of the cycle $\wt {\mathsf{D}}^{d_1,d_2}_{i}$
via the forgetful functor forgetting the choice of the branch and of the node.
Since the stack-theoretic structure of
one-nodal level-$\ell$ curves of compact type is trivial, there is no ramification of ${{\sta f}}$ along ${\mathsf{D}}_{\operatorname{red}}$: we have
${\mathsf{D}}_{\operatorname{red}}={{\sta f}}^*{\mathsf{D}}_{\operatorname{red}}^{\stable}.$
The factor $1/2$ in the above 
expression ${\mathsf{D}}_{\operatorname{red}}$
eliminates the factor $2$ due to 
 ${\mathsf{D}}^{d_1,d_2}_{i}=
{\mathsf{D}}^{d_2,d_1}_{g-i}$,
for any $(i,d_1,d_2)\neq(g/2,\ell,\ell)$, 
and to the degree $2$ of the map 
$\wt {\mathsf{D}}^{\ell,\ell}_{\ell/2}\to 
{\mathsf{D}}^{\ell,\ell}_{\ell/2}$ when $g$ is even.

\begin{example}\label{rem:level3reddecomp}
 For $\ell$ prime, we notice that  
 ${\delta}_i$, ${\delta}_{g-i}$ and ${\delta}_{i:g-i}$
are precisely the divisors ${\mathsf{D}}_i^{\ell,1}$, ${\mathsf{D}}_{i}^{1,\ell}$ and
 ${\mathsf{D}}_i^{\ell,\ell}$ for $i\neq g/2$ and $\frac12 
 {\mathsf{D}}_{g/2}^{\ell,\ell}$ otherwise. For $i\neq g/2$,
 they are the three irreducible components of ${{\sta f}}^*\delta_i^{\stable}$ of
degrees
$(\ell^{2i}-1)/\ell$, $(\ell^{2g-2i}-1)/\ell$ and $(\ell^{2g-2i}-1)(\ell^{2i}-1)/\ell$ over $\delta_i^{\stable}$
 (we check that they add up to $\deg({{\sta f}})=(\ell^{2g}-1)/\ell$).
 
\end{example}
\subsubsection{The closure of the locus of irreducible one-nodal curves:  irreducible components} \label{subsubsect:irredlcomp}
We
study the
divisor $\delta_0^{\stable}$ of stable curves with at least one separating node. 
As in \S\ref{subsubsect:seplcomp} we
use the notation ${\mathsf{D}}^{\stable}_{\operatorname{irr}}=\delta_0^{\stable}$ and we analyze
the degree-$2$ morphism
$\wt{\mathsf{D}}_{\operatorname{irr}}^{\stable}
\to {\mathsf{D}}^{\stable}_{\operatorname{irr}}$ classifying stable curves alongside with a
nonseparating node and a branch of the node.
Consider the stack
$\wt {\mathsf{D}}_{\operatorname{irr}}$ classifying  level-$\ell$ curves $\level$ equipped with a prescribed choice of
a nonseparating node and of a branch of such node:
this yields a notation $\sta x$ and $\sta y$ for the points lifting the node
to the normalisation
$\nor \colon \ssC'\to \ssC$ of the nonseparating node.
On $\wt {\mathsf{D}}_{\operatorname{irr}}$,
 we can define the data $M,d,h$
\begin{itemize}
\item[--] the multiplicity $M\in \ZZ/\ell$,
\item[--] the order $d$ (dividing $\ell$ and multiple of $\ell/\gcd(M,\ell)$)
of $\nor^*\ssL$ on $\ssC'$,
\item[--] the  gluing datum of a 
root of unity $h\in \pmb\mu_{\ell/d}$ satisfying $f(\sta x)=h f(\sta y)$
for the sections $f$ of   $(\nor^*\ssL)^{\otimes d}\cong \cO$.
\end{itemize}
Within $\wt {\mathsf{D}}_{\operatorname{irr}}$, 
we write $\wt {\mathsf{D}}^{M,d,h}_{\operatorname{irr}}$ 
for the locus where the multiplicity, the order and the gluing datum are 
respectively $M,d$ and $h$. 
Since $\sta L$ has order $\ell$, within  
$\wt {\mathsf{D}}_{\operatorname{irr}}$,
the gluing datum is always a primitive $\ell/d$th root of unity; 
however, the same definition, without any condition on the order of $\sta L$, 
yields a stack for any 
$h\in \pmmu_{\ell/d}$ and we have 
$\wt {\mathsf{D}}^{M,d,1}_{\operatorname{irr}}\cong 
\wt {\mathsf{D}}^{M,d,h}_{\operatorname{irr}}$, 
via $\sta L\mapsto \sta L\otimes \tau(\zeta)$ 
for $\zeta \in \pmmu_\ell$ with $\zeta^d=h$ (see \S\ref{sssect:taudefn}). 
The moduli stack $\wt {\mathsf{D}}^{M,d,h}_{\operatorname{irr}}$ 
is connected because 
$\wt {\mathsf{D}}^{M,d,1}_{\operatorname{irr}}$ is.
Indeed, following Remark \ref{rem:backtoschemes}, $\wt {\mathsf{D}}^{M,d,1}_{\operatorname{irr}}$ classifies $\pmmu_d$-covers $\pi\colon P\to C$ of genus-$g$ curves with a specified  
nonseparating node $n$ in $C$ corresponding to an orbit of $d/r$ nodes for 
$r=\ell/\gcd(M,\ell)$. We further require 
the following properties: (1) there is a privileged branch at $n$ and the action of 
$\pmmu_r\subseteq \pmmu_d$ 
is of the form  
$\zeta \cdot(z,w)=(\zeta z, \zeta^{-1}w)$ 
on $P$ and is given at the privileged 
branch by the character $m=M/\gcd(M,\ell)\in \ZZ/r$, (2) the normalisation of $C$ at $n$ and of $P$ at 
the $d/r$ points of $\pi^{-1}(n)$
is a connected $\pmmu_d$-cover $\pi\colon P'\to C'$.
The 
connectedness of  $\wt {\mathsf{D}}^{M,d,1}_{\operatorname{irr}}$ follows
precisely  from the 
connectedness of $P'$ and it may be interesting to see it explicitly.
We do it hereafter.

\begin{lemma}\label{pro:connectedness} The moduli stack 
$\wt {\mathsf{D}}^{M,d,1}_{\operatorname{irr}}$ is connected.
\end{lemma}
\begin{proof}
For simplicity, let us 
first consider the case $r=1$ (\emph{i.e.} $M\in \ell\ZZ$). The connected $\pmmu_d$-cover 
$\pi'\colon P'\to C'$ contains two distinguished orbits $D_x$ and $D_y\subset P'$ 
lying above the pre-images $x$ and $y\in C'$ 
of the node $n$. 
The claim follows from the existence of a family ranging through all possible
ways to glue back this normalised 
$\pmmu_d$-cover of $C'$ to form a $\pmmu_d$-cover of $C$ (in general there are $d/r$ 
distinguished possibilities; here we have $d$ choices). 
By deformation, 
it is enough to show the claim 
when $C'$ is  $\PP^1/(0\sim \infty)$ marked at $x$ and $y$
and $P'$ is the connected \'etale $\pmmu_d$-cover attached to 
$\tau(\xi_d)$.    
We take $P'$ itself as a base scheme and we 
define a family of $\pmmu_d$-covers over it. 
Above any point $p$ of $\Omega_{P'}=P'\setminus \Sing\cup D_x$ 
we can consider the cover $P'\to C'$ 
and two distinguished orbits: $D_x$ and the orbit of 
$p$. 
By taking the limit $\pmmu_d$-cover using the 
properness of 
$\ol{\sta R}_{g,d}$ (or simply by blowing up conveniently within 
$\PP^1\times \PP^1$), 
the family extends uniquely across the nodes of $P$ and the points of $D_x$. We 
obtain in this way a family $\mathcal P'$ of $\pmmu_d$-covers over  
the base scheme $P'$ with a section $\delta$ extending the diagonal 
of $(\Omega_{P'})^2$ and disjoint from the closure of 
the orbit $D_x\times \Omega_{P'}$. Fix a point $t\in D_x$; then, 
for any $g\in \pmmu_d$, we 
glue the image of $g\delta$ with the 
(closure of) $gt\times \Omega_{P'}$ and 
we get a family of $\pmmu_d$-covers which 
goes through all the possible ways to glue back $P'\to C'$ to a 
$\pmmu_d$-cover of the initial curve $C$. These are precisely the $d$ fibres above the $d$ points of 
$D_y$.  

If we drop the condition $r=1$, we regard the initial data 
$\pi'\colon P'\to C'$ as the composite of the 
\'etale $\pmmu_d/\pmmu_r$-cover 
$\varepsilon'\colon E'\to C'$ given by   $\tau(\xi_{d/r})$,
and a branched  
$\pmmu_r$-cover with action $m\in \ZZ/r$ 
at the points of  $D_x=(\varepsilon')^{-1}(x)$ and $r-m\in \ZZ/r$ 
at the points of  
$D_y=(\varepsilon')^{-1}(x)$. By construction, 
this branched cover is unique up to isomorphism 
and amounts to  
extracting an 
$r$th root of $\cO_{E'}(-mD_x-(r-m)D_y)$. 
We proceed as above
by defining a family of $\pmmu_d$-covers over the base scheme $E'$. To this effect, if 
$p$ lies in $\Omega_{E'}=E'\setminus \Sing\cup D_x$  we 
consider the $\pmmu_d$-cover of $C'$ 
given by the $\pmmu_r$ cover of $E'$ itself induced by  an $r$th root of $\cO_{E'}(-mD_x-(r-m)\Delta)$,
where $\Delta$ is the orbit of $p$. Again, this family of $\pmmu_d$-covers extends uniquely across 
$\Sing\cup D_x$ and admits two sections with disjoint orbits 
(lifting the closure of 
the diagonal of $(\Omega_{E'})^2$ and the closure of 
a section of $D_x\times \Omega_{E'}\to \Omega_{E'}$). By gluing along these sections as above, we get a family of $\pmmu_d$-covers going 
through all the possible ways to glue back 
$P'\to C'$ to a 
$\pmmu_d$-cover of the initial curve $C$. These are precisely the
$d/r$ fibres above the $d/r$
points of 
$D_y$.  
\end{proof}

Hence, we have decomposed $\wt{\mathsf{D}}_{\operatorname{irr}}$ 
into a disjoint union of $\sum_{M=0}^{\ell-1} \gcd(M,\ell)$ 
connected loci 
$\wt {\mathsf{D}}^{M,d,h}_{\operatorname{irr}}$, where $M\in\{0,\dots, \ell-1\}$,  $h$ is a $\gcd(M,\ell)$th root of unity and 
$d$ is determined by $h$: we set $d=\ell/\ord(h)$ so that $h$ is a primitive 
root in $\pmmu_{\ell/d}$.
We may remove $d$ from the notation and we 
get the desired decomposition into irreducible components of 
${\mathsf{D}}_{\operatorname{irr}}$:
$$\wt{\mathsf{D}}_{\operatorname{irr}}=
\bigsqcup_{\substack{M\in \ZZ/\ell \\ h\in \pmmu_{\gcd(M,\ell)}}}
\wt {\mathsf{D}}^{M,h}_{\operatorname{irr}} \qquad \text{and} \qquad {\mathsf{D}}_{\operatorname{irr}}= \frac12
\sum_{\substack{M\in \ZZ/\ell \\ h\in \pmmu_{\gcd(M,\ell)}}}
{\mathsf{D}}^{M,h}_{\operatorname{irr}}.
$$
Here ${\mathsf{D}}^{M,h}_{\operatorname{irr}}$ are the pushforwards 
in $\ol{\sta R}_{g,\ell}$ of the 
cycles $\wt {\mathsf{D}}^{M,h}_{\operatorname{irr}}$
via the morphism forgetting the prescribed branch. 

Note that, if  $M\in\{0,\ell/2\}$ and $h\in\{1,-1\}$, 
this forgetful morphism is a degree-$2$ morphism and the factor $1/2$ removes the degree factor 
appearing in the direct image of 
$\wt {\mathsf{D}}^{M,h}_{\operatorname{irr}}$. 
Actually, not all combinations with $M=0,\ell/2$ and $h=1,-1$ occur: 
if $\ell$ is odd, only $(M,h)=(0,1)$ occurs; 
if $\ell\in 2\ZZ\setminus 4\ZZ$, all combinations except $(\ell/2,-1)$ occur
($\pmmu_{\ell/2}$ does not contain $-1$); 
if $\ell$ is in $4\ZZ$, any of the four combinations occurs.

In all the remaining cases ${\mathsf{D}}^{M,h}_{\operatorname{irr}}$
equals  ${\mathsf{D}}^{\ell - M,h^{-1}}_{\operatorname{irr}}$.
For these terms, the sum is redundant and the factor $1/2$ removes the 
factor $2$ arising from summing twice the same divisor. 

Notice also that the order of the ramification of the
morphism ${{\sta f}}$  along ${\mathsf{D}}_{\operatorname{irr}}^{M,d,h}$
equals the order of $M$ in $\ZZ/\ell$; that is precisely $r=\ell/\gcd(M,\ell)$ .

\begin{example}\label{rem:level3irrdecom} 
If $\ell=3$, the stack 
$\wt{\mathsf{D}}_{\operatorname{irr}}$ has five connected components, 
as many as \linebreak $\sum_{M=0}^2\gcd(M,3)=3+1+1$. Since $\ell$ is odd
only one of these yields a connected degree-$2$ cover: 
$\wt{\mathsf{D}}^{0,1}_{\operatorname{irr}}
\to {\mathsf{D}}^{0,1}_{\operatorname{irr}}$.
The remaining cases are paired as follows: we have 
$\wt{\mathsf{D}}^{0,\xi_3}_{\operatorname{irr}}=\wt {\mathsf{D}}^{0,\xi^{2}_3}_{\operatorname{irr}}$
and $\wt{\mathsf{D}}^{1,1}_{\operatorname{irr}}=\wt{\mathsf{D}}^{2,1}_{\operatorname{irr}}$.

The divisors $\delta_0'$, $\delta_0''$ and $\delta_0^{\operatorname{ram}}$ 
over $\delta_0^{\stable}$ can be 
recovered as follows:
\begin{itemize} 
\item
the divisor $\delta_0'$ is 
the image in $\ol{\sta R}_{g,\ell}$ of $\wt{\mathsf{D}}^{0,1}_{\operatorname{irr}}$;
\item the divisor $\delta_0''$ is the 
image of 
 $\wt{\mathsf{D}}^{0,\xi_3}_{\operatorname{irr}}$ and it can be also written as 
$\wt {\mathsf{D}}^{0,\xi^{2}_3}_{\operatorname{irr}}$;
\item 
finally, $\delta_0^{\ram}$
is the image of $\wt{\mathsf{D}}^{1,1}_{\operatorname{irr}}$ 
and it can be also written as 
$\wt{\mathsf{D}}^{2,1}_{\operatorname{irr}}$.
\end{itemize}
These divisors
 coincide with the irreducible components of ${\mathsf{D}}_{\operatorname{irr}}$.
 As  substacks over $\delta_0^{\stable}={\mathsf{D}}^{\stable}_{\operatorname{irr}}$ they have respectively degree $1/3$ times
$3 (3^{2g-2}-1)$, $2$,  and $2 (3^{2g-2})$; 
using \eqref{eq:lramification},
 we count the degree of ${\delta}_0^{\ram}$ over ${\delta}_0$ with multiplicity
 $3$ and  we obtain again $\deg({{\sta f}})=(3^{2g}-1)/3$.

In view of the next example and further generalisations, 
we can perfortm this degree check more systematically 
via pushforward via 
$\wt {\sta f}\colon \wt{\mathsf{D}}_{\operatorname{irr}}\to 
\wt{\mathsf{D}}_{\operatorname{irr}}^{\stable}$, 
$\sta i\colon 
\wt {\mathsf{D}}^{\stable}_{\operatorname{irr}}\to \ol{\mathsf{M}}_{g}^{\stable}$,
$\sta j \colon 
\wt {\mathsf{D}}_{\operatorname{irr}}\to \ol{\mathsf{R}}_{g,\ell}^{\stable}$
and 
${\sta f}\colon  \ol{\mathsf{R}}_{g,\ell}\to 
\ol{\mathsf{M}}_{g}$ (the composite of the first two maps equals 
the composite of the last two maps).
Write $\sta f^*\delta_0^{\stable} = \delta_0'+\delta_0''
+3\delta_0^{\ram}$ as 
$$\sta f^*\delta_0^{\stable} =\frac12 \left(
 {\mathsf{D}}^{0,1}_{\operatorname{irr}}+
 {\mathsf{D}}^{0,\xi_3}_{\operatorname{irr}}+
 {\mathsf{D}}^{0,\xi_3^2}_{\operatorname{irr}}+
3 {\mathsf{D}}^{1,1}_{\operatorname{irr}}+
3 {\mathsf{D}}^{2,1}_{\operatorname{irr}}
 \right)$$
and take the pushforward 
$$\sta f_*\sta f^*\delta_0^{\stable}= \frac12 \left(
\sta f_* {\mathsf{D}}^{0,1}_{\operatorname{irr}}+
\sta f_* {\mathsf{D}}^{0,\xi_3}_{\operatorname{irr}}+
\sta f_* {\mathsf{D}}^{0,\xi_3^2}_{\operatorname{irr}}+
3\sta f_* {\mathsf{D}}^{1,1}_{\operatorname{irr}}+
3\sta f_* {\mathsf{D}}^{2,1}_{\operatorname{irr}}
 \right)$$
We write ${\mathsf{D}}^{M,h}_{\operatorname{irr}}=
\sta j_*\wt{\mathsf{D}}^{M,h}_{\operatorname{irr}}$ and 
we replace each term 
$\sta f_*\sta j_*\wt{\mathsf{D}}^{M,h}_{\operatorname{irr}}$ 
by $\sta i_*\sta f_*\wt{\mathsf{D}}^{M,h}_{\operatorname{irr}}=
d^{M,h}
(\sta i_*\wt{\mathsf{D}}^{\stable}_{\operatorname{irr}})$
where $d^{M,h}$ is 
the degree of the forgetful morphism 
$\wt{\mathsf{D}}^{M,h}_{\operatorname{irr}}\to 
\wt{\mathsf{D}}^{\stable}_{\operatorname{irr}}$ onto its image.
We obtain 
\begin{multline*}\sta f_*\sta f^*\delta_0^{\stable}=\left(d^{0,1}+d^{0,\xi_3}+
d^{0,\xi_3^2}+ 
3d^{1,1}+
3d^{2,1}\right) 
\frac12\sta i_* 
\wt{\mathsf{D}}^{\stable}_{\operatorname{irr}}\\
= 
\left(\frac{3(3^{2g-2}-1)}3+ \frac13 + \frac13 + 3\frac{3^{2g-2}}3+ 3\frac{3^{2g-2}}3 \right) 
\delta_0^{\stable}=\frac{3^{2g}-1}3 \delta_0^{\stable}.\end{multline*}
 \end{example}

\begin{example}\label{rem:level4irrdecom} 
If $\ell=4$, the stack 
$\wt{\mathsf{D}}_{\operatorname{irr}}$ has eight connected components, 
as many as \linebreak $\sum_{M=0}^3\gcd(M,4)=4+1+2+1$.
Four of them are paired and yield 
the same boundary divisor: ${\mathsf{D}}^{0 ,\xi_4}_{\operatorname{irr}}=
{\mathsf{D}}^{0 ,\xi_4^3}_{\operatorname{irr}}$
and 
 ${\mathsf{D}}^{1 ,1}_{\operatorname{irr}}=
{\mathsf{D}}^{3 ,1}_{\operatorname{irr}}$.
The remaining four, 
${\mathsf{D}}^{0 ,1}_{\operatorname{irr}},
{\mathsf{D}}^{0 ,\xi_2}_{\operatorname{irr}},
{\mathsf{D}}^{2 ,1}_{\operatorname{irr}}$ and ${\mathsf{D}}^{2 ,\xi_2}_{\operatorname{irr}}$,
yield boundary divisors with multiplicity $2$. 
We can write the fundamental class of the boundary as  
$$
{\mathsf{D}}_{\operatorname{irr}}= \frac12\left(
{\mathsf{D}}^{0 ,1}_{\operatorname{irr}}+
{\mathsf{D}}^{0 ,\xi_2}_{\operatorname{irr}}
+{\mathsf{D}}^{0 ,\xi_4}_{\operatorname{irr}}+
{\mathsf{D}}^{0 ,\xi_4^3}_{\operatorname{irr}}+ 
 {\mathsf{D}}^{1 ,1}_{\operatorname{irr}}+
{\mathsf{D}}^{2 ,1}_{\operatorname{irr}}+{\mathsf{D}}^{2 ,\xi_2}_{\operatorname{irr}}+
{\mathsf{D}}^{3 ,1}_{\operatorname{irr}}\right)$$
or, equivalently, highlighting its five irreducible components, 
 as 
$$
{\mathsf{D}}_{\operatorname{irr}}= 
\frac12{\mathsf{D}}^{0 ,1}_{\operatorname{irr}}+
\frac12{\mathsf{D}}^{0 ,\xi_4^2}_{\operatorname{irr}}
+{\mathsf{D}}^{0 ,\xi_4}_{\operatorname{irr}}+ 
{\mathsf{D}}^{1 ,1}_{\operatorname{irr}}+
\frac12{\mathsf{D}}^{2 ,1}_{\operatorname{irr}}+
\frac12{\mathsf{D}}^{2 ,\xi_2}_{\operatorname{irr}}.$$
By pulling back 
$\delta_0^{\stable}={\mathsf{D}}^{\stable}_{\operatorname{irr}}$
we get the same locus with multiplicities; following 
the last computation given for $\ell=3$ we check that this decomposition is compatible with 
$\deg(\sta f)=\Phi_{2g}(4)/4=(4^{2g}-2^{2g})/4$. We have
$$\sta f^*\delta_0^{\stable}=
\frac12\left(
{\mathsf{D}}^{0 ,1}_{\operatorname{irr}}
+{\mathsf{D}}^{0 ,\xi_4}_{\operatorname{irr}}+
{\mathsf{D}}^{0 ,\xi_4^2}_{\operatorname{irr}}+
{\mathsf{D}}^{0 ,\xi_4^3}_{\operatorname{irr}}+ 
 4 {\mathsf{D}}^{1 ,1}_{\operatorname{irr}}+
2 {\mathsf{D}}^{2 ,1}_{\operatorname{irr}}+2 {\mathsf{D}}^{2 ,\xi_2}_{\operatorname{irr}}+
4{\mathsf{D}}^{3 ,1}_{\operatorname{irr}}\right),
$$
hence
\begin{multline*}\sta f_*\sta f^*\delta_0^{\stable}=
\left(d^{0,1}+d^{0,\xi_4}+
d^{0,\xi_4^2}+ 
d^{0,\xi_4^3}+
4d^{1,1}+
2d^{2,1}+
2d^{2,\xi_2}+
4d^{3,1}
\right) \frac12\sta i_* 
\wt{\mathsf{D}}^{\stable}_{\operatorname{irr}}\\
= 
\left(\frac{4\Phi_{2g-2}(4)}4+\frac14+ \frac{2\Phi_{2g-2}(2)}4
 + \frac14 + 4\frac{4^{2g-2}}4+ 2\frac{2\Phi_{2g-2}(4)}4
+ 2\frac{2^{2g-2}}4 
+ 4\frac{4^{2g-2}}4\right) 
\delta_0^{\stable}\\=\frac{\Phi_{2g}(4)}4 \delta_0^{\stable},
\end{multline*}
where we used $\Phi_{n}(4)=4^n-2^n$ and $\Phi_{n}(2)=2^n-1$.

 \end{example}
 
 The combinatorics involved in the 
 previous examples is subsumed in the general 
 treatment of \S\ref{rem:CCC}, where we provide the computation of 
 the length of any fibre of moduli of level curves.
 
 \begin{remark}[Compatibility with the terminology of \cite{CEFS}]
We have the following 
relations between the coarse decomposition in terms of 
$\delta$-divisors and the finer analysis in terms of 
$\sta D$-divisors.  
We have 
\begin{align*}
\delta_{i}=\sta D_{i}^{1,\ell}=\sta D_{g-i}^{\ell,1}, \qquad
&\delta_{g-i}=\sta D_i^{\ell,1}=\sta D_{i}^{\ell,1},\qquad
&&\delta_{i:g-i}=\sum\nolimits_{\substack{d_1,d_2 \mid \ell \\
\lcm(d_1,d_2)=\ell}}
\sta D_i^{d_1,d_2} 
\ \ (i\neq g/2),\\
\delta_0'=
\frac12\sum_{h\in \pmmu_\ell} \sta D_{\operatorname{irr}}^{0,h},
\qquad 
&\delta_0''=\frac12\sta D_{\operatorname{irr}}^{0,1},
\qquad 
&&\delta_0^{\ram}=\frac12
\sum\nolimits_{\substack{0\neq 
M\in \ZZ/\ell \\ h\in \pmmu_{\gcd(M,\ell)}}} \sta D_{\operatorname{irr}}^{M,h}
\end{align*}
(for $g\in 2\ZZ$ we have 
$\delta_{\frac g2:\frac g2}=
\frac 12 \sum_{d_1,d_2 \mid \ell}\sta D_{g/2}^{d_1,d_2}$ where again
we impose $\lcm(d_1,d_2)=\ell$).

Since the ramification index at $\sta D_{\operatorname{irr}}^{M,h}$ 
equals $r(M)$, the equation ${{\sta f}}^{*}(\delta^{\stable}_0)=
\delta_{0}' + \delta_{0}'' + \ell \delta_{0}^{\ram}$ only holds 
for $\ell$ prime. 
However, 
we point out that the same equation holds if we replace 
$\delta_0^{\ram}$ by
$\sum_{a=1}^{\lfloor \ell/2\rfloor} \delta_0^{(a)}$ 
and we set for any $a=1,\dots, \lfloor \ell/2\rfloor$ 
$$\delta_0^{(a)}=
\frac 12  
 \sum\nolimits_{\substack{M=a, -a\\ 
 h\in \pmmu_{\gcd(M,\ell)}}}  \frac1{\gcd(a,\ell)} 
 \sta D_{\operatorname{irr}}^{M,h}.$$
When $\ell$ is prime, which is the main focus in \cite{CEFS}, 
the divisor above arises naturally as a substack 
within $\ol{\sta R}_{g,\ell}$.
For composite values of $\ell$, 
the above divisor can be still 
obtained as a codimension-1 substack of 
a suitable  
compactification of $\sta R_{g,\ell}$; 
indeed  in \cite[\S 1.3]{CEFS} 
we illustrate how $\delta_0^{(a)}$ the multiplicities 
$1/{\gcd(a,\ell)}$ arise naturally when working with the 
compactification of \cite{Cstab} 
which simply imposes stabilisers of order $\ell$ at all nonseparating nodes 
instead of imposing the faithfulness condition on $\sta L$.
These compactifications have the same coarse space $\ol {\mathcal R}_{g,\ell}$, 
however,  they are not
very convenient for the study of the singularities
of $\ol{\mathcal R}_{g,\ell}$ because their stabilisers 
are extensions by quasireflections of the 
stabilisers of $\ol{\sta R}_{g,\ell}$. 
\end{remark}

\section{The singularities of the moduli space of level curves}\label{sect:sing}
In this section we assume $g\ge 4$; this is a
standard condition in the study of the singularity locus
of the coarse moduli space of curves essentially motivated by Harris and Mumford's work \cite{HM}
(see Remark \ref{rem:faithful} and
Proposition \ref{pro:quasireflections} and also the role played by this condition in the proof of Theorem \ref{thm:no_junior_ghosts}).

At the point represented by $\level$, the local pictures of ${\ol {\sta R}}_{g,\ell}$ and of $\rr_{g,\ell}$
are given by $[\defo\level/\aut\level]$ and
$\defo\level/\aut\level$. We relate these local pictures to 
$[\defo(C)/\aut(C)]$ and $\defo(C)/\aut(C)$, the local pictures of 
$\ol {\sta{M}}_g$ and $\ol M_g$ at  $C$.

\subsection{Deformation spaces and automorphism groups}\label{subsect:automlevel}
%
The space $\defo\level$
can be expressed in terms of $\defo(C)$.
\subsubsection{Deformations of $C$.}\label{subsubsect:defocurve}
We only consider the stable curve $C$. 
We denote by $\defo(C, \Sing(C))$ 
the space of deformations of the curve $C$ alongside with 
its set of nodes $\Sing(C)$.
It may be decomposed canonically as 
$$\defo(C,\Sing(C))=\bigoplus\nolimits_{v\in V}H^1(C_v',T(-D_v)),$$
where, by $C_v'\subseteq C$, we denote the
connected component of the normalisation
of $C$ attached to $v$, and, by $D_v$, we denote the divisor formed by the
inverse images of the nodes of $C$ under the normalisation map. Indeed,
the group $H^1(C_v',T(-D_v))$ parameterises deformations of the
pair $(C_v',D_v)$. 

Note that $\defo(C,\Sing(C))$ is a subspace of $\defo(C)$; by modding it out 
we obtain 
$$\defo(C)/\defo(C,\Sing(C))= \bigoplus\nolimits_{e\in E} N_e,$$
where the decomposition is canonical and the term
$N_e$ denotes the fibre over $[C]$ of 
the normal bundle to the locus of 
deformations preserving the node attached to $e$. In fact 
$N_e$ is one dimensional; giving a (non canonical) parametrisation
$$N_e\cong \Spec(\CC[t_e])=:\mathbb A^1_{t_e}$$
is equivalent to choosing a smoothing\footnote{A smoothing of a node $n\in C$ is an infinitesimal 
deformation $\mathcal C\to \Spec\CC[t_e]/(t_e^2)$ of the curve $C$, 
where $n$ is a regular point within the scheme $\mathcal C$.}
 of 
the node attached to $e$ 
along $t_e$.

\subsubsection{Deformations of $\level$.}\label{subsubsect:defolevelcurve}
 The deformation space 
 $\defo\level$ is canonically identified with 
 $\defo(\sta C)$ via the 
 \'etale   forgetful functor $\level\mapsto \sta C$. 
 The picture of $\defo(\sta C)$ 
 is analogue to the above picture for $\defo(C)$.

Within $\defo\level=\defo(\sta C)$, 
we consider   
$\defo(\leveli,\levelii,\leveliii,\Sing(\leveli))=
\defo(\leveli,\Sing(\leveli))$, the subspace 
of deformations where 
$\Sing(\sta C)$ deforms alongside $\sta C$ 
(the topological type of the curve is preserved).  
In fact, 
via the natural forgetful map $\defo \level=
\defo(\sta C)\to \defo(C)$, 
this space is canonically identified to $\defo(C,\Sing(C))$
(this happens because $H^1(C_v',T(-D_v))$
and the stack-theoretic counterpart are canonically isomorphic, 
\cite[Lem.~2.3.4]{AV}). Therefore, we have
\begin{equation}\label{eq:defolevel}
\defo(\leveli,\levelii,\leveliii,\Sing(\leveli))=
\defo(\leveli,\Sing(\leveli))=
\bigoplus\nolimits_{v\in V}H^1(C_v',T(-D_v)).                                  \end{equation}
The corresponding  
quotient space canonically decomposes as 
\begin{equation}\label{eq:defomodtopolpres}
\defo\level/\defo(\leveli,\levelii,\leveliii,\Sing(\leveli))=
\defo(\sta C)/\defo(\sta C, \Sing(\leveli))
= \bigoplus\nolimits_{e\in E} K_e.\end{equation}
As in \S\ref{subsubsect:defocurve}, 
$K_e$ is one-dimensional; indeed, it can be parametrized by 
$\tau_e$, the $r(e)$th root  of 
the above mentioned parameter $t_e$ 
($r(e)$ is the local index from 
\S\ref{subsubsect:explicitlevel_twisted_indices}). In this way 
$\tau_e$ may be geometrically interpreted as 
the parameter smoothing the node of $\sta C$ corresponding to $e$ and
the map between quotients
$\defo\level/\defo(\leveli,\levelii,\leveliii,\Sing(\leveli))\rightarrow
\defo(C)/\defo(C,\Sing(C))$
is the direct sum, for $e$ in $E$, of 
\begin{align*} 
\mathbb A^1_{t_e}&\to \mathbb A^1_{\tau_e}\\
t_e&\mapsto t_e^{r(e)}.
\end{align*}

\subsubsection{Automorphisms of $\level$.} An automorphism of a level curve $\level$  is given by
$(\iso,\rho)$ where $\iso$ is an isomorphism of
$\leveli$, and  $\rho$ is an isomorphism of
line bundles $\iso^*\levelii\rightarrow\levelii$ satisfying $\phi\circ \rho^{\otimes \ell}=\iso^*\leveliii$
$$
\xymatrix@C=2.2pc{
\sta s^* \sta (\levelii^{\otimes \ell}) \ar[r]^= \ar[d]^{\sta \iso^*\leveliii}&
(\sta s^* \sta \levelii)^{\otimes \ell} \ar[r]^{\ \ \ \rho^{\otimes r}} &
\levelii^{\otimes \ell}\ar[d]^{\leveliii}\\
\sta s^* \mathcal O\ar[rr]^=& &  \mathcal O.} $$

We write
$${\aut}\level=\{(\iso,\rho)\mid \iso \in \aut(\ssC),\  \rho\colon \iso^*\ssL\xrightarrow{\cong} \ssL,
\ \phi\circ \rho^{\otimes \ell}=\iso^*\rho\}.$$
On the other hand, we consider
$$\ul{\aut}\level=\{\iso\in \aut(\ssC)\mid \iso^*\ssL\cong \ssL\}.$$
It is easy to see that for each element $\iso\in \ul\aut\level$
there exists $(\iso,\rho)\in \aut\level$.
Two pairs of this form differ by
a power of a quasitrivial automorphism of the form
$(\id_{\ssC},\xi_\ell)$ operating by scaling the fibres. We have the following exact sequence
$$1\to \pmb\mu_\ell\to \aut\level \to \ul{\aut}\level\to 1.$$
As already mentioned, quasitrivial isomorphisms act trivially on $\defo\level$.
Therefore, it is natural to study the action of $\aut\level$ on $\defo\level$ by focusing on
$\ul{\aut}\level=\aut\level/\pmb\mu_\ell$.

The coarsening $\iso\mapsto \coaiso$,
induces a group
homomorphism
$$\mathfrak{coarse}\colon \ul\aut\level\to\aut(C).$$
The kernel and the image are natural geometric objects of
independent interest. We denote them by $\ul\aut_C\level$ and $\aut'(C)$ and we refer to them as
the group of ghost automorphisms and the group of
automorphisms of $C$ lifting to $\level$
\begin{equation}\label{eq:ghostsandlifts}
1\to \ul\aut_C\level \to \ul\aut\level\to \aut'(C) \to 1.                                                           \end{equation}

\subsubsection{Ghosts automorphisms.}\label{subsubsect:ghostautom}
The kernel of $\mathfrak{coarse}$ is the group   of ghosts automorphisms:
automorphisms $\iso$ of $\ssC$ fixing
at the same time the underlying curve $C$ and
the isomorphism class of the overlying line bundle
$\ssL$; we write
$$\ul{\aut}_C\level:=\ker(\mathfrak{coarse}).$$
It is worth pointing out that an automorphism of a stack $\sta{X}$ may well be nontrivial
and, at the same time, operate as the identity on the coarse space $X$. In our case, stabilisers are
isolated and we may treat this issue locally. Consider $\sta{U}=[\{xy=0\}/\pmb\mu_r]$ the quotient stack
where $\xi_r$
acts on $(x,y)$ as $(\xi_r x,\xi_r^{-1}y)$. All automorphisms $(x,y)\mapsto (\xi_r^b x,\xi_r^{a}y)$
induce the identity on the quotient space. The automorphisms
fixing the coarsening $U$
up to natural transformations (the
$2$-isomorphisms $(x,y)\mapsto (\xi_r^i x,\xi_r^{-i}y)$)
form a group $\aut_U(\sta{U})\cong\pmb\mu_r$ generated by $(x,y)\mapsto(\xi_r x,y)$. In this way,
the automorphisms of a twisted curve $\ssC$ with order-$r$ stabilisers at $k$ nodes
which fix $C$ are freely generated by $k$ automorphisms each one operating as
$(x,y)\mapsto(\xi_r x,y)$ at a node, \cite[\S7]{ACV}. Note that no branch has been privileged: via
the natural transformation $(x,y)\mapsto (\xi_r x,\xi_r^{-1}y)$,
the automorphism $(x,y)\mapsto(\xi_r x,y)$ is $2$-isomorphic to $(x,y)\mapsto (x, \xi_r y)$.

This explains the canonical identification from 
\cite[\S7, Prop.~7.1.1]{ACV}
\begin{equation}\label{eq:ACVidentif}
\aut_C(\ssC)= \bigoplus\nolimits_{e\in E}\pmb\mu_{r(e)}.
\end{equation}
We notice that, all throughout the paper, 
we adopt for clarity the additive notation for sums and direct sums 
(\emph{e.g.} we write $\bigoplus_{e\in E} \pmmu_{r(e)}$, 
$(\pmmu_\ell)^{\oplus \#V}$, and, 
where sums over a set of indices $I$ are 
not direct, we use the symbol $\sum_{i\in I}$)

The summand labelled by $e$ on the right hand side 
corresponds to 
$\aut_{\sta C\setminus \{{\sta n}_e\}}(\sta C)$, the subgroup
of automorphisms of $\sta C$ operating 
as the identity off the node $\sta n_e$ attached to $e$.
The action 
of $\aut_{\sta C\setminus \{{\sta n}_e\}}(\sta C)$ on 
$\defo(\sta C)/\defo(\sta C,\Sing(\sta C))=
\bigoplus_{e\in E} K_e$ (see in \eqref{subsubsect:defolevelcurve})
coincides with the natural action of $\pmmu_{r(e)}$ on the 
one-dimensional term $K_e$: 
the character $1$ in $\mathrm{Hom}(\pmmu_{r(e)},\GG_m)=\ZZ/r(e)$.

\subsubsection{Automorphisms of $C$ lifting to $\level$.} The image of $\ul{\aut}\level$ via 
$\mathfrak{coarse}$ is the
group of automorphisms $s$ of $C$, which can be obtained as the coarsening of a morphism
$\iso$ of $\ssC$ satisfying $\iso^*\ssL\cong \ssL$. Clearly, this group differs in general from $\aut(C)$;
notice for instance that
automorphisms of the coarse curve $C$ that do not preserve the order of the overlying stabiliser
of $\ssC$ cannot be lifted to $\ssC$. More precisely we have the obvious inclusion
$$\aut'(C):=\im (\mathfrak{coarse})\subseteq\{s\in \aut(C)\mid s_\Gamma^*M=M\}$$
where $s_\Gamma$ is the dual graph automorphism induced by $s$.
The condition $s_\Gamma^*M=M$ is restrictive in general (it is not, of course, when $M$ vanishes), but it does not guarantee the
existence of an automorphism $\iso$ lifting $s$. For a simple
counterexample, consider a point of the divisor $\Delta_{g/2}$ from \S\ref{subsect:boundary} lying over
the isomorphism class in $\Delta^{\stable}_{g/2}$ of two isomorphic  $1$-pointed genus-$g/2$ curves meeting transversely
at their marked point; here the involution of the underlying stable curve
respects the multiplicity cochain, but does not lift to the level structure.
We also point out that in general, even when a lift $\iso$ exists, there may well be no canonical choice for $\iso$.
Lifting a morphism that maps a $\sta B\pmb\mu_k$-node to another $\sta B\pmb\mu_k$-node amounts
to extracting a $k$th root of the identifications between
local parameters on both branches (there may be no distinguished choice, although all choices
can be identified via
a ghost isomorphism, up to natural transformation).

\begin{example} \label{exa:stackytails}
We conclude this subsection with the study of automorphisms of the genus-one curve $\sta E=[\wt E/\pmmu_\ell]$,
stack quotient of a nodal cubic $\PP^1/(0\sim \infty)$, from Example
\ref{exa:level_on_elltails}.
Although the group of automorphisms of
$\sta E$ and of $E=\wt E/\pmmu_\ell$ is not finite ($E$ is not stable),
the study of this case is relevant to the study of  level curves over a stable curve
containing, as a subcurve, a copy of $\sta E$ meeting the rest of the curve at one
separating node $n$ (the orbit $\pmmu_\ell\cdot 1$) with
trivial stabiliser by Proposition \ref{pro:Misclosed}.
To this effect, it is crucial to study the finite group of automorphisms
of $\sta E$ that fix $n$
$$\aut(\sta E,n)=\{\sta s\in \aut(E)\mid \sta s(n)=n\}.$$
The exact sequence $1\to \aut_E(\sta E,n)\to \aut(\sta E,n)\to \aut(E,n)$ reads
$$1\to \pmmu_\ell \to \aut(\sta E, n) \xrightarrow{\ \mathfrak{coarse} \ }  \pmmu_2.$$
After choosing $\xi_\ell$, $\pmmu_\ell$ 
is generated by the automorphism $\sta g$
with coarsening 
$g=\id$ and local picture $(x,y)\mapsto (\xi_\ell x,y)$ at the node.
On the other hand, $\pmmu_2$ is generated 
by the unique involution $i$
fixing $n$ and the node, and interchanging the branches at the node.
In this special case, $\mathfrak{coarse}$
is surjective and   the involution $i$
admits a distinguished  lift $\sta i\in \aut[\wt E/\pmmu_\ell]$ as follows.
At the level of $\wt E$, consider the unique involution of
$\wt E$ fixing the node of $\wt E$ and the point $1$ and
exchanging the branches of the node.
At the level of the group $\pmmu_\ell$,
consider the passage to the inverse.
We obtain $\sta i\colon [\wt E/\pmmu_\ell]\to [\wt E/\pmmu_\ell]$ and we  have the short exact sequence\footnote{One 
can observe explicitly that $\aut(\sta E, n)$ 
is the direct product $\pmmu_\ell\times \pmmu_2$; \emph{i.e.}
the involution $\sta i$ commutes with the ghost $\sta g$ defined
locally at the node as $(x,y)\mapsto (\xi_\ell x,y)$. We only need 
to check 
$\sta g\circ \sta i=\sta i \circ \sta g$ 
at a local picture $[\{xy=0\}/\pmmu_\ell]$ 
at the node of $[\wt E/\pmmu_\ell]$. There,
the morphism $\sta i$ may be described as the map interchanging the branches $(x,y)\mapsto (y,x)$ and
$\sta i\circ \sta g\colon (x,y)\mapsto (\xi_{\ell} y,x)$
equals 
$\sta g\circ \sta i\colon (x,y)\mapsto (y,\xi_{\ell} x)$ up to the natural transformation $(x,y)\mapsto (\xi_{\ell} x,\xi_{\ell}^{-1}y)$.}
$$0\to \pmmu_\ell\to \aut(\sta E,n) \to \pmmu_2\to 0.$$

We now set $\ell=2$ and 
consider the automorphisms of an explicitly defined level-$2$ curve.
 Let $\ssC$ be a twisted curve,
union of $\sta E=[\wt E/\pmmu_2]$ with a smooth $(g-1)$-curve $X$ with
$\aut(X)=\{\id_X\}$. The curves $\sta E$ and $X$  meet
transversely at $n$ and the coarse spaces form a
genus-$g$ stable curve $C$.
Hence, by construction, the above short exact sequence reads $0\to \aut_C(\ssC)\to \aut(\ssC)\to \aut(C)\to 0.$
Let $(\sta C=X\cup \sta E, \ssL=\cO\cup(\ssL_{\ram}\otimes \ssL_{\etale})  ,\iso)$ be the unique level-$2$ curve
obtained by glueing over $n$
the fibre of $\cO_X$ and that of $\ssL_{\ram}\otimes \ssL_{\etale}$
from Example \ref{exa:level_on_elltails}.
By construction $\sta i^*$ operates trivially on both $\sta L_{\ram}$ and $\sta L_{\etale}$; therefore in this example
$\ul \aut(\ssC,\ssL,\iso)\to \aut(C)$ is surjective and $\aut'(C)=\aut(C)$.
On the other hand $\sta g^*$ acts trivially on $\ssL_{\etale}$ but nontrivially on $\ssL_{\ram}$
$$\sta g^*\ssL_{\ram}=\ssL_{\ram}\otimes \ssL_{\etale}$$
(this relation
can be shown directly, but we refer to \eqref{eq:twisters} for a
general rule). 
Notice that, in fact, there is a 
second level-$2$ 
curve $(\sta{C}, \sta{L}_0=\cO \cup \ssL_{\ram} , \sta s_0)$ 
which is isomorphic to $(\leveli,\levelii, \sta s)$ via $\sta g^*$, but 
$\sta L_0 \not \cong \ssL$.

We deduce that $\ul{\aut}_C(\leveli,\levelii, \sta s)$, 
in the example $(\leveli,\levelii, \sta s)$ given above, 
is trivial: there are no ghost automorphisms.
This is a consequence of the more general No-Ghosts Lemma \ref{lem:noqrefl_ghosts}.
The sequence \eqref{eq:ghostsandlifts} reads
$0\to 0\to \pmmu_2\to \pmmu_2\to 0$ and 
$\ul{\aut}(\leveli,\levelii, \sta s)=\pmmu_2$
operates nontrivially only on the parameter $\tau_n=t_n$ appearing in \eqref{eq:defolevel}
and corresponding to the family smoothing the node $n$
(the local picture is $\tau_n\mapsto -\tau_n$ because  $\sta i$ operates trivially on the $y$-branch lying on $X$ and operates
by a change of sign on the $x$-branch lying on the component $\sta E$, and $\tau_n$ equals $xy$).
In other words $\sta i$ fixes a hyperplane of 
$\defo(\leveli,\levelii, \sta s)$; \emph{i.e.}
$\sta i$ is a quasireflection.
\end{example}

\subsection{Dual graph and ghost automorphisms when the level is prime}\label{subsect:primel}
Only for this section the index $\ell$ is assumed to be prime.
Ghost automorphisms of the level curve $\level$ can be described in terms
of the dual graph $\Gamma$ of $\ssC$.
\subsubsection{Setup.}
Consider the characteristic function $\nu=\nu_M$ 
of the support of the multiplicity $M$ of $\level$ and 
the corresponding contraction $\Gamma\to \Gamma(\nu)$ (the condition $\nu>0$, or $\nu=\infty$, 
holds if and only if $M=0$ and singles out contracted edges, see
\eqref{eq:single-valued_characteristic}).
Recall $(C_\nu^\bullet(\Gamma; \pmmu_\ell), \delta)$
\begin{eqnarray}\label{eq:0cochainsupported}
&C^0_{\nu}(\Gamma; \pmmu_\ell)=
\{a\colon V\to \pmmu_\ell\mid a(e_+)=a(e_-) \text{ if $\nu(e)>0$} \},\\
\label{eq:1cochainsupported}
&C^1_{\nu}(\Gamma; \pmmu_\ell)=
\{b\colon \mbb E\to \pmmu_\ell \mid b(\ol e)=b(e)^{-1}, \text{ and } b(e)=1 \text{ if } \nu(e)>0 \}.\end{eqnarray}
By \eqref{eq:imcap} we have the  following identification via $\mathcal B$
$$\im\left(\delta\colon C^0(\Gamma(\nu);\pmmu_\ell)\to C^1(\Gamma(\nu);\pmmu_\ell)\right)\cong C^1_{\nu}(\Gamma;\pmmu_\ell)\cap \im \delta.$$

\subsubsection{Automorphisms of $\ssC$ via $\Gamma$ and $\nu$.}\label{subsubsect:autoACV} It is natural to define the 
group of symmetric 
$\pmmu_\ell$-valued functions vanishing
on the set of edges with zero multiplicity
$$S_{\nu}(\Gamma; \pmmu_\ell)=\{b\colon \mbb E\to \pmmu_\ell\mid b(\ol e)=b(e), \text{ and } b(e)=1 \text{ for } \nu(e)>0 \},$$
canonically isomorphic to $\bigoplus_{e\mid \nu(e)>0} \pmmu_\ell$.
As mentioned in \eqref{eq:ACVidentif}, the group $\aut_C(\ssC)$ is easy to describe by {\cite[\S7]{ACV}}.
For $\ell$ prime,  there is a canonical isomorphism
\begin{equation}\label{eqn:ACVcanoniciso}
\aut_C(\ssC)= S_{\nu}(\Gamma; \pmmu_\ell),\end{equation}
where
$\sta{a}\colon e\mapsto \sta a(e)\in \pmmu_\ell$ 
corresponds to $\sta a\in \aut_C(\ssC)$
acting at the node attached to $e\in E$ as
\begin{equation}\label{eq:localghostpicture}(x,y)\mapsto ({\sta a}(e)
x,y)\equiv(x,{{\sta a}(e)}y).\end{equation}

\subsubsection{Ghost automorphisms via $\Gamma$ and $\nu$.} We characterise ghost 
automorphisms of the level structure $\level$. To begin, 
recall that $\ell$ is prime, $M$ 
is the multiplicity of 
$\level$, and $\nu$ is equal to $\infty$ where $M$ vanishes and to $0$ elsewhere (see \eqref{eq:single-valued_characteristic}).
Via $\ZZ/\ell=\mathrm{Hom}(\pmmu_\ell,\mathbb G_m)$, 
we have the product
\begin{align}\label{defn:aM}
S_{\nu}(\Gamma; \pmmu_\ell)\times C_\nu^1(\Gamma; \ZZ/\ell)&
\rightarrow C_{\nu}^1(\Gamma; \pmmu_\ell);& ({\sta a},f)&
\mapsto \sta{a}\odot f:=f(\sta a).
\end{align}
Indeed, since $\sta a$ takes values in $\pmmu_\ell$ and 
$f$ in 
$\ZZ/\ell=\mathrm{Hom}(\pmmu_\ell,\mathbb G_m)$,
we express the result of the action of 
the automorphism $\sta a$ on $f$ by $\sta a\odot f := f(\sta a)$,
\emph{i.e.} by the evaluation at each edge $e$ of the homomorphism 
$f(e)$ at $\sta a(e)$.
This could be stated more explicitly: within $\pmmu_\ell$ we 
have 
$(\sta a \odot f)(e) = \sta a(e)^{f(e)}$. 
The notation $\sta{a}\odot f$ 
emphasises that $\sta a$ operates on $f$ and becomes 
convenient once we fix isomorphisms $\pmmu_r \cong \ZZ/r$ 
in the last 
part of the paper, see Assumption \ref{assu:roots}; 
then, $\sta{a}\odot f$ is actually a product in $\ZZ/\ell$
(see \eqref{eq:odotprod}).

Since 
$M$ lies in $C^1_\nu(\Gamma;\ZZ/\ell)$
we get the isomorphisms
\begin{equation}\label{eq:M-M-1lprime} 
M\colon S_{\nu}(\Gamma; \pmmu_\ell)
\rightarrow C_{\nu}^1(\Gamma; \pmmu_\ell)\qquad \text{and} 
\qquad
M^{-1}\colon C_{\nu}^1(\Gamma; \pmmu_\ell)
\rightarrow S_{\nu}(\Gamma; \pmmu_\ell)\end{equation} 
mapping ${\sta a}\in S_{\nu}(\Gamma; \pmmu_\ell)$ to 
$\sta{a}\odot M$, and, 
conversely, 
the $1$-cochain $b\colon e\mapsto b(e)$ of
$C_{\nu}(\Gamma; \pmmu_\ell)$
to the symmetric function $\sta a=M^{-1}b$
\begin{equation}\label{eq:M-1explicitlprime}
\sta a\colon e \mapsto \begin{cases} [M(e)^{-1}]_\ell( b(e))={\sta a}(e) &\text{ for $M(e)\neq 0$,}\\
                          1 & \text{ if $M(e)=0,$}
                         \end{cases}
                         \end{equation}
      (where $[M(e)^{-1}]_\ell$ is the inverse of $M(e)$ in $\ZZ/\ell$
      and is regarded as an invertible homomorphism applied 
      to $b(e)\in \pmmu_\ell$; again, this turns into a product 
      under Assumption \ref{assu:roots}).

Now, for any $\sta a\in \aut_C(\ssC)=S_{\nu}(\Gamma; \ZZ/\ell)$, we have (see \cite[Prop.~2.18]{Cstab})
\begin{equation}\label{eq:twisters}
\sta a^*\ssL\cong \ssL \otimes \tau(\sta{a}\odot M),
\end{equation}
where $\tau$ is the homomorphism 
defined in \S\ref{sssect:taudefn} associating to a 
$\pmmu_\ell$-valued $1$-cochain the  line bundle
with the corresponding descent data.
For completeness, we recall here the argument proving the above identity.
Let us write $\{xy=0\}$ for the local picture at a chosen node attached to the oriented edge
$e$ (as already observed
the choice of the notation $(x,y)$ yields $e\in \mbb E$).
Then, consider the pullback via the automorphism
$\sta a\colon (x,y)\mapsto (\xi_\ell x,y)$ of the line bundle
$\sta L$ defined by the action $\xi_\ell\cdot (x,y,t)=(\xi_\ell x,\xi_\ell^{-1}y,\xi_\ell t)$
on $\{xy=0\}\times \mathbb A^1$ locally at the chosen node
and trivial elsewhere. This definition of $\sta L$ makes sense because the
quotient is canonically trivialised off the node by the invariant sections
$xt^{-1}$ on one branch and by
$yt$ on the other branch. Pulling back via $\sta a$ changes the trivialisation
only at one branch; in other words, by \eqref{eq:Picsequence}, it is equivalent to tensoring by $ \tau(\sta{a}\odot M)$.

The above statement implies (via \eqref{eq:Picsequence}) that $\sta a$ is a ghost if
and only if $\sta{a}\odot M$ lies in $\ker \tau=\im \delta$. This completely justifies the following notation.
\begin{definition}\label{defn:ghostgroup}
Set
$G_{\nu}(\Gamma;\pmmu_\ell)=C^1_{\nu}(\Gamma;\pmmu_\ell)\cap \im \delta. $
\end{definition}
\begin{remark}\label{rem:ghostgroup}
Via the contraction  $\Gamma\to \Gamma(\nu)$ and \eqref{eq:imcap},
we get the alternative presentation
\begin{equation}\label{eq:GMwithGammaM}
G_{\nu}(\Gamma;\pmmu_\ell)=\im\left(\delta\colon C^0(\Gamma(\nu);\pmmu_\ell)\to C^1(\Gamma(\nu);\pmmu_\ell)\right)
\end{equation}
yielding the isomorphism
$G_{\nu}(\Gamma;\ZZ/\ell)=(\pmmu_\ell)^{\oplus (\# V(\nu)-1)}$.
\end{remark}
\begin{proposition}\label{pro:ghosts}
For $\ell$ prime, let $\level$ be a  level-$\ell$ curve.
We have a canonical identification
$$\ul\aut_C\level\cong G_{\nu}(\Gamma;\pmmu_\ell).$$
A $1$-cochain $b\colon e\mapsto b(e)$ of
$G_{\nu}(\Gamma; \ZZ/\ell)$
corresponds,
to the symmetric function $M^{-1}b$.
%
\qed \end{proposition}

\begin{remark}\label{rem:fixingirredcomp_prime}
As an easy consequence of the above analysis
a ghost automorphisms $\sta a \in  \ul\aut_C\level$
fixes every irreducible component $\sta Z\subseteq \ssC$.
Indeed, the restriction of
$\sta a$ may operate nontrivially only at the nodes of $\sta Z$. These are represented by loops in the dual graph.
Indeed $G_\nu(\Gamma;\pmmu_\ell)$ is not supported on the loops (cuts are supported off the loops)
\end{remark}

\begin{example}\label{exa:circuit}
Let us assume $\ell=3$. Consider the case where the dual graph is formed by a
single circuit $K$ consisting of $n$ edges.
In this case $\ker\partial\cong\ZZ/3=\langle K\rangle$.
There are two possibilities: $M=0$, where 
$\ul{\aut}_C\level=1$, and $M\neq 0$, where 
$C^1_{\nu}(\Gamma;\pmmu_3)=C^1(\Gamma;\pmmu_3)$
and the group of ghosts 
$G_\nu(\Gamma;\pmmu_\ell)$ is isomorphic to 
$\im \delta\cong (\pmmu_\ell)^{\oplus (\#V-1)}=
(\pmmu_3)^{\oplus n-1}$.
The elements of $\ul{\aut}_C\level$ are the
functions $\sta a\in S_{\nu}(\Gamma;\pmmu_3)$
such that $\sta{a}\odot M(K)=1$, see
\eqref{eq:imandcirc}.
\begin{enumerate}
 \item[(i)] \label{constanta} Assume $n=3$.
In this case $M$ lies in $\im \delta$ and
we get an element of $G_{\nu}(\Gamma;\pmmu_3)$ by taking
$\sta a\colon \mbb E\to \pmmu_3$ constant.
In order 
to fix ideas let us fix a primitive third root of unity $\xi_3$ and 
set $\sta a$ constant and equal to $\xi_3$.
Then $\sta a$ is a ghost 
operating as $(x,y)\mapsto (\xi_3 x,y)$ at all nodes
and acting on $\defo\level$ as $(\xi_3\mathbb I_3)\oplus \id$
(see \eqref{eq:defolevel}).
This argument holds in general whenever $M$ is in $\im\delta_{\ZZ/\ell}$;
then, for $a(e)=\zeta$ all $e$, we have 
$\sta a\odot M(e)=\zeta^{M(e)}$ all $e$, 
and, by \eqref{eq:imandcirc}, 
the $\ZZ/\ell$-valued $1$-cochain $M\in \im\delta_{\ZZ/\ell}$
yields 
a $\pmmu_\ell$-valued 
$1$-cochain $\sta{a}\odot M\in \im\delta_{\pmmu_\ell}$.

\item[(ii)] Assume $n=2$, let $e_1$ and $e_2$ be the two edges.
Here $M=0$ or $M\not\in\im\delta$.
Again by choosing $\xi_3$, we 
define a symmetric function $\sta a\colon \mbb E\to \ZZ/\ell$ 
mapping one edge to $\xi_3$
($e_1, \ol e_1\mapsto \xi_3$)
and  the other to its inverse $\xi_3^2$ 
($e_2, \ol e_2\mapsto \xi_3^2$); then $\sta{a}\odot M$ is a cut, lies in $\im \delta$ 
and
acts
on  \eqref{eq:defolevel} as $\diag(\xi_3,\xi_3^2)\oplus \id$.
\item[(iii)] If the circuit has a single edge, then $\im\delta=(0)$. There are no nontrivial ghosts.
\end{enumerate}
\end{example}

\begin{example} The argument at point (iii) 
shows that the level structures 
$\cO\cup (\ssL_{\ram}\otimes \ssL_{\etale})$ 
and $\cO\cup \ssL_{\ram}$ introduced in Example \ref{exa:stackytails}
have no nontrivial ghosts. Indeed, the dual graph in that case has two vertices $v_X$ and $v_{\sta E}$
corresponding to $X$ and $\sta E$, one edge $e_n$ connecting them and corresponding to
the node $n$ and a second edge $e_{\looop}$ with both ends on $v_{\sta E}$. The multiplicity is
supported on this last vertex, and the vertex set $V(\nu)$ of the
graph $\Gamma(\nu)$ obtained by contracting
all edges with vanishing multiplicity  reduces to  a single vertex.
We have $\im(\delta\colon C^0(\Gamma(\nu);\pmmu_\ell)\to C^1(\Gamma(\nu);\pmmu_\ell))\cong (\pmmu_\ell)^{\oplus \#V(\nu)-1}=1.$
Notice that this argument holds for any tree-like graph (that is a
graph that becomes a tree once the loops are removed), see Corollary \ref{cor:noghosts}.
\end{example}

\begin{example} \label{exa:vine} Consider a dual graph with two vertices $v_1,v_2$ and three
edges, each of them linking the two  vertices to each other. As a multiplicity cochain
we choose $e\mapsto M(e)$ equal to $1$ on the oriented edges of $\mbb E$ oriented from $v_1$ to $v_2$.
For $\ell=3$,
it is easy to check
that $M$ belongs to the kernel and is indeed the
sum of two different
two-edged circuits. The cochain $M$ lies
also in $\im\delta$ (it is  the
$\pmmu_3$-valued cut attached to the proper nonempty subset  $H=\{v_1\}$).
Therefore a constant 
$\sta a\equiv \zeta\in \pmmu_3 \in S_{\nu}(\Gamma;\pmmu_3)$
satisfies $\sta{a}\odot M\in \im \delta$ and 
acts on $\defo\level$ as $\zeta\mathbb I_3\oplus \id$. (See also 
Example 
\ref{exa:circuit},(i).)
\end{example}

\subsection{The singular points of the moduli space}
Notice that, in all the above examples of ghost automorphisms $\sta g\in \ul \aut\level$,
the fixed space $\{v\in \defo(\level)\mid \sta g\cdot v=v\}$ is
never a hyperplane. An automorphism of an affine space whose fixed space
coincides with an hyperplane is called a \emph{quasireflection}.
A general property of nontrivial ghosts is that they never act as quasireflections. Let us recall that this is crucial for classifying singularities.

\begin{fact} \label{fact:singularities}
The scheme-theoretic quotient
$\defo\level/\aut\level$ is smooth
if and only if $\aut\level$ is spanned by elements acting as the identity or as
quasireflections (see \cite{Prill}).
\end{fact}

\subsubsection{Nontrivial ghosts are not quasireflections.}
Here is a consequence of Proposition \ref{pro:ghosts}.
\begin{lemma}\label{lem:noqrefl_ghosts}
If $\sta a\in \ul \aut_C \level$ fixes a hyperplane of 
$\defo\level$, then $\sta a=\id_{\ssC}$.
\end{lemma}
As we 
argue in Remark \ref{rem:generalisations}, 
this lemma generalises word by word and 
straightforwardly to the case when $\ell$ is composite;
so, we did not impose the condition on $\ell$ to be 
prime in the statement. 

\begin{proof}
Let $b$ be a nontrivial ghost $G_\nu(\Gamma;\pmmu_\ell)$; \emph{i.e.}
a $1$-cochain $b\in C_{\nu}^1(\Gamma; \pmmu_\ell)$ lying in $\im \delta$.
Then 
there exists a (nonseparating) edge $e$ with $\nu(e)\neq \infty$ 
and $b(e)\neq 1$. In this case,
there is a circuit $K$ passing through $e$. 
Now, $K$ satisfies 
$b(K)=1$  by \eqref{eq:imandcirc}.
Hence, the support of $b$ contains an oriented edge $e'$ which differs from $e$ regardless of its orientation.
Proposition \ref{pro:ghosts} claims that 
the unique automorphism $\sta a$ such that $\sta a \odot M=b$ 
acts nontrivially on $\mathbb A^1_{\tau_e}$ and $\mathbb A^1_{\tau_{e'}}$.
\end{proof}

\begin{remark}[$\ul{\aut}\level$ operates faithfully on $\defo\level$]\label{rem:faithful}
Under the assumption $g\ge 4$,
any nontrivial automorphism $a\in \aut(C)$
acts nontrivially on $\defo(C)$, see \cite{HM}.
Then, the faithfulness of $\ul{\aut}\level$ follows from that of $\ul \aut_C\level$ and from
the above lemma.
\end{remark}

\subsubsection{Elliptic tail involutions.}
In \cite[Thm.2, \S2]{HM}, Harris and Mumford prove
that an automorphism $a\in \aut(C)$
is a quasireflection of $\defo(C)$
if and only if $a$ is an \emph{elliptic tail involution} (ETI):
the curve $C$ contains a genus-$1$
subcurve $E$ meeting the rest of the curve at a single point $n$ and
$a$ is the identity on $\ol{C\setminus E}$
and is the nontrivial canonical involution $i$ of $\aut(E,n)$. This involution
is canonically identified both if $E$ is elliptic or rational: it is the hyperelliptic involution in the first case
whereas, in the second case, it is the unique involution
fixing the point $n$ and the node of $E\cong \PP^1/(0\equiv\infty)$ and interchanging the branches of such node.
We need to generalise to twisted curves the notion of ETI.
Because all separating nodes of  level-$\ell$ curves have trivial stabilisers,  a
genus-$1$ subcurve $\sta E$ meeting the rest of the curve at a single point $n$
is either a scheme $(E,n)$ or is isomorphic to the pointed
stack-theoretic curve $(\sta E,n)$
of Example \ref{exa:level_on_elltails}.
In both cases, these tails are
equipped with a canonical involution  $\sta i$.
\begin{definition}[elliptic tail and ETI]\label{defn:elltail}
Let $\level$ be a level-$\ell$ twisted curve. An elliptic tail
is a genus-one subcurve $\sta E$ meeting the rest of the curve $\ssC$ at a single point.
An elliptic tail involution of $\level$ is
an automorphism of $\ssC$ such that the restriction to $\ol {\ssC\setminus\sta E}$ is the identity
and the restriction to $\sta E$
equals the canonical involution $\sta i$ and satisfies $\sta i^*(\sta L\!\mid_{\sta E})=\sta L\!\mid_{\sta E}.$
\end{definition}

\begin{proposition} \label{pro:quasireflections}
Consider a stable genus-$g$ level-$\ell$ curve with $g\ge 4$.
An automorphism $\sta s\in \ul{\aut}\level$ acts as a quasireflection on $\defo\level$ if and only if it is an
ETI.
\end{proposition}

As we 
argue in Remark \ref{rem:generalisations}, 
this proposition also 
generalises immediately to the case when $\ell$ is composite.

\begin{proof}
Let $\iso$ be an automorphism
of $\level$ acting as a quasireflection on $\defo\level$.
Then, its coarsening $s$ acts either as the identity or as a quasireflection
on $\defo(C)$. We rule out $s=\id_C$: in this
case $\iso$ would be a ghost, and, by Lemma \ref{lem:noqrefl_ghosts}, there is no
ghost acting as quasireflection. Then, by \cite{HM}, 
$s$ operates as an ETI on $C$. If
the elliptic tail is represented by a scheme,
then $\sta s$ is an ETI (using Lemma \ref{lem:noqrefl_ghosts} on $\ol{C\setminus \sta E}$).
Otherwise, the elliptic tail  is the curve $\sta E$
of Example \ref{exa:level_on_elltails} and we need
to check that $\sta i$ is the only automorphism lifting the ETI $i$
and operating as a quasireflection on $\defo\level$. By \eqref{eq:ghostsandlifts}
the remaining automorphisms are of the form
$\sta i\circ \sta g^{n}$ with $\sta g ^n\neq \id$ (using the notation of Example \ref{exa:level_on_elltails});
due to Proposition \ref{pro:ghosts}, the automorphism
$\sta g^n$ acts nontrivially on $\defo(\ssC)$ and
$\sta i\circ \sta g^n$ is not a quasireflection. \end{proof}

\subsubsection{No-Ghosts.}
By Remark \ref{rem:circuits} and  Proposition \ref{pro:ghosts},
$\ul\aut_C\level$ is trivial if and only if the multiplicity graph $\Gamma({\nu})$
has only one vertex. We call such graphs \emph{bouquets}.
\begin{corollary}\label{cor:noghosts}
Let $\ell$ be prime. The group of ghost automorphisms $\ul\aut_C\level$ is trivial if and only if $\Gamma(\nu)$ is a bouquet. \qed
\end{corollary}

%

Combining Corollary \ref{cor:noghosts} and Proposition \ref{pro:quasireflections}
we get the following result.

\begin{theorem}\label{thm:smooth}
Let $\ell$ be prime and assume $g\ge 4$.
The following conditions are equivalent.
\begin{enumerate}
\item The point of $\rr_{g,\ell}$ representing $\level$ is smooth.
\item The group $\ul\aut\level$ is spanned by ETIs of $\sta C$.
 \item The graph $\Gamma(\nu)$ is a bouquet and $\aut'(C)$ is spanned by ETIs of $C$.
\label{secondcond}
\end{enumerate}
 \end{theorem}

\begin{proof}
The point representing $\level$ is smooth if and only $\ul \aut\level$
is generated by elements operating on
$\defo\level$ as the identity or as quasireflections.
Nontrivial elements of $\ul\aut\level$ never operate as the identity, see Remark \ref{rem:faithful}.
By Proposition \ref{pro:quasireflections} elements operating as quasireflections are
precisely the ETIs of $\ssC$; hence (i) $\Leftrightarrow$ (ii). Now (iii) implies $\ul{\aut}_C\level=1$ and $\ul{\aut}\level
=\aut'(C)$ generated by ETIs
of $C$; we deduce (ii) because the ETIs generating $\aut'(C)$ lift canonically
to ETIs generating $\aut\level$. Conversely,
(ii) holds only if there are no nontrivial ghosts ($\ul{\aut}_C\level=1$) because any nontrivial composition
of ETIs has a nontrivial coarsening. Hence, $\Gamma(\nu)$ is a bouquet,
 $\ul{\aut}\level=\aut'(C)$, 
 and the coarsening of the ETIs spanning
$\ul{\aut}\level$ are ETIs spanning $\aut'(C)$.
\end{proof}

\subsection{Generalisation to the case of level curves of 
composite level}\label{subsect:compositel}
The generalisation of the above statement requires
a modification of the condition ``$\Gamma(\nu)$ is a bouquet'' in part (iii); we introduce a
new set of contractions. We are grateful to Roland Bacher for several ideas that helped
us a great deal in finding
the correct setup for this section.

\subsubsection{The truncated valuation of $\ZZ/p^n$.} \label{subsubsect:truncatedval}
For any prime $p$ we recall that the  ring $\ZZ/p^n$
is a truncated valuation ring in the sense of
\cite[\S1.1]{Deligne}. We recall the definition, which applies to any
local ring $R$ whose maximal ideal
$\mathfrak m$ is generated by a nilpotent element.
We set the valuation
$\val_{\mathfrak m} \colon R \rightarrow \ZZ\cup \{\infty\}$,
$x \mapsto    \sup\{i\mid x\in {\mathfrak m}^i\}$,
taking values in $\ZZ\cap [0,\length(R)-1]$
on $R\setminus \{0\}$ and satisfying $\val_{\mathfrak m}(0)=\infty$ (if $R=\ZZ/p^n$, then $\mathfrak m=(p)$ and $\length(R)=n$).

\subsubsection{The vector-valued function $\pmb \nu$.}
Consider the prime factorisation of $\ell$
$$\ell=\prod_{p\mid \ell} p^{e_p},$$
where $e_p$ is the $p$-adic valuation of $\ell$.
Then, the following vector-valued function $\pmb \nu_M$, or simply $\pmb \nu$,
encodes the truncated valuations of $M(e)\ \modulo \ p^{e_p}$ in $\ZZ/p^{e_p}$ for all $p\mid \ell$.
\begin{equation}
\label{eq:vector-valued_characteristic}
e\mapsto \pmb \nu(e)=\big(\nu_p(e)\big)_{p\mid \ell} \qquad
\text{where} \quad  \nu_p(e):=\val_{(p)}\big(M(e)\ \modulo\ p^{e_p}\big).
\end{equation}
 Notice that, when $\ell$ is prime, we recover the characteristic 
 function $\nu_M$  of the support of $M$   $$\val_{(p)}(M(e))=\nu_M(e).$$
\subsubsection{Contractions.}\label{subsect:contractions} For each $p\mid \ell$, the coordinate $\nu_p$ of
$\pmb \nu=(\nu_p)_{p\mid \ell}$
yields a filtration
\begin{equation*}\label{eq:filtration}
\varnothing \subseteq \{\nu_{p}\ge e_p\}_E\subseteq \{\nu_{p}\ge e_p-1\}_E \subseteq \ldots\subseteq \{ \nu_{p}\ge k\}_E
 \subseteq \ldots\subseteq\{\nu_{p}\ge 1\}_E
 \subseteq\{\nu_{p}\ge 0\}_E=E.                                             \end{equation*}
To each of the 
above edge subsets 
we can naturally associate a subgraph (the 
vertex set is formed by the heads and the tails of the chosen edges):
\begin{equation}\label{eq:graphfiltration}
\varnothing \subseteq \Delta(\nu_{p}^{e_p})
\subseteq \Delta(\nu_{p}^{e_p-1}) \subseteq
\dots \subseteq \Delta(\nu_{p}^{k})
\subseteq \dots \subseteq \Delta(\nu_{p}^{1}) \subseteq \Delta(\nu_{p}^{0})=\Gamma.
\end{equation}
 The respective contractions $\Gamma({\nu_p^k})$
of $\{ \nu_{p}\ge k\}_E$ fit in the  sequence of contractions
\begin{equation}\label{eq:graphscontraction}
\Gamma
 \longrightarrow\Gamma({\nu_{p}^{e_p}})
 \longrightarrow \Gamma({\nu_{p}^{e_p-1}}) \longrightarrow\ldots\longrightarrow
\Gamma({\nu_{p}^k}) \longrightarrow \ldots
\longrightarrow \Gamma({\nu_{p}^{1}})
 \longrightarrow \Gamma({\nu_{p}^{0}}),                                                                                \end{equation}
where the graph $\Gamma({\nu_{p}^{0}})$ is the null graph ($\Gamma$ is connected). 
The sets of vertices $V(\nu_p^k)$ fit in 
\begin{equation}\label{eq:verticessetscontraction}
V
 \twoheadrightarrow  V({\nu_{p}^{e_p}})
 \twoheadrightarrow V({\nu_{p}^{e_p-1}}) \twoheadrightarrow\ldots\twoheadrightarrow
V({\nu_{p}^k}) \twoheadrightarrow \ldots  \twoheadrightarrow V({\nu_{p}^{1}})
\twoheadrightarrow V({\nu_{p}^{0}})=\{\bullet\}.                                                                                        \end{equation}
The sets of edges $E(\nu_p^k)$ are related by the reversed inclusions
\begin{equation}\label{eq:edgessetscontraction}
E\supseteq E({\nu_{p}^{e_p}})
 \supseteq E({\nu_{p}^{e_p-1}}) \supseteq \ldots\supseteq
E({\nu_{p}^k}) \supseteq \ldots \supseteq E({\nu_{p}^{1}}) \supseteq E({\nu_{p}^{0}})=\varnothing.                                                                                        \end{equation}
In the introduction, for brevity, we used the notation $\Delta_p^k$ and $\Gamma_p^k$ the graphs $\Delta(\nu_p^k)$ and $\Gamma(\nu_p^k)$.
Contracting  $\{ \nu_{p}\ge k\}_E$ makes sense for any $k$ in
$\ZZ \cup \{\infty\}$; for $k\ge e_p$ we get
$\Gamma(\nu_p^k)=\Gamma({\nu_{p}^{e_p}})$,
for $k\le 0$ we get the null graph $\Gamma(\nu_p^k)=\Gamma({\nu_{p}^{0}})$.
For $k\in\{0,\dots,e_p\}$, the following holds.
\begin{definition}[the graph $\Gamma(\nu_p^k)$]\label{defn:contractpq}
For $p$ prime dividing $\ell$ and $k\in\{0,\dots,e_p\}$,
the map $\Gamma\to \Gamma(\nu_p^k)$ is given by contracting
the edges $e$ for which $p^k$ divides $M(e)\in \ZZ/p^{e_p}$.
\end{definition}

\subsubsection{The subcomplex $C_{\pmb \nu}^\bullet(\Gamma;\pmmu_\ell)$.}
Let us point out that, {for } $$\ol \nu_p(e)=\min(e_p,\nu_p(e)),$$ 
we have
$\gcd(M(e),\ell)= \prod\nolimits_{p\mid \ell} p^{\ol\nu_p(e)}(=\ell/r(e))$. We systematically use the canonical morphisms
 \begin{equation}\label{eq:injectioninZl}
\bigoplus\nolimits_{p\mid \ell} {} \pmmu_{p^{e_p- \ol\nu_p(e)}} \underset{\subseteq }{\xrightarrow{\ \qquad \ }}
  {\bigoplus}_{p\mid \ell} 
\ \pmmu_{p^{e_p}}= \pmmu_\ell,                                                                                                                                 \end{equation}
where  the term on the left hand side 
may be regarded, 
via a canonical identification, as  
$\pmmu_{r(e)}$. 

 Now, we generalise
the above mentioned subcomplex $C_{\nu}^\bullet(\Gamma;\pmmu_\ell)$ (see  \eqref{eq:0cochainsupported} and \eqref{eq:1cochainsupported}). Set
\begin{align*}
C_{\pmb \nu}^0(\Gamma;\pmmu_\ell)&=
\left\{a\colon V\to \pmmu_\ell\mid a(e_+)(a(e_-))^{-1}\in \bigoplus\nolimits_{p\mid \ell} \pmmu_{p^{e_p- \ol\nu_p(e)}}  =
\pmmu_{r(e)}\right \}\\
C_{\pmb \nu}^1(\Gamma;\pmmu_\ell)&=\left\{b\colon \mathbb E\to \pmmu_\ell\mid
b(\ol e)=b(e)^{-1} \  \text{ and }  \ b(e)\in
\bigoplus\nolimits_{p\mid \ell}  \pmmu_{p^{e_p- \ol\nu_p(e)}}= 
\pmmu_{r(e)}\right\}.
\end{align*}
By restricting $\delta$ we get the differential 
$$C_{\pmb \nu}^0(\Gamma;\pmmu_\ell)\xrightarrow{\ \delta(\pmb \nu)\ } C_{\pmb \nu}^1(\Gamma;\pmmu_\ell);$$
and $(C_{\pmb \nu}^\bullet(\Gamma;\pmmu_\ell), \delta)$ is a subcomplex of $(C^\bullet(\Gamma;\pmmu_\ell), \delta)$.
Definition \ref{defn:ghostgroup} extends word for word.
\begin{definition}\label{defn:ghostgroupl}
 Set $G_{\pmb \nu}(\Gamma;\pmmu_\ell)=C^1_{\pmb \nu}(\Gamma;\pmmu_\ell)\cap \im \delta$.
\end{definition}
By construction $G_{\pmb \nu}(\Gamma;\pmmu_\ell)$ equals 
$\im(\delta(\pmb \nu))$ via the inclusion
$C_{\pmb \nu}^1(\Gamma;\pmmu_\ell)\subseteq C^1(\Gamma;\pmmu_\ell)$.
The following theorem
proves that, with this setup,
 $G_{\pmb \nu}(\Gamma;\pmmu_\ell)$
is again isomorphic to the group of ghost automorphisms. First, we introduce 
the generalised group of symmetric functions
$$S_{\pmb \nu}(\Gamma;\pmmu_\ell)=\left\{b\colon \mathbb E\to \pmmu_\ell\mid b(\ol e)=b(e) \ \text{ and } \ b(e)\in
\bigoplus\nolimits_{p\mid \ell}  \pmmu_{p^{e_p-\ol \nu_p(e)}}=\pmmu_{r(e)}\right\}.$$
Via $\ZZ/r(e) =\mathrm{Hom}(\pmmu_{r(e)},\mathbb G_m)$, 
we have the product
\begin{align*}
S_{\pmb \nu}(\Gamma; \pmmu_\ell)\times C_{\pmb\nu}^1(\Gamma; \ZZ/\ell)&
\rightarrow C_{\pmb \nu}^1(\Gamma; \pmmu_\ell);& ({\sta a},f)&
\mapsto \sta{a}\odot f:=f({\sta a}).
\end{align*}

Again, since 
$M$ lies by construction in $C^1_{\pmb \nu}(\Gamma;\ZZ/\ell)$
we get the isomorphisms
\begin{equation}\label{eq:MandM-1}
M\colon S_{\pmb \nu}(\Gamma; \pmmu_\ell)
\rightarrow C_{\pmb \nu}^1(\Gamma; \pmmu_\ell)\qquad \text{and} 
\qquad
M^{-1}\colon C_{\pmb \nu}^1(\Gamma; \pmmu_\ell)
\rightarrow S_{\pmb \nu}(\Gamma; \pmmu_\ell).
\end{equation}
Here $M$ maps the symmetric function ${\sta a}\colon e\mapsto 
\sta a(e)$ to 
the $1$-cochain $\sta{a}\odot M$ given by applying at each edge $e$
the homomorphism 
$m(e)\in \ZZ/r(e) =\mathrm{Hom}(\pmmu_{r(e)},\mathbb G_m)$ 
to the $r(e)$th root $\sta a(e)$. Conversely  
$M^{-1}$ maps the $1$-cochain $b\colon e\mapsto b(e)$ of
$C_{\pmb \nu}(\Gamma; \pmmu_\ell)$
to the symmetric function $M^{-1}b=\sta a$, 
defined as
\begin{equation}\label{eq:M-1explicit}
M^{-1}b\colon e \mapsto \begin{cases} [m(e)^{-1}]_{r(e)} ( b(e)) &\text{ for $M(e)\neq 0$,}\\
                          1 & \text{ if $M(e)=0,$}
                         \end{cases}\end{equation}
(where $[m(e)^{-1}]_{r(e)}$ is the inverse of $m(e)$ in $\ZZ/r(e)$).

\begin{theorem} \label{thm:ghosts_compositel} 
Let $\level$ be a  level curve of level $\ell\in \mathbb N^{\times}$;
write $M$ for its multiplicity and $\pmb \nu$
for the corresponding vector-valued function \eqref{eq:vector-valued_characteristic}.
We have the following statements.
\begin{enumerate}\item 
There is a canonical isomorphism
${\aut}\level=S_{\pmb \nu}(\Gamma; \pmmu_\ell).$ The above local
description 
of  $\sta a\in S_{\pmb \nu}(\Gamma; \pmmu_\ell)$ holds without changes
if we write $\sta a$ as a $\pmmu_\ell$-valued function.
\item Let $\sta a \in S_{\pmb \nu}(\Gamma; \pmmu_\ell)$;
then, we have (using \eqref{eq:twisters})
$\sta a^*\ssL=\ssL\otimes \tau(\sta{a}\odot M).$
\item We have \label{eq:multinr(e)}
$$\ul{\aut}_C\level\cong G_{\pmb \nu}(\Gamma; \pmmu_\ell).$$
The $1$-cochain $b\in G_{\pmb \nu}(\Gamma; \pmmu_\ell)\subset
C_{\pmb \nu}^1(\Gamma; \pmmu_\ell)$
identifies the ghost automorphism $\sta a$ 
corresponding to the symmetric function $M^{-1}b$ explicitly 
defined above.
\end{enumerate}
\end{theorem}
\begin{proof} 
Since $S_{\pmb \nu}(\Gamma; \ZZ/\ell)={\bigoplus}_{e\in E}\pmmu_{r(e)}$,
we recover the group $\aut_C(\ssC)$ of \cite[\S7, Prop.~7.1.1]{ACV}.
Point (ii) yields (iii) immediately and is a direct consequence of \cite[Prop.~2.18]{Cstab} as before.
\end{proof}

\begin{remark}
If we work with a fixed primitive 
$\ell$th root of unity; then, giving 
${\sta a}(e)\in \pmmu_{r(e)}$
amounts to specifying a multiple 
$\wt{\sta a}(e)$ of $\frac{\ell}{r(e)}$ modulo $\ell$ 
and the ghost $\sta a$ operates on
$\defo\level$ as
$$\left\{\bigoplus\nolimits_e \left(\tau_e\mapsto 
\xi_\ell^{{\wt{\sta a}(e)}}\tau_e\right)\right\}\oplus \id.$$
This follows from the analysis of the 
quotient 
$\defo\level/\defo(\leveli,\levelii,\leveliii,\Sing(\leveli))$ 
carried out in \S\ref{subsubsect:defolevelcurve} and 
 of the action of the ghosts on it, \ref{subsubsect:ghostautom}.
\end{remark}	

\begin{remark}\label{rem:fixingirredcompl}
Every ghost 
restricts to the identity on the irreducible components of $\ssC$.
\end{remark}

\subsubsection{Computing the group $G_{\pmb \nu}(\Gamma;\pmmu_\ell)$.}
When $\ell$ is a prime integer, 
the group of ghosts is a free $\pmmu_\ell$-module and
Remark \ref{rem:ghostgroup} allowed us to compute its
rank over $\pmmu_\ell$: the number of vertices of the contracted graph minus $1$.
In general, when $\ell$ is composite, the group of ghosts is not free on $\pmmu_\ell$.
By generalising Remark \ref{rem:ghostgroup},
we provide an explicit formula
for its elementary divisors.

\begin{remark}
 Once an orientation
 $E\to \mbb E$ is specified, $C_{\pmb \nu}^1(\Gamma;\pmmu_\ell)$ 
 may be written as
 $${\bigoplus}_{{{e\in E}}} \bigoplus\nolimits_{p\mid \ell}  
 \pmmu_{p^{e_p- \ol \nu_p(e)}}.$$
We may invert the order of the direct sums and rewrite the summands as usual 
$$\bigoplus\nolimits_{p\mid \ell}\left({\bigoplus}_{{{e\in E}}} \   \pmmu_{p^{e_p}- p^{\ol \nu_p(e)}}\right).$$  
The  summands of the first direct sum (over the prime divisors $p$) 
equal the (non-direct) sums of subgroups
$$\sum\nolimits_{1\le k\le e_p}\mathcal B 
C^1(\Gamma(\nu_p^k);\pmmu_{p^{e_p-k+1}})\subseteq 
C^1(\Gamma;\pmmu_\ell).$$
We deduce from this characterisation the following identity
which does not involve any fixed orientation $E\to \mbb E$ and holds both for $1$-cochains and
for $0$-cochains. We have
\begin{equation}\label{eq:complexexpression}
C_{\pmb \nu}^i(\Gamma;\pmmu_\ell)=
{\bigoplus}_{p\mid \ell}
\sum\nolimits_{1\le k\le e_p} \mathcal B C^i(\Gamma(\nu_p^{k});
\pmmu_{p^{e_p-k+1}}) \qquad i=0,1.\end{equation}
Moreover, we immediately get an explicit computation of the groups $C_{\pmb \nu}^1(\Gamma;\pmmu_\ell)$: because
$\mathcal B C^0(\Gamma(\nu);\pmmu_{p^{h}})\cong (\pmmu_{p^{h}})^{\oplus
 \# V(\nu)}$ and
$\mathcal B C^1(\Gamma(\nu);\pmmu_{p^{h}})\cong (\pmmu_{p^{h}})^{\oplus \#E(\nu)}$,
we have
$$C^i_{\pmb \nu}(\Gamma;\pmmu_\ell)\cong 
{\bigoplus}_{p\mid \ell} {\bigoplus}_{k=1}^{e_p} 
(\pmmu_{p^{k}})^{\oplus \eta^i(\nu_p^k)},$$
where, using the Kronecker delta, we can compute $\eta^i(\nu_p^k)$ from \eqref{eq:verticessetscontraction} and \eqref{eq:edgessetscontraction}
$$\eta^i(\nu_p^k):=\begin{cases}  \# V(\nu_p^{e_p-k+1}) -\delta_{k,e_p}\# V(\nu_p^{e_p-k})  & i=0, \\
\# E(\nu_p^{e_p-k+1}) -\delta_{k,e_p}\# E(\nu_p^{e_p-k}) & i=1.
                                                                                                    \end{cases}
$$
\end{remark}
The following lemma, embodying 
the corollary stated in the introduction, follows.
\begin{lemma}\label{tsohg}
We have
 \begin{equation*}\label{eq:GpmbnuwithGammanu}
 G_{\pmb \nu}(\Gamma;\pmmu_\ell)={\bigoplus}_{p\mid \ell} \sum\nolimits_{1\le k\le e_p}  \im\delta(\nu_p^k),
 \end{equation*}
where $\delta(\nu_p^k)$ is $C^0(\Gamma(\nu_p^{k});\pmmu_{p^{e_p-k+1}})
\to C^1(\Gamma(\nu_p^{k});\pmmu_{p^{e_p-k+1}}).$
More explicitly, set
$$\al_p^k:=\# V(\nu_p^{e_p-k+1}) -\# V(\nu_p^{e_p-k});$$
then, the group $G_{\pmb \nu}(\Gamma;\pmmu_\ell)$ decomposes as 
$$G_{\pmb \nu}(\Gamma;\ZZ/\ell)\cong {\bigoplus}_{p\mid \ell} {\bigoplus}_{k=1}^{e_p} (\pmmu_{p^{k}})^{\oplus \al_p^k} $$
and has order $\frac1\ell \prod_{p\mid \ell} p^{\#V_p}$ 
for $V_p=\sqcup_{k=1}^{e_p} V(\nu_p^k)$.
 \qed
\end{lemma}
\begin{remark} For $\ell$ prime, we recover \eqref{eq:imiso}: 
\eqref{eq:verticessetscontraction} reads $V(\nu)\to \{\bullet\}$,
we get $\al_p^1=\#V(\nu)-1$.
(Note that Kronecker delta does not occur in the formula for the elementary divisors
$\al_p^k$.)
\end{remark}
\begin{example} \label{exa:ghostgroup}
We consider the dual graph $\Gamma$ in Figure \ref{fig:dualgraphcrible} of a level-$8$ curve. The multiplicities
assigned to each oriented edge define a cocycle $M\in \ker\partial$.
Here, $2$ is the only prime divisor of $\ell$.
In Figure \ref{fig:truncatedval}, we write next to each edge $e$ the value of $\nu_2(e)$.
  Then, in Figure \ref{fig:graphscontraction}, we show the corresponding contractions. We observe
that, in this case, at each step the number of vertices decreases by $1$.
Therefore, by Lemma \ref{tsohg}, we compute $\al_2^k=1$ for $k=1,2,3$. We finally obtain $64$ ghosts
$$G_{\pmb \nu}(\Gamma;\pmmu_8)\cong \pmmu_2 \oplus \pmmu_4\oplus \pmmu_8$$
that can be spanned by ghosts of order $2$, $4$ and $8$ corresponding to the
$\pmmu_8$-valued symmetric functions in $S_{\pmb \nu}(\Gamma,\pmmu_8)$ displayed in Fig. \ref{fig:generators} 
(in order to simplify the notation 
we specify $8$th roots of unity
with respect to a chosen primitive root $\xi_8$
of unity: we write 
integers mod $8$ next to each edge). 
We check the corollary stated in the 
introduction: there are $9$ vertices in $\Gamma(\nu_2^3), \Gamma(\nu_2^2),$
and $\Gamma(\nu_2^1)$ and there 
are $2^{9}/8$ (\emph{i.e.} $64$) ghosts.
\end{example}

  \begin{figure}
\xymatrix@=3.2pc{
&&&                          &                                                          &   *{\bullet} \ar@{-}[dl]_{\underset{\nearrow}{1}} \ar@{-}[dr]^{\underset{\searrow}{1}} &  \\
& &&*{\bullet}  \ar@{-}@/_/[r]_{\underset{\rightarrow}{0}} & *{\bullet}  \ar@{-} @(dr,dl)[rdru] _{  \underset{\rightarrow}{1}}  \ar@{-}@/_/[l]_{\underset{\leftarrow}{0}}
\ar@{-}[r] _{  \ \ \underset{\leftarrow}{6}} &   *{\bullet}\ar@{-}@/_/[l]_{\underset{\leftarrow}{4}}
\ar@{-}[r]_{\underset{\leftarrow}{2}\ \ }  &   *{\bullet}  & 
}
 \caption{A dual graph $\Gamma$ of a level-$8$ curve with multiplicities} \label{fig:dualgraphcrible}
  \end{figure}
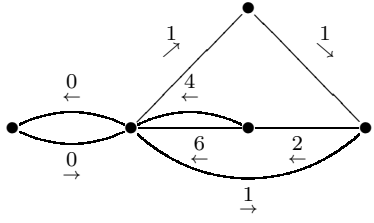

\begin{figure}
\xymatrix@=3.2pc{
&&&                          &                                                          &   *{\bullet} \ar@{-}[dl]_{{0}} \ar@{-}[dr]^{{0}} &  \\
&&&*{\bullet}  \ar@{-}@/_/[r]_{{\infty}} & *{\bullet}  \ar@{-} @(dr,dl)[rdru]_{0}\ar@{-}@/_/[l]_{{\infty}}   \ar@{-}[r] _{ {1}} &   *{\bullet}\ar@{-}@/_/[l]_{{2}}
\ar@{-}[r]_{{1}}  &   *{\bullet}  & 
}
%
 \caption{The truncated valuation associated to $(2)$ in $\ZZ/8$.} \label{fig:truncatedval}
  \end{figure}
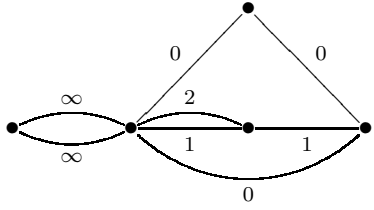

  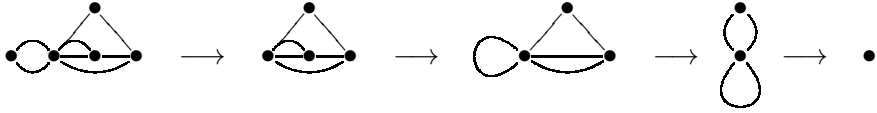
\begin{figure}
\xymatrix@=.8pc{
&&&                           &                                              &   *{\bullet} \ar@{-}[dl] \ar@{-}[dr]  &                &  &                       &   *{\bullet}\ar@{-}[dl]\ar@{-}[dr]      &       &                  &     &      & *{\bullet}\ar@{-}[dl]\ar@{-}[dr]      &      & & *{\bullet}\ar@{-}@/_/[d] \ar@{-}@/^/[d] \\
&&&*{\bullet}  \ar@{-}@/_/[r] & *{\bullet}\ar@{-}@/_/[l]   \ar@{-}[r] \ar@{-}@/_/[rr]    &   *{\bullet}\ar@{-}@/_/[l]\ar@{-}[r]  &   *{\bullet}   & \longrightarrow &   *{\bullet}   \ar@{-}[r]  \ar@{-}@/_/[rr] &   *{\bullet}\ar@{-}@/_/[l] \ar@{-}[r]   & *{\bullet} &  \longrightarrow &   & *{\bullet} \ar@{-} @(ul,dl)  \ar@{-}@/_/[rr] \ar@{-}[rr] & &*{\bullet} &   \longrightarrow  & *{\bullet}  \ar@{-} @(rd,dl)    & \longrightarrow  & *{\bullet}\\
&&                           & \\
}
%
 \caption{The contractions $\Gamma\to \Gamma(\nu_2^3) \to \Gamma(\nu_2^2)\to \Gamma(\nu_2^1)\to \bullet$; \!we have $\#V_3=9$.} \label{fig:graphscontraction}
  \end{figure}

  \begin{figure}
\xymatrix@=2.1pc{
&                          &                                                          &   *{\bullet} \ar@{-}[dl]_{{0}} \ar@{-}[dr]^{{0}} &                          &
                                                     & &   *{\bullet} \ar@{-}[dl]_{{2}} \ar@{-}[dr]^{{0}} &&   &  &   *{\bullet} \ar@{-}[dl]_{{1}} \ar@{-}[dr]^{{7}} & \\
&*{\bullet}  \ar@{-}@/_/[r]_{0} & *{\bullet}\ar@{-}@/_/[l]_{0}   \ar@{-}[r] _{ {4}} \ar@{-} @(dr,dl)[rdru]_{0}  &   *{\bullet}\ar@{-}@/_/[l]_{\ {4}}
\ar@{-}[r]_{{4}}  &   *{\bullet}                    &  *{\bullet}  \ar@{-}@/_/[r]_{0} & *{\bullet}\ar@{-}@/_/[l]_{0}   \ar@{-}[r] _{ {0}} \ar@{-} @(dr,dl)[rdru]_{2}&   *{\bullet}\ar@{-}@/_/[l]_{\ {0}}
\ar@{-}[r]_{{6}}  &   *{\bullet}    &  *{\bullet}  \ar@{-}@/_/[r]_{0} & *{\bullet}\ar@{-}@/_/[l]_{0}   \ar@{-}[r] _{ {0}} \ar@{-} @(dr,dl)[rdru]_{0}&   *{\bullet}\ar@{-}@/_/[l]_{ \ {0}}
\ar@{-}[r]_{{0}}  &   *{\bullet}  &\\
}
%
 \caption{Three generators of order $2,4$ and $8$ of the group of ghosts.} 
\label{fig:generators}
\end{figure}
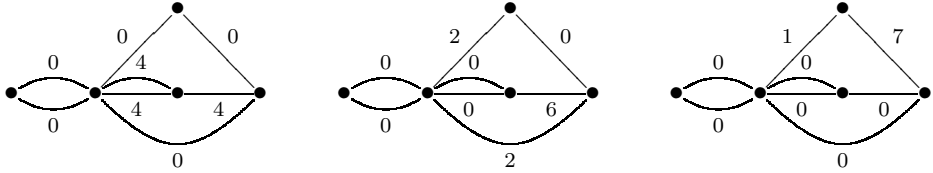

\subsubsection{The length of the fibre of the morphism forgetting level structures.}
\label{rem:CCC} 
In \cite{CCC}, Caporaso, Casagrande and Cornalba check that the length of 
the fibre of the forgetful morphism from their compactified moduli 
of $\ell$th roots to the 
moduli of stable curves equals $\ell^{2g}$ (see \cite[\S4.1, after Lem.~4.1.1]{CCC}).
We show that, in the twisted curve approach used in this paper the length of the fibre is still 
$\ell^{2g}$. The computation 
here is more involved because our moduli functor yields more geometric points, 
which reflects the  fact that the compactified moduli spaces of this paper are 
smooth and actually provide the normalisations of the possibly singular spaces of \cite{CCC} (see  \cite[\S1.2-3]{CEFS}).
The contractions
$\Gamma\to \Gamma(\nu_\ell^k)$ allows to organise our computation efficiently. 

The authors of \cite{CCC} consider 
$\ell$th roots of a line bundle $N$
and the respective moduli functor which can be naturally regarded as a 
fibred category (over the category of schemes).
The authors are not considering the stack representing such a category 
and are mainly interested in the scheme \emph{coarsely} representing this 
moduli functor. For any family of curves $\pi \colon C\to B$,
there exists a scheme $S^\ell(N,\pi)$ representing 
$\ell$th roots of $N$. 
In general  $S^\ell(N,\pi)$ is smooth over $B$, but not proper over $B$.
In the same spirit of 
the present paper, the authors introduce 
a new, less restrictive, notion of root: the 
 ``limit $\ell$th root of $N$''.
The corresponding moduli functor 
is shown in \cite{CCC} to be 
coarsely represented by a proper, but possibly singular, 
scheme $\ol S{ }^\ell(N,\pi)$ over $B$
(singularities occur when $\ell$ is not prime, see discussion 
after \cite[Thm. 4.2.3]{CCC}).

We assume $N=\cO$; then 
for any family $\pi\colon C\to B$ we have a possibly singular scheme $\ol S{}^\ell(\cO,\pi)$ and,
as a byproduct, a moduli space $\ol S{}^{\ell}_g$, which is a 
finite $\ell^{2g}$-cover of the proper moduli space $\ol{\mathcal M}_g$.
We consider, for any $g\ge 2$ and $\ell\ge 1$, the moduli stacks 
$\sta T_{g,\ell}=\sqcup_{d\mid \ell} {\sta R}_{g,d}$  
and the compactifications 
$\ol{\sta T}_{g,\ell}=\sqcup_{d\mid \ell} \ol{\sta R}_{g,d}$ yielding 
the finite cover 
$$\sta p\colon \ol{\sta T}_{g,\ell}\to \ol{\sta M}_{g}.$$ 
We can now state precisely what we mean (both 
here and in \cite{CCC}) by ``the fibre over a point of'' 
$\ol{\sta M}_g$. 
In \cite{CCC}, for a stable curve $C$ over $k$, the main focus is 
the scheme 
$\ol S{}^\ell(\cO, C\to \Spec k)$, which is the zero-dimensional scheme coarsely representing 
the fibred product of categories obtained by pulling back 
 the fibred category of limit $\ell$th roots of $\cO$ 
over $\ol{\sta M}_g$
via the map 
$b\colon \Spec k\to \ol{\sta M}_g$ induced by $C$.
In complete analogy,  our  focus here is 
the coarsening  
of the fibred 
product $\Spec k\ _b\!\times_{\sta p}\ol{\sta T}_{g,\ell}$, which we denote by $F_b$.
Notice that, by definition, 
this does not involve the automorphism group of $C$ (as it would have been the case if we 
had considered  $\Spec k\ _b\!\times_{\sta p}\ol{T}_{g,\ell}$ instead).
Of course the reader may also read this computation under the 
additional assumption that the curve $C$ has trivial automorphism 
group (in this case we are actually computing $\Spec k\ _b\!\times_{\sta p}\ol{T}_{g,\ell}$). 
 
Fix a stable curve $C$ with dual graph $\Gamma$, 
that is a geometric point $b\to \ol {\sta M}_g$. 
We check  
$\mathrm{length}(F_b)=\ell^{2g}$ 
for scheme-theoretic fibre $F_b$,
coarsening of the 
base change of $\ol{\sta T}_{g,\ell}\to \ol{\sta M}_g$
to 
$b$.
Each connected component of $F_b$ 
is a, possibly nonreduced, 
zero-dimensional scheme 
corresponding to an isomorphism 
class of a triple $(\ssC,\ssL,\phi\colon \ssL^{\otimes \ell}\to \cO)$
with a multiplicity $M$ and corresponding 
characteristic functions $\pmb \nu=(\nu_p)$. 
By Theorem \ref{thm:ghosts_compositel}, the length 
of such zero-dimensional scheme is 
$$\#\aut_C(\ssC)/\#\ul{\aut}\level=
\left(\prod\nolimits_{p\mid \ell} 
\prod\nolimits_{k=1}^{e_p} p^{\#E(\nu_p^k)}/ p^{\#V(\nu_p^k)-1}\right)=
\prod\nolimits_{p\mid \ell} \prod\nolimits_{k=1}^{e_p} 
p^{b_1(\Gamma (\nu_p^k))}.	
$$
(by the definition of $F_b$, we are not considering
the action of $\aut(C)$.)

The number of connected components is 
$$\sum_{\stackrel{M\in \ker\partial}{\pmb \nu_M=(\nu_p)}} 
\ell^{2p_g(C)} \prod\nolimits_{p\mid \ell} \prod\nolimits_{k=1}^{e_p}
p^{b_1(\Delta(\nu_p^k))}.$$
This happens because the multiplicities range over the elements of $\ker \partial$ 
by Proposition \ref{pro:Misclosed}.
Furthermore, once the multiplicity is specified, the  numbers of 
$\ell$th roots equal the summands appearing above. Indeed,  
we can count by taking a product on prime factors of $\ell$ and reduce to 
showing the claim for $\ell=p^e$. Then, we need to show that the number of $\ell$th roots sharing the 
same multiplicity $M$ is 
$$ \prod\nolimits_{k=1}^{e} p^{2p_g(C)} p^{b_1(\Delta(\nu_p^k))}.$$
This amounts to showing that the factors above are the numbers of 
 $p^k$th roots up to $p^{k-1}$st roots 
for any $k=1,\dots, e$. The factor $p^{2p_g(C)}$
counts $p^k$th roots up to $p^{k-1}$st roots on the normalisation.
The last factor involves  $\Delta(\nu_p^k)$, the subgraph of $\Gamma$ formed by the edges $e$
where $p^k\mid M(e)$. By \eqref{eq:twisters}, if $p^k$ does not divide $M(e)$,
iterated pullbacks 
via $(x,y)\mapsto (\xi_{r(e)}x, y)$  at the node $\sta n$ corresponding
to $e$ identify to each others all gluing data in $p^{k-1}\ZZ/p^k \ZZ$
along $\sta n$. Therefore the gluings up to automorphisms are determined by the subgraphs 
$\Delta(\nu_p^k)$ and their number is the number of elements of 
$H^1(\Delta(\nu_p^k)), \pmmu_{p^k}/\pmmu_{p^{k-1}})$. We get exactly the
power of $p$ appearing in the last factor of the displayed formula above.

Finally, since $\Gamma (\nu_p^k)$ is given by 
collapsing the subgraph $\Delta(\nu_p^k)$, the Betti numbers 
$b_1(\Delta(\nu_p^k))$ and $b_1(\Gamma (\nu_p^k))$ add up to $b_1(\Gamma)$; we get
$$\mathrm{length}(F_b)=\sum_{\stackrel{M\in \ker\partial}{\pmb \nu_M=(\nu_p)}} 
\ell^{2p_g(C)} \prod\nolimits_{p\mid \ell} \prod\nolimits_{k=1}^{e_p}
p^{b_1(\Delta(\nu_p^k))+ b_1(\Gamma (\nu_p^k))} = 
\sum_{\stackrel{M\in \ker\partial}{\pmb \nu_M=(\nu_p)}} 
\ell^{2p_g(C)}\ell^{b_1(\Gamma)}=\ell^{2g}.$$

\subsubsection{No-Ghosts}
Thm.~\ref{thm:ghosts_compositel} and Lem.~\ref{tsohg} imply a no (nontrivial) ghost criterion.
\begin{corollary}\label{cor:noghosts_compositel}
Let $\ell$ be any positive integer.
The group $\ul\aut_C\level$ is trivial if and only if
for any prime factor $p$ of $\ell$ the graph $\Gamma(\nu_p^{e_p})$
is a bouquet.
\qed \end{corollary}

\begin{remark} \label{rem:conditions_comments} In analogy with the case where $\ell$ is prime,
one may consider the condition
``the contraction $\Gamma'$ of $\{e\mid
\ell$ divides $M(e)\}$ is a bouquet'', which clearly implies 
the above no-ghosts condition.
The converse is false:
for $\ell=6$, consider $\Gamma$ with vertices $v_1,v_2$,
edges $e_1,e_2, e_3$
 going from $v_1$ to $v_2$, set $M(e_i)=i$.\end{remark}

\begin{remark}\label{rem:generalisations}
 Lem.~\ref{lem:noqrefl_ghosts},
 Prop.~\ref{pro:quasireflections} generalise \emph{verbatim},
and, by Cor.~\ref{cor:noghosts_compositel}, the same holds for \
Thm.~\ref{thm:smooth} once we replace ``$\Gamma(\nu)$ is a bouquet'' by ``$\Gamma(\nu_p^{e_p})$ is
a bouquet for any prime $p\mid\ell$''. 
\end{remark}
We can also state the generalisation as follows.
\begin{theorem}\label{thm:smooth_compositel}
Let $g\ge 4$ and let $\ell\ge 1$.
The point representing $\level$ in $\rr_{g,\ell}$ is smooth if and only if
the group $\aut'(C)$ is generated by ETIs of $C$
and the graphs $\Gamma(\nu_p^{e_p})$
(obtained by contracting
the edges $e$ for which $M(e) \in \ZZ/p^{e_p}$ vanishes) are bouquets
for any prime $p\mid\ell$.
\end{theorem}

\subsubsection{Automorphism group of level structures over stack-theoretic elliptic tails.}
We describe the action of the automorphism
group of the stack-theoretic elliptic
tail of Exa.~\ref{exa:level_on_elltails} on $\Pic$.
We are interested in level structures 
on a curve with an elliptic 
tail $\sta E$; it is natural to fix a divisor $l$ of 
$\ell$, which should be thought of as the 
the order of the restriction of the level-$\ell$ 
structure to $\sta E$.
We refer the reader to Example \ref{exa:stackytails} 
for the study of level structures on $\sta E$ when $\sta E$ 
is an irreducible 
twisted curve with a single node.

The $1$-pointed $1$-nodal
twisted genus-$1$ curve
$(\sta E,n)$ is given by the stack-theoretic quotient of
$\wt E=\PP^1/(0\sim \infty)$ by $\pmmu_r$ where $r$ 
divides $l$ and $\pmmu_r$ acts by 
multiplication as usual.
Consider $\sta p\colon \wt E\to \sta E$, 
the isotypical decomposition 
$\sta p_*\cO=\bigoplus _{ \chi\in \ZZ/r} \sta L_\chi$, 
and the $l$-torsion line bundle 
$\ssL_{\ram}:=\sta L_{\chi=1}$ on $\sta E$
with $\leveliii_{\ram}\colon \ssL_{\ram}^{\otimes l}\to \cO$ 
obtained by taking the $(l/r)$th tensor power of
the isomorphism $\ssL_{\ram}^{\otimes r}\cong \cO$.
We also consider the $l$-torsion line bundle  $\sta L_{\etale}$, pull-back via $\epsilon_{\sta E}\colon \sta E\to E$
of the sheaf of regular functions $f$ on the normalisation satisfying $f(\infty)=\xi_{l}f(0)$ for an $l$th primitive root of unity $\xi_l$.

We have
\begin{equation}\label{eq:picelltail}\Pic(\sta E)[l] \cong 
\pmmu_l\oplus \ZZ/r. 
\end{equation}
The second summand has the distinguished generator 
$\ssL_{\ram}:=\sta L_1$. The first 
summand is generated by $\sta L_{\etale}$, 
defined after choosing a primitive root of unity $\xi_l$. 

We have 
$$\aut(\sta E,n)=
\{\sta a\in \aut(\sta E)\mid \sta a(n)=n\}\cong\pmmu_2\oplus \pmmu_r,$$
where the first summand is generated by 
the distinguished involution $\sta i$, whereas 
the second summand is generated by $\sta g$, defined 
after choosing an $r$th root of unity $\xi_r$ 
by the local picture 
$\sta g\colon (x,y)\mapsto (\xi_r x,y)$ at $n$,
and the condition $\sta g|_{\sta E\setminus \{ n\}}=\id$.

Then $\sta i$ operates on $\Pic(\sta E)[l]$ as the 
passage to the inverse
$$\sta i\colon (\al \in \pmmu_l, k\in \ZZ/r)\mapsto 
(\al^{-1}, -k).$$
On the other hand any given 
root of unity $\zeta  \in\pmmu_r$ operates on 
$\Pic(\sta E)[l]$ as 
$$\zeta \colon (\al \in \pmmu_l, k\in \ZZ/r)\mapsto 
(\al k(\zeta), k),$$
where the product between $\al\in \pmmu_l$ and $k(\zeta)\in \GG_m$ is 
obviously taken within $\GG_m$.

More explicitly, in terms of the explicit bases mentioned above, 
we have the additive groups
$\Pic(\sta E)[l]\cong \langle \ssL_{\etale}, \ssL_{\ram}\rangle
=\ZZ/l\oplus \ZZ/r$
and $\aut(\sta E,n)\cong \langle \sta i, \sta g\rangle
=\ZZ/2\oplus \ZZ/r$
and the action of $(a_1,a_2)=\sta i^{a_1}\circ \sta g^{a_2}\in
\aut(\sta E,n)$ on the line bundle
$(k_1,k_2)= (\sta L_{\emph{\etale}})^{\otimes k_1} \otimes \sta L_{\ram}^{\otimes k_2}$ in $\Pic(\sta E)[l]$
yields
\begin{equation}\label{eq:actionontail}
(a_1,a_2)\cdot (k_1,k_2)= ((-1)^{a_1}k_1+(l/r)a_2k_2, (-1)^{a_1}k_2),                                                                  \end{equation}
where $a_2k_2$ is the product in $\ZZ/r$.

In view of the study of ghost automorphisms of level-$l$ curves
we consider a faithful order-$l$ line bundle $\ssL$ on $\sta E$;
in other words,
we consider an order-$l$ 
element $(\al,k)\in \pmmu_l\oplus \ZZ/r\cong \Pic(\sta E)[l]$
where $k$ is prime to $r$ (faithfulness).\begin{proposition}\label{pro:autelltail}
The complete list of
nontrivial automorphisms $(\sigma\in \pmmu_2,\zeta\in \pmmu_r)\in 
\aut(\sta E,n)$ fixing the isomorphism class of the 
order-$l$ line bundle $\ssL$ 
is as follows
\begin{enumerate}
\item $l=1$, $r=1$,  $\ssL=\cO$, and $(\sigma,\zeta)=(-1,1)$;
\item $l=2$, $r=1$, $\ssL\in \Pic[2]\setminus \{\cO\}$, 
and $(\sigma,\zeta)=(-1,1)$;
\item $l=2$, $r=2$, $\ssL= (1,\ssL_{\ram})$ or 
$(-1,\ssL_{\ram})\in\Pic[2]=\pmmu_2\oplus \ZZ/2$, and $(\sigma,\zeta)=(-1,1)$;
\item $l=4$, $r=2$, $\ssL= (\al,\ssL_{\ram})\in\Pic[4]=\pmmu_4\oplus \ZZ/2$ ($\al$ primitive), and $(\sigma,\zeta)=(-1,-1)$.
\end{enumerate}
\end{proposition}

\begin{proof}
There are no nontrivial solution $(\sigma, \zeta)$ of the form 
$(1,\zeta)$, because this yields $\al k(\zeta)=\al$, which implies $\zeta=1$ ($\ker(k)=1$). Then,
we look for  solutions $(\sigma, \zeta)$ of the form 
$(-1,\zeta)$; hence we solve 
the equations $\al^{-1} k(\zeta)=\al$ 
and $-k=k\mod r$ (with $k$ prime to $r$). 
Then $k=0$ (and $r=1$) or $k=r/2$ (and $r=2$). 
Cases (i) and (ii) arise from $k=0$, 
which yields $\zeta=1$ and $\al=1$ (case (i))
or $\al=-1$ (case (ii)).
Cases (iii) and (iv) arise from $k=1$, 
which yields $\zeta=1$ and $\al^2=1$ (case (iii))
or $\zeta=-1$ and $\al^2=-1$ (case (iv)).
\end{proof}

\begin{remark}\label{rem:explicitautomstackytail_1_2}
Notice that in the cases 
(i),(ii), (iii), the automorphism is the
 canonical involution $\sta i$. 
 This may be thought of 
 as the restriction on an elliptic tail 
$(\sta E,n)$ 
of the automorphism of a
  level-$\ell$ curve $\level$; then, the ETI 
fixing $\ol{\sta C\setminus \sta E}$ and yielding 
$\sta i$ on $\sta E$
operates on $\defo\level$ as 
the quasireflection $(-\mathbb I_1) \oplus \id$.

Again, if we choose explicit bases 
$\Pic(\sta E)[l]\cong \langle \ssL_{\etale}, \ssL_{\ram}\rangle
=\ZZ/l\oplus \ZZ/r$
and $\aut(\sta E,n)\cong \langle \sta i, \sta g\rangle
=\ZZ/2\oplus \ZZ/r$ 
we can explicitly realise the
fixed line bundle: $\cO$ in case (i), 
$\ssL_{\etale}$ in case (ii), and 
 $\ssL_{\ram}$ and $\ssL_{\ram}\otimes \ssL_{\etale}$
 in case (iii).
\end{remark}
\begin{remark}\label{rem:explicitautomstackytail_3}
In case (iv), the automorphism 
is the involution obtained as the composition of $\sta i$
with the order-$2$ 
ghost $\sta g$ operating locally 
at the node as $(x,y)\mapsto (-x,-y)$.  
Again
$(\sta E,n)$ with this automorphism and its fixed 
$4$ torsion bundle, may be thought of as 
the elliptic tail of a  level-$\ell$ 
curve $\level$. The involution 
fixing $\ol{\sta C\setminus \sta E}$ and yielding 
$\sta i\circ \sta g$ on $\sta E$ does 
not act as a quasireflection (see
Prop.~\ref{pro:quasireflections}). Indeed, the action 
on $\defo\level/\defo(\leveli,\levelii,\levelii,\Sing(\leveli))$ 
is nontrivial only on the parameter $\tau_1$
smoothing $n$: 
$\tau_1 \mapsto -\tau_1$.
On the other hand,  on 
$\defo(\leveli,\levelii,\levelii,\Sing(\leveli))$ 
the action is trivial except on the parameter $\tau_2$   
deforming only the tail: we have 
$\tau_2\mapsto -\tau_2$. 
Therefore the involution fixes a codimension-$2$ subspace 
of $\defo\level$ and operates as $-\mathbb I_2\oplus \id$.
Finally, when 
we choose the above explicit bases of $\Pic$ and $\aut$, 
we may realise the
level-$4$ structure on $(\sta E,n)$ as $\sta L_{(\sta E,n)}:=\ssL_{\ram}\otimes \ssL_{\etale}$. Indeed, we have 
$$(\sta i\circ \sta g)^*(\sta L_{(\sta E,n)})=\sta i^*(\sta g^*\ssL_{\ram}\otimes \sta g^*\ssL_{\etale})\stackrel{\small{\eqref{eq:actionontail}}}{=}
\sta i^*(\ssL_{\ram}\otimes \ssL_{\etale}^{\otimes 2}  \otimes \ssL_{\etale})=
\sta i^*(\ssL_{\ram}^{\vee}\otimes \ssL_{\etale}^{\vee})=
\sta L_{(\sta E,n)}.$$
\end{remark}

\subsection{Noncanonical singularities}
The problem of describing the locus of 
noncanonical singularities
within the moduli space of 
level-$\ell$ curves  
is treated locally: we systematically study 
the action of $\aut\level$ on $\defo\level$. 
By the \rsbt\ criterion, the age invariant introduced below
detects in terms of rational numbers 
the cases where noncanonical singularities occur.

Throughout the rest of the paper we use the notation $\{x\}$, which 
stands for the fractional part of a real number $x$; in other words, we set
$\{x\}:=x-\lfloor x\rfloor$.

Although we do not use this point of view in this paper,
we mention in passing that Abramovich, Graber, and Vistoli
have introduced in \cite{AGV}, a global age grading function
defined on the cyclotomic inertia stack
$$\sta{AGE} \colon I_{\pmmu}(\overline{\sta R}_{g,\ell})\longrightarrow \QQ_{\ge 0}.$$
One could state our description of the noncanonical singularities
locus as a description 
of the locus ${\sta{AGE}}^{-1}(]0,1[)$ within 
the cyclotomic inertia stack. We are indebted to the authors 
of \cite{AGV} for this point of view; nevertheless, the 
following introduction of the age grading is elementary and can be read without 
referring to \cite{AGV}.

\subsubsection{The age of representations of $\pmmu_r$.}
We consider the group $\pmmu_r$ for any 
positive integer $r$ and we define an additive 
age grading over the representation ring $R\pmmu_r$.
Since  $\mathrm{Hom}(\pmmu_r,\GG_m)$ is 
canonically identified with $\ZZ/r$,
we can define the age grading of the character 
$k\in \ZZ/r$ as $k/r\in \QQ$.
Since the characters  
in $\mathrm{Hom}(\pmmu_r,\GG_m)$
form a basis for the representation ring $R\pmmu_r$, 
this yields an additive homomorphism
$\age\colon R\pmmu_r\to \QQ.$
\subsubsection{Cyclotomic injections and group elements.}
Let $G$ be a finite group. 
When working over the complex numbers there is 
a canonical 
identification between the set of 
group elements and the set of cyclotomic 
injections 
\begin{equation}\label{eq:identif_injection-elem}
\{g\mid g\in G\}\overset{1:1}\longleftrightarrow 
\bigsqcup\nolimits_{r\ge 1} \{\gamma\mid \gamma \colon \pmmu_r \hookrightarrow G\}.\end{equation}
The identification is the obvious one: to an 
element $g\in G$ of order $r$ 
we attach the homomorphism 
$\gamma\colon \pmmu_r\hookrightarrow G$ 
mapping $\exp(\frac{2\pi i}r)$ to $g$; conversely, 
we set $g=\gamma(\exp(\frac{2\pi i}r))$. 

Over any base field this identification depends 
on the choice of a primitive 
root of unity $\xi_r\in \pmmu_r$ for any positive integer $r$. 
Below, we define---without the need of 
any such choice---the age grading of cyclotomic injections
within a group $G$ operating on $V=\mathbb A^m$.

\subsubsection{The age grading for a $G$-representation}
Consider a $G$-representation: $\rho\colon G\to \GL(V)$
where $V= \mathbb A^m$.
Any injective homomorphisms 
$\gamma\colon \pmmu_r\hookrightarrow G$
yields, by composing with $\rho$, 
a  $\pmmu_r$-representation. We get an 
invariant of the $G$-representation
\begin{align}\label{eq:ageinvariantofGrep}
\age_V\colon \bigsqcup\nolimits_{r\ge 1} \{\gamma\mid \gamma \colon \pmmu_r \hookrightarrow G\}& \longrightarrow \QQ\\
\gamma&\longmapsto \age(\rho\circ\gamma).\nonumber 
\end{align}
Explicitly, 
$\age_V(\gamma)$ is defined as follows: for any 
primitive root of unity $\zeta$ in $\pmmu_r$ the matrix 
corresponding to the action of 
$\gamma(\zeta)$ on $V$  is conjugate  to 
$\diag((\zeta)^{a_1} , \ldots ,(\zeta)^{a_m})$ and 
we have  $$\age_V(\gamma)=
\frac{a_1}
r + \ldots + \frac{a_m}
r\in \QQ.$$
The coefficients $a_1,\dots, a_m$ are uniquely determined 
by imposing $0\le a_i<r$ and do 
not depend on the choice of the primitive root of unity $\xi_r$.

Over the complex numbers, 
the notion of group element and that 
of cyclotomic injection are interchangeable and 
$\age_V$ can be defined directly on $G$.
Moreover, the explicit definition above can 
be given by fixing $\zeta:= \exp(\frac{2\pi i}r )$. 
\subsubsection{The \rsbt{} criterion.} 
Assume that the point at the origin of
$V$ modulo $G\in \GL(V)$ is singular.
Such a singularity is
\emph{canonical} if and only if any pluricanonical form on the smooth locus
extends to any desingularisation of ${V/G}$. 
In other words, for all $q\in \ZZ$
sufficiently high and divisible, we have
$$\Gamma((V/G)^{\reg}, \omega^{\otimes q})=\Gamma(\widehat{V/G}, \omega^{\otimes q})
\qquad \text{for any desingularisation $\widehat{V/G}\to {V/G}$.} $$

\begin{theorem}[\rsbt{} criterion \cite{re1980, ta1982, re2002}]
\label{thm:rsbt}
Let us assume that the finite 
group $G$ 
operates on $V$ without quasireflections.
The scheme-theoretic quotient $V/G$ has a
noncanonical singularity at the origin
if and only if the image of 
$\age_V$ (see \eqref{eq:ageinvariantofGrep}) 
intersects $]0,1[$.
\end{theorem}

\begin{remark}
The above definition does not depend on any
choice of  primitive roots of unity. If we 
fix a root of unity $\xi_r$ for every order $r\ge 1$ 
the age grading $\age_V$ can be defined directly on $G$ via 
\eqref{eq:identif_injection-elem}: we get 
$$G\xrightarrow{\ 1:1 \ }
\bigsqcup\nolimits_{r\ge 1} \{\gamma\mid \gamma \colon \pmmu_r \hookrightarrow G\}\xrightarrow {\ \ \ } \QQ.$$
It is important to stress that 
the image of the above map only depends on the second morphism.
More explicitly,
for a fixed element $g\in G$ of order $r$,
we have the following relation between the gradings 
$\age'$ and $\age''$ attached to two 
choices $\zeta'$ and $\zeta''$ 
of primitive $r$th roots of unity in $\pmmu_r$.
If $\zeta'=(\zeta'')^a$
for a suitable 
$a$ prime to $r$, then $\age''(g)=\age'(g^a)$.

Therefore, we  
fix, once and for all, a system of roots of unity in order to 
simplify the combinatorial analysis. In particular, 
this will allow us to specify ghosts simply 
by writing $\ZZ/\ell$-valued
symmetric functions 
(as we already did in Fig.~\ref{fig:generators}). 
Furthermore, this will allow us to 
define the age of a ghost acting on the deformation space. 

\end{remark}

\begin{assumption}[choice of $r$th
primitive roots of unity for all $r$]\label{assu:roots}
We now fix, for any positive integer $r$ 
a primitive $r$th root of unity $\xi_r\in \pmmu_r$. 
This is the same as fixing isomorphisms 
$\ZZ/r \rightarrow \pmmu_r, $
$ k\mapsto (\xi_r)^k$, all $r\in \ZZ_{\ge 1}$. 
In particular, we work with a fixed 
identification \eqref{eq:identif_injection-elem} and, to a given representation $G\to \GL(V)$
we attach $\age_V\colon G\to \QQ$, 
the non-negative grading directly defined on $G$.
Note that, under any 
chosen identification $\ZZ/r\cong \pmmu_r$,
the pairing $\ZZ/r\times \pmmu_r=\mathrm{Hom}(\pmmu_r,\GG_m)\to
\pmmu_r$ matches the product of the ring $\ZZ/r$.

In this way for 
$M\in \ZZ/\ell$ and
$\sta a\in \pmmu_r= \pmmu_{\ell}^M$ 
we can express $\sta a \odot M$ as a product. Via
$\pmmu_r \cong \ZZ/r\subseteq \gcd(M,\ell)\ZZ/\ell$, we write 
$\sta a$ as a multiple of $\gcd(M,\ell)$
modulo $\ell$; then we have 
\begin{equation}\label{eq:odotprod}\sta a\odot M=\frac{aM}{\gcd(M,\ell)}\in \pmmu_r \cong  \gcd(M,\ell)\ZZ/\ell.\end{equation}
When $\ell$ is prime the product $\odot$ is simply the product 
within the ring $\ZZ/\ell$.
\end{assumption}

\begin{definition}[junior and senior group elements]\label{defn:junior}
An element $g\in G$ operating nontrivially on $V$ is 
\emph{senior on $V$} if $\age_V(g)\ge 1$,
and is \emph{junior on $V$} if $0<\age_V(g)<1$ 
(Ito and Reid's terminology, \cite{IR}).
\end{definition}

Now, Theorem \ref{thm:rsbt} may be regarded as saying:
$V/G$ has a noncanonical singularity at the origin
if and only if there exists an element 
$g\in G$ which is junior on $V$.

\subsubsection{The computation of the age of an automorphism on $\defo\level/\qr$.} \label{subsubsect:modqr}
We mod out
$\defo\level$ by the group $\qr$ of automorphisms spanned by quasireflections; \emph{i.e.},
by Proposition \ref{pro:quasireflections} this amounts to
modding out the ETIs
restricting to the identity on the entire curve except for an elliptic tail component $\sta E$
where the canonical involution $\sta i$ fixes $\sta L\!\mid_{\sta E}$.
These involutions operate simply by changing the sign of
the parameter $\tau_e$ smoothing the node where $\sta E$ meets the rest of the curve.
We refer to $\sta E$ as a quasireflection elliptic tail component
or, simply, \emph{quasireflection tail} (QR tail).
We refer to the node joining the quasireflection tail to the rest of the curve as
a \emph{quasireflection node} (QR nodes) and
we identify in this way a partition
$\Sing(\ssC)\cong\Sing_{\mathrm{QR}}(\ssC)\sqcup \Sing_{\mathrm{nonQR}}(\ssC)$
and a partition
$E=E_{\mathrm{QR}}\sqcup E_{\mathrm{nonQR}}$.
Equations \eqref{eq:defolevel} and \eqref{eq:defomodtopolpres} yield
$$\defo\level/\qr \cong
\Bigl(\bigoplus_{e\in E}\mathbb A^1_{\ol \tau_e}\Bigr)\oplus
\Bigl(\bigoplus_{v\in V}H^1(C_v',T(-D_v))\Bigr),  \text{ with }\ol \tau_e =\begin{cases}
\tau_e^2 & \text{for $e\in E_\lqr$}; \\
\tau_e   & \text{for $e\in E\setminus E_\lqr$}.
                                                        \end{cases}
$$
The action of $\aut\level$ on $\level$ descends to an action
without quasireflections on the above space.
\begin{corollary}\label{cor:rsbt_for_us}
The point at the origin of $\defo\level/\aut\level$
is a  noncanonical singularity if and only if there exists
an automorphism  which is junior on
$\defo/\qr$. 
\end{corollary}

\begin{example}\label{exa:trivialaction}
The stack-theoretic ETI
of Rem.~\ref{rem:explicitautomstackytail_1_2} acts trivially on $\defo/\qr$.
\end{example}
\begin{example}
The automorphism $\sta a$ extending $\sta i\circ \sta g$ 
 in Rem.~\ref{rem:explicitautomstackytail_3}
 operates on $\mathbb A^1_{\tau_1}\oplus 
 \mathbb A^1_{\tau_2}\oplus \mathbb A^{3g-1}$ 
 as $(-\mathbb I_2)\oplus \id$.
 Furthermore, by Proposition \ref{pro:autelltail}, $\sta i$ is not an automorphism of $\level$.
We have $\defo/\qr=\mathbb A^1_{\tau_1}\oplus \mathbb A^1_{\tau_2}\oplus \mathbb A^{3g-1}$ 
and $\sta a$ is senior: $\age= 1/2+1/2=1$.
\end{example}

\subsubsection{The computation of the age of a ghost.}\label{subsubsect:ageofghosts}
Using Proposition \ref{pro:ghosts} and Theorem \ref{thm:ghosts_compositel}
we can easily compute the age of a ghost automorphism $\sta a\in 	\ul{\aut}_C\level$ attached to $b\in G_{\pmb \nu}(\Gamma;\pmmu_\ell)$. 
Assumption \ref{assu:roots} allows us to regard $b$ as a 
$\ZZ/\ell$-valued $1$-chain $b\in G_{\pmb \nu}(\Gamma;\ZZ/\ell)$.
We point out that the explicit expressions 
${[M(e)^{-1}]_{\ell} b(e)}$ in 
\eqref{eq:M-1explicitlprime} and 
${[m(e)^{-1}]_{r(e)} b(e)}$ in \eqref{eq:M-1explicit} 
may be interpreted as multiplications in the ring 
$\ZZ/\ell$.

When $\ell$ is prime we have
\begin{equation}\label{eq:ageprime}
\age(\sta a)=\sum_{e\in E} \left \{ \frac{{\sta a}(e)}\ell
\right\} =\sum_{e\in E} \left\{ \frac{[M(e)^{-1}]_\ell b(e)}\ell\right\} \qquad \text{($\ell$ prime)},                                                                                                                                                                   \end{equation}
where $\{\ \}$ denotes the fractional part and 
the terms at the numerators are integer representatives of ${\sta a}(e), b(e)$ and $[M(e)^{-1}]_\ell$ in $\ZZ/\ell$
(each summand in the above expression is clearly independent of the
choice of the representatives modulo $\ell$).
For composite $\ell$, Thm.~\ref{thm:ghosts_compositel}, \eqref{eq:multinr(e)}
yields
\begin{equation}\label{eq:agel}
\age(\sta a )=\sum_{e\in E} \left \{ \frac{{\sta a}(e)}\ell \right\} = \sum_{e\in E} \left\{ \frac{[m(e)^{-1}]_{r(e)} b(e)}{r(e)} \right\}=
\sum_{e\in E} \left\{ \frac{[m(e)^{-1}]_{r(e)} \wt b(e)}{\ell} \right\},                                                                               \end{equation}
where $\wt b(e)\in \ZZ/\ell$ is the image of $b(e)\in \ZZ/r(e)$ via the identification
$\wt b(e)=(\ell/r(e))b(e)=\gcd(\ell,M(e))b(e)$.
Again the above definition does not depend on the choices of the integer representatives
of ${\sta a}(e)\in \ZZ/\ell$ and of $[m(e)^{-1}]_{r(e)}, b(e)\in \ZZ/r(e)$.

In Example \ref{exa:ghostgroup},  we presented
 three ghost automorphisms, the corresponding
symmetric functions $e\mapsto \sta a(e)$ on the set
of oriented edges are given in Fig.~\ref{fig:generators}. Equation \ref{eq:agel} allows
us to compute their age.
According to \eqref{eq:agel},
the order-$2$ automorphism has age $3/2$,
the order-$4$ automorphisms has age $5/4$, whereas the order-$8$ automorphism
has age $1$. Hence, in all these three cases the ghosts are senior.
However, ghost automorphisms operating as junior ghosts actually occur,
we provide some examples,
which will play a role in the proof of Thm.~\ref{thm:no_junior_ghosts}.
\begin{example}\label{exa:aut5}
Let $\ell=5$. Consider a level curve whose dual graph has multiplicity $M$, pictured in the first diagram of Fig.~\ref{fig:aut5}; write
${\pmb \nu}$ for the characteristic function of the support of $M$. Here we have $\pmb\nu=\bf0$.
In the second and third diagram we specify the symmetric function
$\sta a\in S_{\pmb \nu}(\Gamma;\ZZ/5)$ and the corresponding  $1$-cochain $b\in G_{\pmb \nu}(\Gamma;\ZZ/5)$.
Using \eqref{eq:ageprime} we get
$\age(\sta a)= 1/5+1/5+1/5+1/5=4/5.$
 \begin{figure}[h]
\xymatrix@=2.1pc{
&&*{\bullet} \ar@{-}[r] _{\ \ \ \ \underset{\leftarrow} {2}} \ar@{-} @(dr,dl)[rr]_{\underset{\rightarrow}1}
\ar@{-} @(ur,ul)[rr]^{\underset{\rightarrow}{1}}  &   *{\bullet}
\ar@{-}[r]_{\underset{\leftarrow}{2}\ \ \ \ }  &      *{\bullet}                 &
&*{\bullet} \ar@{-}[r] _{{1}} \ar@{-} @(dr,dl)[rr]_{1}
\ar@{-} @(ur,ul)[rr]^{{1}}  &   *{\bullet}
\ar@{-}[r]_{{1}}  &      *{\bullet}                 &
&*{\bullet} \ar@{-}[r] _{\ \ \ \ \underset{\leftarrow} {2}} \ar@{-} @(dr,dl)[rr]_{\underset{\rightarrow}1}
\ar@{-} @(ur,ul)[rr]^{\underset{\rightarrow}{1}}  &   *{\bullet}
\ar@{-}[r]_{\underset{\leftarrow}{2}\ \ \ \ }  &      *{\bullet}                 & \\
}
 \caption{The multiplicity cochain $M$, 
 the symmetric function $e\mapsto {\sta a}(e)$ and the cochain
$e\mapsto b(e)$.} \label{fig:aut5}
\end{figure}
\end{example}

\begin{example}\label{exa:aut5bis}
We consider again a level-$5$ curve, but this time we only need three nodes 
and two components.
The dual graph has the multiplicity $M$
pictured in the first diagram of Fig.~\ref{fig:aut5bis}. Again, 
we have $\pmb\nu=\bf0$ and 
in the second and third diagram we specify the symmetric function
$\sta a\in S_{\pmb \nu}(\Gamma;\ZZ/5)$ and the corresponding  $1$-cochain $b\in G_{\pmb \nu}(\Gamma;\ZZ/5)$.
Using \eqref{eq:ageprime} we get
$\age(\sta a)= 2/5+1/5+1/5=4/5.$
 \begin{figure}[h]
\xymatrix@=2.1pc{
&&
*{\bullet} \ar@{-}[rr] _{ \underset{\leftarrow} {3}} 
\ar@{-} @(dr,dl)[rr]_{\underset{\rightarrow}2}
\ar@{-} @(ur,ul)[rr]^{\underset{\rightarrow}{1}}  & &  
*{\bullet}
       &
&
*{\bullet} \ar@{-}[rr] _{ {1}} 
\ar@{-} @(dr,dl)[rr]_{1}
\ar@{-} @(ur,ul)[rr]^{{2}}  & &  
*{\bullet}
&&
*{\bullet} \ar@{-}[rr] _{ \underset{\leftarrow} {3}} 
\ar@{-} @(dr,dl)[rr]_{\underset{\rightarrow}2}
\ar@{-} @(ur,ul)[rr]^{\underset{\rightarrow}{2}}  & &  
*{\bullet}&&\\
}
 \caption{The multiplicity cochain $M$, 
 the symmetric function $e\mapsto {\sta a}(e)$ and the cochain
$e\mapsto b(e)$.} \label{fig:aut5bis}
\end{figure}
\end{example}

\begin{example}\label{exa:aut8}
Let $\ell=8$. We adopt the notation $M$ and ${\pmb \nu}$ as above. This time $\pmb\nu$ is the vector-valued function
attached to $M$. Again, the second and third diagrams specify the
 symmetric function $\sta a\in S_\nu(\Gamma;\ZZ/8)$ and
the corresponding $1$-cochain $b\in G_{\pmb \nu}(\Gamma;\ZZ/8)$.
 More precisely, we have written next to each edge
the values of $\wt{\sta a}$ and $\wt b$ in $\ZZ/8$; \emph{e.g.}, ``$2$''
appearing in the second diagram
represents the order-$4$ element $2 \mod 8$ in $\ZZ/8$.
Using \eqref{eq:agel} we get
$\age(\sta a)= 1/8+1/8+1/8+1/8+2/8=3/4.$
 \begin{figure}[h]
\xymatrix@=1.9pc{
   &                    &                                                          &   *{\bullet} \ar@{-}[d]_{{5\uparrow}} \ar@{-}[dr]^{\!\!\!\!\!\!\underset{\footnotesize{\searrow}}{1}\ \ \ \ \ }
\ar@{-}@(ur,ur)[dr]^{\!\!\!\!\!\!\underset{\footnotesize{\searrow}}{1}\ \ \ \ \ }                           &    &
                   &   &   *{\bullet} \ar@{-}[d]_{{1}} \ar@{-}[dr]^{1}
\ar@{-}@(ur,ur)[dr]^{1}                           &    &
                   &   &   *{\bullet} \ar@{-}[d]_{{5\uparrow}} \ar@{-}[dr]^{\!\!\!\!\!\!\underset{\footnotesize{\searrow}}{1}\ \ \ \ \ }
\ar@{-}@(ur,ur)[dr]^{\!\!\!\!\!\!\underset{\footnotesize{\searrow}}{1}\ \ \ \ \ }                           &
                   \\
& &    &   *{\bullet}  \ar@{-}@(l,l)[u]^{\ {3\downarrow}}
\ar@{-}[r]_{\underset{\leftarrow}{2}}                 & *{\bullet} &
 &    &   *{\bullet}  \ar@{-}@(l,l)[u]^{{1}}
\ar@{-}[r]_{{2}}                 & *{\bullet}  &
 &    &   *{\bullet}  \ar@{-}@(l,l)[u]^{\ {3\downarrow}}
\ar@{-}[r]_{\underset{\leftarrow}{2}}                 & *{\bullet} \\
}
%
 \caption{The multiplicity cochain $M$, 
 the $\ZZ/8$-valued symmetric function 
 $e\mapsto \wt {\sta a}(e)={\sta a}(e)\gcd(8,M(e))$ and 
 the  cochain
$e\mapsto \wt b(e)=b(e)\gcd(8,M(e))$.} \label{fig:aut8}
\end{figure}
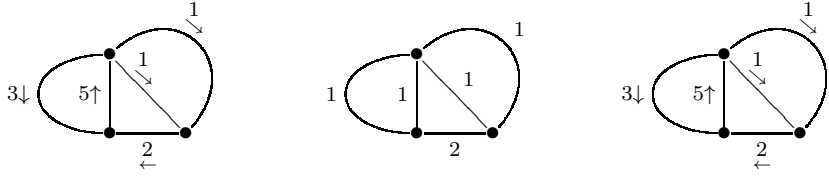
\end{example}

\begin{example}\label{exa:aut12}
Let $\ell=12$. In view of the proof of Theorem \ref{thm:no_junior_ghosts} we slightly generalise Example \ref{exa:aut8}.
We refer to Figure \ref{fig:aut12}, where we adopt the established conventions.
Using \eqref{eq:agel} we get
$\age(\sta a)= 1/12+1/12+1/12+1/12+2/12+2/12=2/3.$ 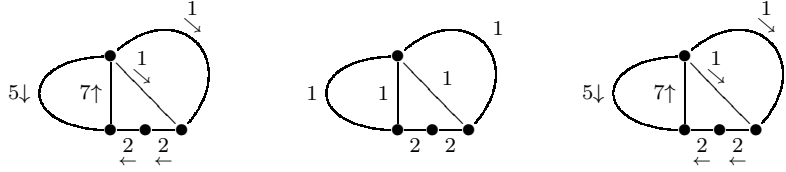
\begin{figure}[h]
\xymatrix@=.6pc{
   &&                    &&                                                          &&
*{\bullet} \ar@{-}[dd]_{{7\uparrow}} \ar@{-}[ddrr]^{\!\!\!\!\!\!\underset{\footnotesize{\searrow}}{1}\ \ \ \ \ }
\ar@{-}@(ur,ur)[ddrr]^{\!\!\!\!\!\!\underset{\footnotesize{\searrow}}{1}\ \ \ \ \ }                           &&    &&
                   &&   &&   *{\bullet} \ar@{-}[dd]_{{1}} \ar@{-}[ddrr]^{1}
\ar@{-}@(ur,ur)[ddrr]^{1}                           &&    &&
                   &&   &&   *{\bullet} \ar@{-}[dd]_{{7\uparrow}} \ar@{-}[ddrr]^{\!\!\!\!\!\!\underset{\footnotesize{\searrow}}{1}\ \ \ \ \ }
\ar@{-}@(ur,ur)[ddrr]^{\!\!\!\!\!\!\underset{\footnotesize{\searrow}}{1}\ \ \ \ \ }                           &&
                   \\
\\
&& &&    &&   *{\bullet}  \ar@{-}@(l,l)[uu]^{\ {5\downarrow}}
\ar@{-}[r]_{\underset{\leftarrow}{2}}                 & *{\bullet} \ar@{-}[r]_{\underset{\leftarrow}{2}}      & *{\bullet} &&
&&&&   *{\bullet}  \ar@{-}@(l,l)[uu]^{{1}}
\ar@{-}[r]_{{2}}                 & *{\bullet} \ar@{-}[r]_{{2}}      & *{\bullet} &&
&&&&   *{\bullet}  \ar@{-}@(l,l)[uu]^{\ {5\downarrow}}
\ar@{-}[r]_{\underset{\leftarrow}{2}}                 & *{\bullet} \ar@{-}[r]_{\underset{\leftarrow}{2}}      & *{\bullet} &&
%
}
%
 \caption{The multiplicity cochain $M$, the $\ZZ/12$-valued symmetric function $e\mapsto \wt {\sta a}(e)={\sta a}(e) \gcd(12,M(e))$ and the  cochain
$e\mapsto \wt b(e)=b(e)\gcd(12,M(e))$.} \label{fig:aut12}
\end{figure}
\end{example}

\subsubsection{The locus of noncanonical singularities of $\ol {\mathcal R}_{g,\ell}$.}
We may apply the above criterion as follows.
Within Deligne and Mumford's moduli stack $\ol {\sta M}_g$ of stable curves,
consider the locus $$\ol {\sta M}_{g}^{\circ}=\{ C\mid \aut(C)=0\}_{/\cong}$$
of stable curves with trivial automorphism group.
This is a stack that can be represented by a smooth scheme.
We study the overlying stack
$$\ol {\sta R}_{g,\ell}^{\circ}=\{\level\mid \aut(C)=0\}_{/\cong}$$
of  level curves $\level$ such that
the coarsening $C$ of $\sta C$ has trivial automorphism group $\aut(C)$.
The scheme $\ol {\mathcal R}_{g,\ell}^{\circ}$ coarsely representing $\ol {\sta R}_{g,\ell}^{\circ}$
may well have singularities;  this happens as soon as
$\ul\aut\level=\ul\aut_C\level$ is nontrivial
(by Lemma \ref{lem:noqrefl_ghosts},
nontrivial ghosts cannot operate as the identity or as quasireflections; hence
singular points in $\ol {\mathcal R}_{g,\ell}^{\circ}$ are characterised by
the presence of nontrivial ghosts).
Furthermore, since the action of $\aut\level$ on $\defo\level$ satisfies the hypotheses of
the \rsbt{} criterion (Theorem \ref{thm:rsbt}),
noncanonical singular points are characterised 
by the
presence of {junior nontrivial ghosts} in the sense of Definition \ref{defn:junior}.
Examples \ref{exa:aut5}-\ref{exa:aut12} already allow a few 
remarks on the codimension of the locus of noncanonical singularities.
By Example \ref{exa:aut5bis}, within
$\ol {\mathcal R}^{\circ}_{g,5h}$,
the locus of noncanonical singularities has codimension $3$. 
Furthermore, the theorem that follows may be regarded as saying: 
for level $2,3,4,$ and $6$, all singularities of
the scheme
$\ol {\mathcal R}_{g,\ell}^{\circ}$ are canonical.
We spell out the statements in the following remark and theorem in 
terms of the entire space $\ol{\mathcal R}_{g,\ell}$. 

\begin{remark}\label{rem:codim}
In Examples \ref{exa:aut5}-\ref{exa:aut12} the edges are all nonseparating and 
are more than $2$. 
These are general features: 
the edges are nonseparating because $r(e)$ vanishes 
on separating edges. 
In other words $J_\ell$ lies over the divisor $\delta_0^{\stable}$ 
of curves having at least 
one nonseparating node. 
Furthermore, a graph $\Gamma$ with 
only two 
separating edges can only 
be a graph whose nonseparating edges are two loops, 
or a graph with a single circuit of length-$2$.
In any case, two circuits never overlap in $\Gamma$.
As a consequence ghosts are always senior if the graph has only 
two nonseparating edges.
We conclude that the codimension of $J_\ell$ is higher than $2$. 
This means that within the locus of non canonical singularities there is only one 
irreducible 
component which has codimension $2$ in $\ol{\mathcal R}_{g,\ell}$: the locus $T_\ell$.
%
\end{remark}

\begin{theorem}[No-Junior-Ghosts Theorem]\label{thm:no_junior_ghosts}
For  $g\ge 4$ and $\ell\ge 1$ consider the stack of level-$\ell$ genus-$g$ curves $\ol {\sta R}_{g,\ell}$.
In $\ol {\sta R}_{g,\ell}$, every nontrivial ghost automorphism
  is senior if and only if $\ell\le 6$ and $\ell\neq 5$.  
\end{theorem}
\begin{proof}
Proving
the ``only if''-part of the statement for a given level $\ell$ and genus $g$
amounts to exhibiting a dual graph $\Gamma$ attached to an object of
$\ol {\sta R}_{g,\ell}$ with a multiplicity $M\in \ker\partial$ and
a symmetric function $\sta a\colon E\to \ZZ/\ell$ defining a junior ghost.
Notice that if there exists such a triple $(\Gamma , M, \sta a)$
for $\ol {\sta R}_{g,\ell}$, then we can exhibit a triple
$(\Gamma, (\ell'/\ell)\times M, (\ell'/\ell)\times \sta a)$ for $\ol {\sta R}_{g,\ell'}$
for any multiple $\ell'$ of $\ell$
(here, Proposition \ref{pro:Misclosed} has been implicitly used).
Examples \ref{exa:aut5}, \ref{exa:aut8}, \ref{exa:aut12} actually occur for $g\ge 4$ and
exhibit junior ghosts
for positive levels $\ell\in5 \ZZ\cup 8\ZZ \cup 12\ZZ$.
By halving a single straight edge in Figure \ref{fig:aut5}
($\xymatrix@=0.7pc{ *{\bullet} \ar@{-}[r] & *{\bullet}\ar@{-}[l] } \to \xymatrix@=0.7pc{ *{\bullet} \ar@{-}[r] &*{\bullet} \ar@{-}[r]\ar@{-}[l]  &*{\bullet}\ar@{-}[l]  }$),
we can immediately generalise
Example \ref{exa:aut5} from $\ell=5$ to $\ell=7$; iterating this procedure,
for all odd levels $\ell \ge 5$ and for their multiples, we exhibit   junior ghosts.
The ``only if'' part is proven: in order for junior ghost not
to occur, $\ell$ should be a positive integer
of the form $2^a 3^b$ with $a\in \mathbb N$ and $b=0,1$ (\emph{i.e.} not a multiple of an odd integer $\ge 5$),
with $a<3$ (\emph{i.e} not a multiple of $8$)
and with $a<2$ if $b=1$ (\emph{i.e} not a multiple of $12$).

The ``if''-part of the statement claims that there is no junior
ghost $\sta a$ in $G_{\pmb \nu}(M;\ZZ/\ell)$ for any stable graph $\Gamma$ with $\pmb \nu=\pmb \nu_M$ attached
$M\in \ker \partial$. 
Throughout the entire proof, we will use the following 
necessary conditions numbered (i), (ii) and (iii) for the existence of 
a junior ghost $\sta a $: 
\begin{enumerate}
\item $\age(\sta a)<1$ (\emph{i.e.} $\sta a$ is junior); \label{item:ajunior}
\item $M=\sum_{i\in I} K_i$, where $I$ is a finite set of circuits 
(\emph{i.e.} $M\in \ker\partial$);\label{item:MinKer}
\item $\sta{a}\odot M(K)\equiv 0$ for any circuit $K$ (\emph{i.e.} $\sta{a}\odot M\in \im \delta$).\label{item:aMinIm}
                          \end{enumerate}
In order to conclude that such conditions are incompatible 
we provide tables showing
all possible values of $M$ and $\sta a\in \ZZ/\ell$
and the corresponding value of $\sta a \odot M$ for  $\ell=2,3,4,6$.
In the first line we list all values $M=0,1,\dots,\ell-1$.
In the first column, we list the possible 
values $i=0,1,\dots, \ell-1$ 
that $\sta a$ may take at an edge $e$
of multiplicity $M$. We have $\sta a(e)=i$ 
only if $i$ satisfies the compatibility condition 
$\gcd(M,\ell)\mid i$.
Then, we fill the $i$th slot
of the $j$th column  
in the table with the corresponding value of 
$\sta a\odot M$ 
if and only if $\sta a=i$ is compatible with $M=j$.
\begin{figure}[h]
\begin{eqnarray*}
\begin{tabular}{c||c|c|}
   $\ell=2$ 
   & $0$ &$ 1$    \\
  \hline
$ 0$ &
0 & 0 \\
1&& 1
\end{tabular}
\quad \begin{tabular}{c||c|c|c|}
   $\ell=3$& $0$ &$ 1$ & 2    \\
  \hline
   $0$ &0
 & 0 & 0\\
1&& 1 & \framebox[1.3\width]{2}\\
2&& 2 & 1
\end{tabular}
\quad 
\begin{tabular}{c||c|c|c|c|}
   $\ell=4$& $0$ &$ 1$ & 2& 3   \\
  \hline
0 &
0 & 0& 0& 0 \\
1&& 1& &\framebox[1.3\width]{3}\\
2&& 2& 2 & 2\\
3&& 3& & 1
\end{tabular} \quad
\begin{tabular}{c||c|c|c|c|c|c|}
   $\ell=6$& $0$ &$ 1$ &2&3&4&5   \\
  \hline
 0&
0 & 0 & 0 & 0 & 0& 0\\
1&& 1 &       &   & & 
\framebox[1.3\width]{5}\\
2&& 2 & 2 & & \framebox[1.3\width]{4} & \framebox[1.3\width]{4}\\
3&& 3 &   &   3 & & 3\\ 
4&& 4 & 4 & & 2 & 2\\
5&& 5 &    &   && 1
\end{tabular}
\end{eqnarray*}
\caption{Multiplication tables for $\odot$ and $\ell=2,3,4$ and $6$.}
\end{figure}
We draw a box around the configurations 
where $\sta a=i$ is strictly less than 
 $\sta a\let \sta a\odot M\in\{0,\dots, \ell-1\}$. 
Indeed, the presence of 
oriented edges $e$ with the corresponding values 
$\sta a(e), M(e)$ is a necessary condition for $\sta a$ to be junior. 
If no such oriented edges occur, as it happens for $\ell=2$, then, 
for a nontrivial element $\sta a$,
condition (iii) is incompatible with condition (i). Indeed, 
we have 
$$\age(\sta a)= \sum_E \frac{\sta a (e)}\ell \ge 
\sum_K \frac {(\sta a \odot M)(e)}{\ell}\ge 1 \ \ (\text{by (iii)}),$$
where $K$ is a circuit passing through an edge where $\sta a(e)$ 
is nontrivial. This settles the case $\ell=2$ 
and motivates the following definitions. 
\begin{definition}\label{defn:agedelay}
An edge $e$ 
is \emph{active} with respect to an automorphism $\sta a$ if $\sta a(e)$
is nontrivial.
For an oriented edge $e$, 
we refer to the values $(M(e), \sta a(e))$
as the \emph{type} of $e$.
An oriented edge $e$ 
is an \emph{age-delay} edge if 
after reduction modulo $\ell$ within $\{0,\dots, \ell-1\}$ we have
\begin{equation}\label{eq:agedelay}
a(e) < {\sta{a}\odot M(e)}.\end{equation} 
An age-delay edge is automatically 
active, otherwise both sides of the inequality 
vanish.
We say that a circuit $K$ is active if it passes through an active 
edge and we say that it is age-delay if it passes through an age-delay
edge. An age-delay circuit is automatically active.
\end{definition}
With this terminology, the previous argument may be rephrased.
\begin{lemma}
Let $\sta a$ be a junior automorphism. An active circuit is 
necessarily age-delay.\qed \end{lemma}
We can also prove that the type of the 
active edges of an active circuit 
cannot be constant.
\begin{lemma}\label{lem:typenotconstant}
Let $\sta a$ be a junior automorphism. 
Consider an active circuit $K=\sum_{i=0}^{n-1} e_i$ 
where the head of $e_i$ is the tail of $e_{i+1}$ for all $i\in \ZZ/n$. 
Then the active edges $e_i$ of $K$ cannot be all of 
the same type.
\end{lemma}
\begin{proof}
By way of contradiction, assume 
there is an active circuit $K$ whose active edges are all of type 
$(M(e),\sta a(e))=(J,I)$ for some values $J,I\in \{1,\dots, \ell-1\}$ 
with $\gcd(J,\ell)\mid I$. 
Then, condition (iii) may be expressed via \eqref{eq:odotprod} as
$\sum_K IJ/\gcd(J,\ell)^2\in \ell/\gcd(J,\ell)\ZZ$. In particular, 
$\ell/\gcd(J,\ell)$ divides $k^\# IJ/\gcd(J,\ell)^2$ where 
$k^\#$ is the number of active edges in $K$. We conclude 
that $\ell/\gcd(J,\ell)$, which is prime to 
$J/\gcd(J,\ell)$, divides $k^\# I/\gcd(J,\ell)$
and $\ell$ divides $k^\# I>0$.
This contradicts $\age(\sta a)<1$, because 
$\age(\sta a)\ge k^\#I/\ell>1$.
\end{proof} 

In the case $\ell=3$ (resp. $\ell=4$) 
the only 
age-delay edges are of type $(\ell-1,1)$.
A nontrivial junior automorphism should contain 
an age-delay circuit $K=\sum_{i=0}^{n-1}e_i$ 
with $(M(e_0),\sta a(e_0))=
(\ell-1,1)$. For $i\neq 0$ 
the total value of $\sta a$ should be strictly 
less then $2$ (resp. $3$) by condition (i).
Furthermore the total value of $\sta a\odot M$
reduced within $\{0,\dots,\ell-1\}$ modulo $\ell$   
should be $1$ (resp. $1$) by condition (iii). 
Then, only one of the edges $e_i$ for $i\neq 0$ 
is active, and its type is $(1,1)$ (resp. 
is (1,1)). 
For a junior automorphism, 
any active circuit contains exactly two active edges of 
type $(1,1)$ and $(\ell-1,1)$. Then, the automorphism cannot be junior;
the claim follows from this slightly more general statement
(which we make in view of $\ell=6$).
\begin{lemma}\label{lem:inout}
Let $\sta a$ be an automorphism for which all 
active circuits 
have only an even number $2k$ 
of active edges equally divided into $k$ edges of type
$(1,1)$ and $k$ edges of type $(\ell-1,1)$. 
Then $\sta a$ is either trivial or senior.
\end{lemma}
\begin{proof}
We consider the set of active edges, which by the 
hypothesis, can only be of multiplicity $M=1$ or $\ell-1$ 
depending on their orientation. We pick an orientation 
for all edges in the edge set $E$ 
is such a way that $M=1$ on all active edges.
Notice that a circuit, which is by definition a sequence of  
oriented edges $e_0,\dots, e_{n-1}\in \mathbb E$,
can now be regarded, with respect to the chosen orientation, as a characteristic function $\chi_K\colon E\to \ZZ/\ell$, 
which equals $1$
(resp. $-1\in\ZZ/\ell$) 
on $e$ if $e=e_i$ (resp $e=\ol e_i$ for some $i$)
and vanishes elsewhere. 
Condition (ii) may be regarded as saying that 
the multiplicity is a $\ZZ/\ell$-valued 
sum of these characteristic functions of 
circuits. If we add up the values of the 
multiplicities $M$ of the active circuits we obtain $0\in \ZZ/\ell$
because each function $\chi_K$ 
restricted on the active circuits has total value
$k-k=0$ by the hypothesis of the lemma. 
Since $M=1$ on all active edges, the number of active edges is a 
multiple of $\ell$. Then $\sta a$ is senior or trivial, since it
equals $1$ on all active edges.
\end{proof}
For $\ell=6$ the value of 
$\sta a$ on an age-delay edge is $1$ or $2$. 
By excluding the second case the claim will be deduced below
as for $\ell=3$ and $4$.
\begin{lemma}
Let $\ell=6$ and $\sta a$ be junior. Then $\sta a(e)\neq 2$
on all edges.
\end{lemma}
\begin{proof}
By way of contradiction let $e$ be an oriented 
edge with $\sta a(e)=2$ and consider a circuit 
$K=\sum_{i=0}^{n-1}e_i$ through it with $e=e_0$. 
The value of $M$ can be $1,2,4,5$ because $\gcd(M(e_0),\ell)$ should 
divide $2$. By conveniently choosing the orientation of $e_0$ 
and of the circuit $K=\sum_{i=0}^{n-1}e_i$ we assume
$M(e_0)=4$ or $5$, which implies 
$(\sta a\odot M)(e_{0})=4$ ($e_0$ is age-delay). 

The age contribution of $e_{0}$ is 
$a(e_0)/\ell=1/3$.
Furthermore, the function $\sta a\odot M$ should 
add up to $2\in \ZZ/6$ on the remaining
active edges of $K$ (condition (iii)). 
Condition (i), $\age(\sta a)<1$, imposes edges 
with $\sta a<4$.
We have two possibilities for 
the set of 
active edges of $K$: \begin{enumerate}
\item[(a)] it is formed by $e_0$ and two active edges
$e'',e'''$ of type $(1,1)$ where $\sta a\odot M$ equals $1$; 
\item[(b)] it is of the form $\{e_0, e''\}$ with $e''$ 
of type $(1,2)$ or $(2,2)$, and  
$\sta a\odot M(e'')=2$. \end{enumerate}
In any of these cases $K$
contributes $2/3$ to $\age(\sta a)$ and there is exactly another
 active edge $e'$ of $\Gamma$ outside
$K$ (by Lemma 
\ref{lem:typenotconstant} and $\age(\sta a)<1$):
we have $\sta a(e')=1$ and $M(e')$ odd.
We argue by parity, that is we compose   
the functions $M$ and $\sta a\odot M$  
with $P\colon \ZZ/6\to \ZZ/2$. 
Note that $P(M)$ does not depend on 
the choice of the orientation. 
The value of $P(\sta a \odot M)$ on the only edge of $H$ that 
lie off $K$ is odd; therefore, 
$P(\sta a \odot M)$ should add up to $1\in \ZZ/2$ also on the set 
of edges shared by $K$ and $H$. 
Then we exclude case (b), where $P(\sta a \odot M)$ 
is zero identically.

The 
set of active edges of $\Gamma$ is 
formed by $e_0,e',e'', e'''$.
The function $P(\sta a \odot M)$ is even on $e_{0}$ and odd on 
$e', e''$ and $e'''$; hence, any active circuit should go through 
$\{e',e'', e''\}$ an even number of times; by (ii), 
this implies $P(M(e'))+P(M(e''))+P(M(e'''))
=0$.
This is impossible, because 
$P(M)$ equals $1$ identically on $\{e', e'',e'''\}$.
\end{proof}

For $\ell=6$, any age-delay edge is of type $(5,1)$; therefore, in order to 
be compatible with (i) and (iii) and the above lemma, an 
age-delay circuit has either exactly two
 active edges of type (1,1) and (5,1)
or four active edges equally divided into two edges of type (1,1) and 
two edges of type (5,1). Lemma \ref{lem:inout} implies the claim.
\end{proof}

\begin{definition}\label{defn:Jcurve}
A level-$\ell$ curve $\level$ is a J-curve if $\aut\level$ contains a junior ghost.
\end{definition}
The points representing J-curves are noncanonical singularities by definition.
Noncanonical singularities may occur even if the
level curves has no junior ghost automorphisms and regardless of the level $\ell$.
Indeed this is the case of level curves of type T (or simply  {T-curves}) which we now illustrate.
T-curves represent a codimension-$1$ locus within
the divisor $\Delta_{g-1}$; \emph{i.e.} a codimension-$2$ locus in $\overline {\sta R}_{g,\ell}$.

\begin{definition}\label{defn:Tcurve}
A  level-$\ell$ curve $\level$ is a T-curve if \begin{itemize}
\item $\sta C$ contains an elliptic tail
(that is $\subset E\subset \ssC$ with
${\ssC\cap \ol {\ssC\setminus E}}=\{n\}$) (Tail-condition);
\item $E$ admits an order-$3$ automorphism (that is $\aut(E,n)\cong \pmmu_6$) (Three-condition);
\item $\ssL$ is trivial on the elliptic tail; \emph{i.e.}
$\level \in \Delta_{g-1}$ (Triviality-condition). \end{itemize}
\end{definition}

\begin{theorem}\label{thm:noncanonical}
The point representing $\level$ in $\ol {\sta R}_{g,\ell}$ is
a noncanonical singularity if and only if  $\level$ is a T-curve or a J-curve.
\end{theorem}
\begin{proof}
For the ``if''-part of the statement
we only need to show that a T-curve $\level$ has a junior automorphism.
Let us define the automorphism $\sta a_{1/6}\in \aut\level$,
whose restriction to $\ol{\ssC\setminus E}$ is the identity and
whose restriction to $E$ generates $\aut(E,n)\cong \pmmu_6$
and operates on the local parameter of $E$ at $n$ as $z\mapsto \xi_6z$.
The coordinates $\tau_1$ and $\tau_2$ correspond to the
direction smoothing the node $n$ and to the direction
preserving the node and varying along $\Delta_{g-1}$.
The action of
$\sta a_{1/6}$ on $\defo\level$ is given by
$\diag(\xi_6,\xi_6^2,1,\dots,1)$, where the first coordinates are $\tau_1$ and $\tau_2$.
The action  of
$\sta a_{1/6}$ on $\defo\level/\qr$ is given by
$\diag(\xi^2_6,\xi_6^2,1,\dots,1)$, where
the first coordinates are $\ol\tau_1=\tau_1^2$ and $\ol\tau_2=\tau_2$ (see \S\ref{subsubsect:modqr}).
The age of $\sta a_{1/6}$ on $\defo/\qr$ is $1/3+1/3=2/3<1$,
and $\defo/\aut$
has a noncanonical singularity.
Notice that all the junior automorphisms 
operating as $z\mapsto \lambda z$ on the tail (and fixing the rest of the curve) 
are of the form $\sta a_{1/6}$ up to ETI.

The ``only if''-part reduces to the following proposition.
\begin{proposition}\label{pro:onlyifinanutshell}
Let $\level$ be a  level-$\ell$ curve which is not a J-curve
and has a junior automorphism $\sta a$,
then it is a T-curve and the isomorphism $\sta a$ coincides, up to ETIs, with the above
isomorphism $\sta {a}_{1/6}$.
\end{proposition}

Preliminary 1. As in \cite[p.33]{HM} we begin by slightly simplifying the problem by adding
 a further condition to
the hypotheses.
A level curve $\level$ representing a noncanonical singularity in $\rr_{g,\ell}$
is \emph{$(\star)$-smoothable} if the following conditions are satisfied.
\begin{enumerate}
\item[{(a)}] There is a junior automorphism $\sta a\in \aut\level$
and $m$ nodes
$\sta n_0, \dots, \sta n_{m-1}$ lying in $\Sing(\ssC)\setminus \{\text{QR nodes}\}$
labeled by $j\in \ZZ/m$ so that $\sta a(\sta n_j)=\sta n_{j+1}$.
\item[{(b)}] We have $\textstyle{\prod_{i=0}^{m-1}}c_j=1$,
where $c_j$ are the complex nonvanishing constants
satisfying
$\sta a^*\tau_{j+1}=c_j \tau_j$ for all $j\in \ZZ/m$,
and $\tau_j$
is the parameter smoothing $\sta n_j$.
\end{enumerate}
By  \cite[p.33]{HM},
if $\level$ is $(\star)$-smoothable, then
the data $\sta a \in \aut\level$
can be deformed to $\sta a'\in \aut(\leveli',\levelii', \leveliii')$
in such a way that the $m$ nodes above are smoothed
and the age of the action on $\defo/\qr$ is preserved.
In
\cite[Prop.3.6]{lu2007}, Ludwig proves a generalisation applying to
moduli of roots of any line bundle; in particular we can use this fact for  level-$\ell$ curves.
Hence, by iterating such deformations, within the locus of noncanonical singular points in $\rr_{g,\ell}$,
we can smooth any $(\star)$-smoothable curve to a curve which is no more $(\star)$-smoothable; we refer to
this condition as \emph{$(\star)$-rigidity}.
The loci of T-curves and of J-curves are closed:
in the above deformation,
if $(\leveli',\levelii', \leveliii')$ is a J-curve (a T-curve), then
$\level$ is a J-curve (a T-curve). 
Proposition \ref{pro:onlyifinanutshell}
can be shown under the following assumption.
\begin{assumption}\label{assum:nonsmoothability}
In the proof of Proposition \ref{pro:onlyifinanutshell} we assume that
$\level$ is $(\star)$-rigid.
\end{assumption}

Preliminary 2. Set $\mathfrak{ord}(\sta a)=\ord(a)$; this is also
the least integer
for which $\sta a^m$ is a ghost and is a divisor of $\ord(\sta a)$. We can provide
lower bounds for the age of $\sta a$.
\begin{lemma}\label{lem:age1/m}
Consider a  level-$\ell$ curve $\level$.
 \begin{itemize}\item[0.]
  For any automorphism $\sta a\in \aut\level$,
 we have 
  $\age(\sta a)\ge  {(\#E-N)}/{2},$
 where $N$ is the number of cycles of the permutation of $E$ induced by $\sta a$.
 \end{itemize}
 We
can  improve the lower bound in the following situations.
\begin{itemize}
\item[1.] Assume that $\level$ is a
noncanonical $(\star)$-rigid singularity in $\rr_{g,\ell}$; then, for any subcurve
$\sta Z$  such that $\sta a(\sta Z)=\sta Z$ and for any length-$k$ cycle of the induced permutation of
$\Sing_{\mathrm {nonQR}}(\ssC)\cap \Sing(\sta Z)$,
we have
$\age(\sta a)\ge {k}/{\ord(\sta a\!\mid_{\sta Z})}+ ({\#E-N})/{2}.$
\item[2.] If $\sta a^{\mathfrak{ord}(\sta a)}$ is a senior ghost, then we have
$\age(\sta a)\ge 1/{\mathfrak{ord}(\sta a)}+ ({\#E-N})/{2}$.
\end{itemize}
\end{lemma}
\begin{proof}
We can express the action of $\sta a$ on $\bigoplus_{e\in E}\mathbb A^1_{\ol \tau_e}$
(see \S\ref{subsubsect:modqr} for the notation $\ol \tau_e$)
in terms of a block-diagonal matrix
$H$ whose blocks $H_1,\dots, H_N$ are of the form
$$H_i=D_iP_i=\diag((\xi_R)^{q^{(i)}_0} ,\dots,
(\xi_R)^{q^{(i)}_{n_i-1}})P_i,$$
where $P_i$ is the permutation matrix attached to the
cycle permutation operating on $\ZZ/n_i$ as $\sigma(j)=(j+1)$,
$R$ is a suitable positive integer and the exponents 
$q^{(i)}_j$ are contained in $[0,R-1]$.
Note that $n_i$ divides the order of the permutation of $E$ induced by
$\sta a$.
Then, (see also \cite[Prop.~3.7]{lu2007}), since the characteristic polynomial of
$H_i$ is $x^{n_i}- \det D_i$, we have 
$$\age(\sta a)\ge \sum_{i=1}^N \left ( \left \{ \sum_{j=0}^{n_i-1} \frac{q_j^{(i)}}{R}\right \} +\frac{n_i-1}{2}\right)=
\sum_{i=1}^N   \left \{ \sum_{j=0}^{n_i-1} \frac{q_j^{(i)}}{R}\right\} 
+\frac {\#E-N}{2},$$
where the right hand side is of the form $A+\frac {\#E-N}{2}$ with $A\ge 0$.

In claim (1), there is a 
$k\times k$-block $H=H_{i_0}$ with
$D=D_{i_0}=
\diag((\xi_R)^{q_0} ,\dots,
(\xi_R)^{q_{k-1}})$,
 $H^k=(\xi_R)^{q}\mathbb I$, and
$\frac qR= \{ \sum_{j=0}^{k-1} \frac{q_j}R\}\neq 0$ (see condition (b) defining
$(\star)$-smoothability).
Since, for $w=\ord(\sta a\!\mid_{\sta Z})$ we have  $H^w=\id$,
we have $\frac{w}{k}\frac{q}R\in \ZZ$; hence, we have $A\ge q/R\ge k/w$ as  required.

In case (2) we are assuming that $\sta a^m$ is senior for $m=\mathfrak{ord}(\sta a)$. Notice that 
$m/n_i$ is integer for all $i$. 
We want to show $A\ge 1/m$. Assume $A< 1/m$; 
then, for all $i$, we preliminarily 
notice 
\begin{equation}\label{eq:after<1/m}\left \{
\frac{m}{n_i}\sum_{j=0}^{n_i-1} \frac{q_{j}^{(i)}}R \right \}=\frac{m}{n_i}\left
\{ \sum_{j=0}^{n_i-1} \frac{q_j^{(i)}}R\right \}.\end{equation}
This happens because we have 
$$\frac{m}{n_i}\left \{ \sum_{j=0}^{n_i-1} \frac{q_j^{(i)}}R\right\}\le
\sum_{i=1}^N m\left \{ \sum_{j=0}^{n_i-1} \frac{q_j^{(i)}}R\right\}=mA<1.$$
On the other hand $\sta a^m$ is a ghost automorphism and operates on
$\bigoplus_{e\in E}\mathbb A^1_{\ol\tau_e}$ as the diagonal matrix $H^m$
with $n_i$ eigenvalues equal to $(\det D_i)^{m/n_i}$ for $i=1,\dots,N$; using \eqref{eq:after<1/m}, we get
$$\age(\sta a^m)=\age(H^m)=\sum_{i=1}^N n_i \left \{
\sum_{j=0}^{n_i-1} \frac{m}{n_i} \frac{q_{j}^{(i)}}R \right \} {=}\sum_{i=1}^N m\left \{ \sum_{j=0}^{n_i-1} 
\frac{q_j^{(i)}}R\right \}<1.$$
contradicting the assumption that $\sta a^m$ is senior.
\end{proof}

Step 1: the automorphism $\sta a$ fixes all nodes except, possibly,  from a single transposition of two nodes.
Indeed, by \cite[p.34]{HM} (embodied in the first part of Lemma \ref{lem:age1/m}),
each node transposition contributes $1/2$ to $\age(\sta a)$.

Step 2: for each irreducible component $\sta Z$ we have $\sta a(\sta Z)=\sta Z$.
Harris and Mumford's argument excludes\footnote{The space 
parametrizing the deformations of a 
hypothetical component $\sta Z$ for which
$\sta a(\sta Z)\neq \sta Z$ (alongside with its special points) 
should have dimension $d=0$ or $1$ 
and in this second case we must have $\sta a(\sta a(\sta Z))=\sta Z$,
\emph{i.e.} the cycle of irreducible components obtained by 
applying the automorphism $\sta a$ iteratively starting from $\sta Z$
must have length $m=2$ (indeed, via
an age estimate analogue to  
Lemma \ref{lem:age1/m}, one can prove $\age\ge d(m-1)/2$). 
The case $d=1$ 
corresponding to (c), (d) and (e) in
\cite[p.~35]{HM} is ruled out by the authors as well as 
the case named (b) where
$\sta Z$ is a singular elliptic tail, because it 
yields $g(\ssC)\le 3$.} 
the condition $\sta a(\sta Z)\neq \sta Z$ 
apart from one situation 
which we now state precisely.
  \begin{figure}[h]
  \begin{picture}(200,25)(-35,15)
   \qbezier[30](30,20),(40,27),(55,50)   
     \qbezier[30](71,40),(70,12),(95,10)
      \qbezier[30](108,38),(103,12),(80,10)
   \put(72,28){\circle*{3}}
    \put(52,18){\small{$\sta Z$}}
   \put(41,31){\circle*{3}}
 \put(103,25){\circle*{3}}
 \qbezier(33,32),(73,28),(113,24)
 \end{picture}
 \end{figure}
Case (a) of p.35 in \cite{HM}
concerns a smooth, rational,
irreducible component $\sta Z$  meeting the rest of the curve at
three special points; in the present setup
we should of course allow nontrivial stabilisers at these three points.
Then \cite{HM} relies on the following claim
in the special case
 $\sta a=a\in \aut(C)$. We state a generalised version, which is due to Ludwig, see
\cite[{end of Proof of Prop.~3.8}]{lu2007}.
\begin{lemma}\label{lem:LudwigXlemma}
Assume the coarsening of $\sta a\in \aut\level$ operates locally at a scheme-theoretic
 node $\Spec\CC[x,y]/(xy)$ of $C$ as
$(x,y)\mapsto (y,x)$.
Then  $\sta a$ fixes the parameter smoothing the node in $\defo(C)$ and operates
on the parameter $\tau$ smoothing the node in $\defo\level$ as
either $\tau\mapsto \tau$ or $\tau\mapsto -\tau$.
In the first case the curve is not ($\star$)-rigid.
\qed
\end{lemma}
It is worthwhile to sketch the proof since
Ludwig uses the different setup of quasistable curves
(which is equivalent in this case).
If $\sta a=a$ we have $xy=t=\tau=\tau$, hence
$\tau\mapsto\tau$.
Otherwise
note that the multiplicity at the oriented edge $e$ 
corresponding to the above node
satisfies
$M(e)=-M(\ol e)=-M(e)$; hence
$M(e)=\ell/2$ and the action on $\tau_e$ is $\tau_e\mapsto \tau_e$
or $\tau_e\mapsto -\tau_e$. 

Now, let us assume $\sta a(\sta Z)\neq \sta Z$ 
and apply the fact that $\sta a$ is junior and that 
$\level$ is $(\star)$-rigid.
The three special points of $\sta Z$ are nodes of $\ssC$. 
If they are fixed they have two branches, one in $\sta Z$ and one 
in $\sta a(\sta Z)$. Since the coarsening $Z$ of $\sta Z$ is a projective line
these fixed nodes satisfy the condition of the above lemma and,
by ($\star$)-rigidity, yield age contribution $1/2$.
Recall that each non-fixed node also contributes $1/2$. The age is at least $1$ (with one pair of nodes exchanged
and the remaining node is fixed). 
So, the argument of \cite{HM} holds true:
$\sta a(\sta Z)\neq \sta Z$ is ruled out.

Step 3: classification of the
irreducible components. For any irreducible component
$\sta Z$ of $\sta C$ let us set up the notation for the rest of the proof.
We write $\sta N\to \sta Z$ for its normalisation,
 $\sta D\subset \sta N$ for the divisor representing
special points lifting the nodes of $\sta C$,
$\sta r$ for the restriction $\sta a\!\mid _{\sta Z}$,
and $\sta r_{\sta N}\in \aut(\sta N)$ for the lift to $\sta N$.
Coarsening yields $N\to Z$, $D\subset N$,
$r\in\aut(Z)$
and $r_N\in \aut(N)$.
Since all components are fixed, we establish a list of possible cases
by recalling the classification
\cite[Prop. p.28]{HM} of nontrivial automorphisms $r_N$
of a smooth scheme-theoretic curve $N$ paired with a divisor $D\subset N$
operating on $H^1(N,T(-D))$ with age less then $1$:
\begin{enumerate}[(I)]
\item $N$ rational with $r_N\colon z\mapsto \xi_n z$ for $n=2,4$,
\item $N$ elliptic with $r_N$ of order $2,3,4,$ or $6$,
\item $N$ hyperelliptic of genus $2$ or $3$ with $r_N$ the hyperelliptic involution,
\item $N$ of genus-$2$ with an involution $r_N$ such that $N/\langle r_N \rangle$ is an elliptic curve.
\end{enumerate}

Step 4.  Classification of the irreducible components
$\sta Z$ satisfying the following extra condition: 
$\sta a$ fixes all nodes of $\sta Z\cap \Sing (\ssC)$.
We keep the above notation 
$\sta N, \sta D,  \sta r, \sta r_{\sta N}, N, D, r, r_{N}$.

First, case (I) does not occur. Indeed, we argue 
 as in \cite[case (b), p.37]{HM}: since the nodes of $\sta Z\cap \Sing(\ssC)$ 
 are fixed, the special points of $\sta N$ are either fixed or form orbits 
 of $2$ points with
 respect to $\sta r_{\sta N}$.
We deduce that
$r_N$ necessarily operates on the coarse space as $z\mapsto \xi_2 z$.
Then, using the stability condition,
it is easy to show that there is
at least one pair of points with opposite coordinate on $N$ mapping to a node $n$ of $Z$.
By Lemma \ref{lem:LudwigXlemma} and $(\star)$-rigidity this
yields age contribution $1/2$ and the nontriviality of the stabiliser over $n$;
we 
deduce (using $\age(a)<1$) that there is exactly  one node of
$Z$ whose preimages in $N$ are interchanged by $r_N$.
The remaining nodes lying in $Z$ are contained in the images of the two
fixed points of $z\mapsto \xi_2 z$. Note that there cannot be 
two such nodes, otherwise the 
action on $H^1(N,T (-D ))$ 
gives extra contribution of at least $1/2$, 
because the order-$2$ 
automorphism does not deform to the general 
four-pointed rational curve.
Therefore, the only possibility is that
$\sta Z$ is a stack-theoretic genus-$1$ tail as in
Definition \ref{defn:elltail}.
Since $\sta a$ operates by changing
the sign of the parameter deforming the elliptic tail, we are necessarily
in the situation (iv) of Proposition \ref{pro:autelltail} and, by Remark \ref{rem:explicitautomstackytail_3},
we have $\age(\sta a)\ge 1$, a contradiction.

By a simple age computation\footnote{The dimension of the 
$(-1)$-eigenspace of $r_N$ on $H^1(N,T(P))$ 
(where $P$ is a fixed point of $r_N$)
is $2$.}
\cite[p.39, case (e)]{HM}
rules out, without changes,
the genus-$2$ curve of case (IV). 

Second, case (II) occurs only if $r_N$
fixes at least one point.
Assume, by way of contradiction, that $r_N$ is a
nontrivial translation $z\mapsto z+t_0$.
Since the
translation does not allow fixed points it
should allow two-points orbits.
In this way $r_N$ is a translation of order-$2$.
This implies that $\sta C=\sta Z$; \emph{i.e.} $\sta C$ is irreducible.
Since $g(\sta C)\ge 4$, then there are at least three nodes satisfying the conditions
of Lemma \ref{lem:LudwigXlemma}. Applying $(\star)$-rigidity
we get age contribution $3/2$ and we can conclude as in \cite{HM}
that $r_N$ fixes at least  one point; then we can use
Harris and Mumford's list of cases ``(c2)-(c5)'' at \cite[p.~37-39]{HM}
specifying the configuration of the elliptic component
and their age contribution. We summarise this in (i) and (ii), below.

We can reproduce Harris and Mumford's list of possible irreducible
components $\sta Z$ for which the restriction
$\sta r=\sta a\!\mid_{\sta Z}$ does not satisfy $r=\id_Z$ and fixes all points of $\sta Z\cap \Sing(\sta C)$.
\begin{enumerate}[(i)]
\item $\sta Z$ is a scheme-theoretic elliptic tail
$\sta r$
is the ETI (age contribution $0$) or an automorphism of
order $3,4$ or $6$ of a smooth elliptic tail $\sta Z=Z$ meeting the rest of the curve at $n$ acting on $H^1(Z, T(n))$
with age $1/3,1/2$ or $1/3$ (see figures at pages 38 and 39 in \cite{HM}).
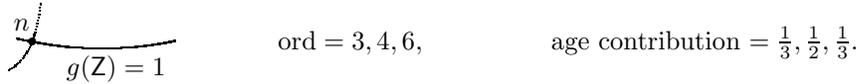
 \begin{figure}[h]
  \begin{picture}(400,10)(-35,25)
   \qbezier[30](30,20),(40,25),(42,45)   
       \put(52,18){\small{$g(\sta Z)=1$}}
       \put(32,35){\small{$n$}}
       \put(132,28){\small{$\ord=3,4,6,$}}
\put(232,28){\small{ \text{age contribution }$=\frac13,\frac12,\frac13$.}}
   \put(39,31){\circle*{3}}
  \qbezier(33,32),(63,25),(93,31)
 \end{picture}
 \end{figure}
\item $\sta Z$ is a smooth genus-$1$ component ($\sta Z=Z$) meeting the rest of the curve at two points $\sta p$ and $\sta q$;
the action on $H^1(Z,T(p+q))$ has order $2$ or $4$ and age 
$1/2$ or $3/4$  (see figures at pages 38 and 39 in \cite{HM}).
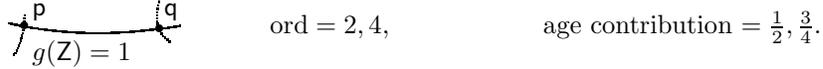
\begin{figure}[h]
  \begin{picture}(400,10)(-35,25)
   \qbezier[15](39,35),(39,26),(35,20)   
       \put(42,18){\small{$g(\sta Z)=1$}}
            \put(42,35){\small{$\sta p$}}
              \put(92,35){\small{$\sta q$}}
       \put(132,28){\small{$\ord=2,4,$}}
       \put(232,28){\small{ \text{age contribution }$=\frac12,\frac34$.}}
   \put(39,31){\circle*{3}}
     \put(90,30){\circle*{3}}
  \qbezier[15](90,40),(87,30),(94,25) 
  \qbezier(33,32),(63,25),(98,31)
 \end{picture}
 \end{figure}
\item $\sta Z$ is an hyperelliptic tail 
of genus $g=2$; the restriction $\sta r$ is the hyperelliptic involution
contributing $1/2$ to $\age(\sta a)$ (see case (d) of \cite[p.~39]{HM}).
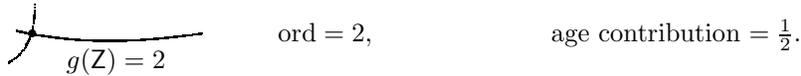
 \begin{figure}[h]
  \begin{picture}(400,10)(-35,25)
   \qbezier[30](30,20),(40,25),(40,42)   
       \put(52,18){\small{$g(\sta Z)=2$}}
       \put(132,28){\small{$\ord=2,$}}
\put(232,28){\small{ \text{age contribution }$=\frac12$.}}
   \put(39,31){\circle*{3}}
  \qbezier(33,32),(63,25),(103,31)
 \end{picture}
 \end{figure}
\end{enumerate}

Step 5. We now argue that $\sta a \in \aut(\ssC)$ fixes all nodes.
To this effect, by Step 1, we need to rule out the cases where $\sta a$ transposes a pair of nodes ($\sta n_1$,$\sta n_2$).
Since the node transposition contributes $1/2$ to $\age (\sta a)$ we can exclude
the presence in the curve of components of the form (ii) and (iii).
We can assume that $\sta a$ operates as the identity on the elliptic tails (i).
If this is not the case, we can simply modify $\sta a$ by
restricting to
$\sta B=\ol {\ssC\setminus \{\text{elliptic tails}\}}$
and by trivially extending
to $\ssC$; the resulting automorphism has lower age but it is still
nontrivial because it exchanges two nodes; hence it is a
nontrivial junior automorphism, which we will refer to it as $\sta a$ in this step.
 
 We now see that $\sta n_2=\sta a(\sta n_1)$ yields a contradiction; since 
all irreducible components are (globally) fixed by Step 2, 
we reduce to the following cases.

 \begin{figure}[hhh]
  \begin{picture}(400,60)(-35,-0)
   \qbezier(35,35),(35,32),(15,10) 
  \qbezier(35,35),(35,42),(30,42)  
   \qbezier(25,35),(25,42),(30,42)  
   \qbezier(25,35),(25,26),(55,26)     
   \qbezier(55,26),(85,26),(85,40)   
   \qbezier(85,40),(85,50),(78,50)  
   \qbezier(78,50),(71,50),(71,40) 
   \qbezier(71,40),(70,12),(95,10)
   \put(72,28){\circle*{3}}
    \put(31,29){\circle*{3}}
       \put(48,-10){\small{\text{case (a)}}}
       \put(55,12){\small{$\sta Z$}}
                  \put(32,20){\small{$\sta n_1$}}
          \put(82,20){\small{$\sta n_2=\sta a(\sta n_1)$}}
       \put(248,-10){\small{\text{case (b)}}}
       \put(262,12){\small{$\sta Z$}}
           \put(259,43){\small{$\sta N$}}
       \put(232,40){\small{$\sta n_1$}}
          \put(282,40){\small{$\sta n_2=\sta a(\sta n_1)$}}
    \put(237,31){\circle*{3}}    
     \put(288.5,31){\circle*{3}}    
  \qbezier(233,28),(263,50),(293,28) 
  \qbezier(233,34),(263,12),(293,34) 
 \end{picture}
 \end{figure}

 \begin{enumerate}[(a)]
  \item  All the branches of $\sta n_1$ and of $\sta n_2=\sta a(\sta n_1)$ lie in the  same irreducible component $\sta Z$.
  Then, Lemma \ref{lem:age1/m} (1),
  yields age contribution
  $2/n+1/2$, where $n=\ord (\sta r)$ and fits
in the conditions required by \ref{lem:age1/m}.
We observe that $\ord(\sta r)=\ord(r)$
because every ghost is the identity on the irreducible components of $\sta C$ (see Remark \ref{rem:fixingirredcompl}).
The age contribution  coincides with that used in \cite[p~36-37]{HM} in order to rule out this case.
\item  There is a component $\sta Z$ containing exactly one branch for each node $\sta n_1$ and $\sta n_2$.
Then let $\sta H$ be the second component through  $\sta n_1$ and $\sta n_2$.
Notice that $\sta a^{\mathfrak{ord}(\sta a)}$
is either the identity or a senior ghost 
because $\level$ is not a J-curve. 
Then, by 
Lemma \ref{lem:age1/m} (1-2), the age of $\sta a $ is
at least $1/n+1/2$ where $n$ is the order of the coarsening of
$\sta a\!\mid_{\sta Z\cup \sta H}$.
According to the list of cases (I-IV), $n$ can be $2,4,6,$ or $12$.
Since the lower bound $1/n+1/2$ is smaller than the lower bound $2/n+1/2$ found
in \cite[p.~36-37]{HM} we can only conclude for $n=2$.
In particular we should study more carefully the case $n=4$ where no 
extra argument was needed in \cite{HM}.
The same issue arises in \cite[Proof of Prop.3.10]{lu2007}, 
where Ludwig notices that, when $n$ equals $4$, there is  
extra age contribution of  $1/4$.
Indeed,
either $Z$ or $H$ is an elliptic curve
on which the coarsening $a$ operates, locally at a point $p\neq n_1,n_2$, as $z\mapsto \xi_4z$.
This yields extra age contribution $1/4$. 
(Ludwig also checks that the arguments of Harris and Mumford allow 
to conclude for $n=6$ and $12$ 
because of the respective 
extra age contributions $1/3$ and $1/2$ that they find 
in these two cases. The
argument fits equally well here.)
\end{enumerate}

Step 6. We are left with the problem of patching together the few curves of genus $1$ and $2$ listed in
 (i), (ii), and  (iii) with lots of identity components; \emph{i.e.} components
where the coarsening of $\sta a$ restricts to the identity.
We do it by following  \cite[p.39]{HM}, see also \cite[Propositions 3.12-15]{lu2007}.
In case there is a component of type (iii),
the second component $\sta H$ through the node separating $\sta Z$ from the rest of the curve cannot be
of type (iii) (each component of type (iii) adds $1/2$ to the $\age(\sta a)$).
By the same argument we should rule out $\sta H$ of type (ii).
On the other hand $\sta H$ cannot be an elliptic tail because $g(\ssC)\ge  4$.
Finally $\sta H$ cannot be an identity component because, this yields
a $1/2$-age contribution due to the parameter smoothing the node $\sta H\cap \sta Z$.
As a consequence, case (iii) is impossible.

Let us assume that there is a component $\sta Z$ of type (ii), that is a so-called elliptic ladder.
Since such components contribute at least
$1/2$ to $\age(\sta a)$ we assume there is exactly one such case.
We argue as in Step 5 where we have replaced $\sta a$ by another junior automorphism
operating as the identity on the elliptic tails.
In this way, we have $n=\mathfrak{ord}(\sta r)=\mathfrak{ord}(\sta a)$.
If $\sta a^{\mathfrak{ord}(\sta a)}$ is a nontrivial ghost, then
it is senior, because $\level$ is not a J-curve; by Lemma
\ref{lem:age1/m},(2) we have $\age(\sta a)\ge 1/n$.
The same inequality holds, by Lemma \ref{lem:age1/m} (1), if $\sta a^{\mathfrak{ord}(\sta a)}$ is trivial, \emph{i.e.}
if $\mathfrak{ord}(\sta a)=\ord(\sta a)$ (since $g\ge 4$
there is at least one fixed node in $\ssC\setminus \{\text{elliptic tails}\}$).
Now, for
$n=2$, the total age contribution is $1/2+1/2$
and, for $n=4$, the total age contribution if $3/4+1/4$. We may rule out this case.

Now the coarsening of $\sta a$ is the identity on all components that are not elliptic tails.
In fact $\sta a$ is actually the identity on all such components;
if this were not the case, we could replace $\sta a$ by a junior ghost automorphism of $\level$ contradicting
the assumption that $\level$ is not a J-curve.
So, $\sta a$ is the identity everywhere
except for some scheme-theoretic elliptic tails. We can now go through
the study of elliptic tails (i) and add the age contribution 
from the parameter smoothing the QR node where 
the tail meets the rest of the curve. As in \cite{HM} and \cite{lu2007} we conclude that
$\sta a$ has order $6$ and should operate on the elliptic tail
precisely as prescribed by the
statement of Proposition \ref{pro:onlyifinanutshell}.
\end{proof}

By definition, noncanonical singularities are local obstructions to
the extension of pluricanonical forms.
On the other hand Harris and Mumford show that noncanonical singularities at
T-curves do not pose a global obstruction:
pluricanonical forms extend across the locus $\sta T$ of level curves of type T
as soon as they are
globally defined off of $\sta T$. Their statement can immediately
adapted to level curves (the argument is spelled out in \cite[Thm.~6.1]{FaLu} and \cite[Thm~4.1]{lu2007}
and relies on the fact that the morphism forgetting the level structure is not ramified along
$\delta_{g-1}$). The precise statement is as follows.
\begin{corollary}\label{cor:globalpluricanonicalformsextend}
We fix $g\geq 4$ and $5\neq\ell\le 6$. Let
$\widehat{\cR}_{g,\ell}\rightarrow\rr_{g,\ell}$
be any desingularisation.
Then every pluricanonical form defined on
the smooth locus $(\rr_{g,\ell})^{\mathrm{reg}}$ of $\rr_{g,\ell}$ extends
holomorphically to $\widehat{\cR}_{g,\ell}$, that is,  for all
integers $q\geq 0$ we have isomorphisms
\[
\Gamma \bigl((\rr_{g,\ell})^{\mathrm{reg}},K_{\rr_{g,\ell}}^{\otimes q}\bigr) \cong
\Gamma\bigl(\widehat{\cR}_{g,\ell},
K_{\widehat{\cR}_{g,\ell}}^{\otimes q} \bigr).
\]\qed
\end{corollary}

\end{document}